\documentclass{amsart}

\usepackage{hyperref}
\usepackage{tikz-cd}
\usepackage{amsfonts,amsmath,amssymb,mathrsfs, stmaryrd, amsthm}
\usepackage{enumitem}
\usepackage{cite}
\usepackage{amsrefs}
\usepackage{bm}

\title{A non-unitary approach to the~$q$-deformation~of~$\mathrm{SL}(2,\mathbb{R})$}
\author{Yvann Gaudillot-Estrada}
\address{Université de Lorraine, CNRS, IECL, F-57000 Metz, France}
\email{yvann.gaudillot-estrada@univ-lorraine.fr}
\thanks{This research was supported by the projects OpART (ANR-23-CE40-0016) and CroCQG (ANR-25-CE40-5010) of the \emph{Agence Nationale de la Recherche}.  It is based upon work from COST Action CaLISTA CA21109 supported by COST (European Cooperation in Science and Technology) \url{www.cost.eu}.  Part of the research was carried out during the special trimester on Representation Theory and Noncommutative Geometry at the Institut Henri Poincaré (UAR 839 CNRS-Sorbonne Université) --- LabEx CARMIN (ANR-10-LABX-59-01).}

\newtheorem{theorem}{Theorem}[section]
\newtheorem{proposition}[theorem]{Proposition}
\newtheorem{lemma}[theorem]{Lemma}
\newtheorem{corollary}[theorem]{Corollary}
\newtheorem*{theoremA}{Theorem A}
\newtheorem*{theoremB}{Theorem B}

\theoremstyle{definition}
\newtheorem{definition}[theorem]{Definition}
\newtheorem{remark}[theorem]{Remark}

\newtheorem{convention}[theorem]{Convention}
\newtheorem{notation}[theorem]{Notation}
\newtheorem{example}[theorem]{Example}

\begin{document}

\begin{abstract}
    We study the representation theory of various convolution algebras attached to the $q$-deformation of $\mathrm{SL}(2,\mathbb{R})$ from an algebraic perspective and beyond the unitary case. We show that many aspects of the classical representation theory of real semisimple groups can be transposed to this context. In particular, we prove an analogue of the Harish-Chandra isomorphism and we introduce an analogue of parabolic induction. We use these tools to classify the non-unitary irreducible representations of $q$-deformed $\mathrm{SL}(2,\mathbb{R})$. Moreover, we explicitly show how they converge to the classical admissible dual of $\mathrm{SL}(2,\mathbb{R})$. For that purpose, we define a version of the quantized universal enveloping algebra defined over the ring of analytic functions on $\mathbb{R}_+^*$, which specializes at $q = 1$ to the enveloping $\ast$-algebra of $\mathfrak{sl}(2,\mathbb{R})$.
\end{abstract}

\maketitle
\section*{Introduction}

Studying a certain class of representations of a group $G$ often amounts to examining modules over a related convolution algebra. When $G$ is a noncompact semisimple Lie group, one is mainly interested in the unitary representations of $G$ and in Harish-Chandra's $(\mathfrak{g}, K)$-modules. The associated convolution algebras are respectively the group C*-algebra $C^*(G)$ and the Hecke algebra $R(\mathfrak{g}, K)$ \cite{KnappVoganInd}*{§I.6}.

In a recent paper \cite{DCquantisation}, De Commer provided a way to construct $q$-deformations of both algebras, which we respectively denote by $C^*_q(G)$ and $R_q(\mathfrak{g},K)$, the index $q$ being a positive real parameter different from one. The starting point of the whole approach consists in the definition of a new $q$-deformation $U_q(\mathfrak{g})$ of the universal enveloping algebra which is not of Drinfeld-Jimbo type. These constructions are mainly inspired by Letzter's theory of quantum symmetric pairs \cites{Letzter99,Kolb} and by the quantum duality principle \cites{Drinfeld, GavariniQDP}. For complex semisimple groups, they coincide with the commonly considered quantizations \cites{PWLorentz, VoigtYuncken}, namely Drinfeld doubles of compact semisimple quantum groups.

The goal of the present paper is to make a systematic study of the representation theory of these new $q$-deformations in the case $G = \mathrm{SL}(2,\mathbb{R})$. This is indeed a natural starting point for further developments. We adopt a non-unitary perspective and emphasize the similarities with the classical representation theory of $\mathrm{SL}(2,\mathbb{R})$, paying special attention to the behavior of the limit $q \to 1$ in a formal setting.

By non-unitary, we mean that we mainly study the modules over $R_q(\mathfrak{g},K)$ without considering its involution. Put differently, we are interested in the $q$-analogues of the $(\mathfrak{g}, K)$-modules. As we will see, this perspective sheds light on the unitary theory. Moreover, we extend the definition of the deformed convolution algebras to any suitable extension of $\mathbb{C}$. Within that framework, we introduce an analogue of parabolic induction. This allow us to classify the non-degenerate simple $R_q(\mathfrak{g},K)$-modules. At the same time, by working over a field of functions in $q$, we obtain an algebraic gluing of these $q$-deformations of $\mathrm{SL}(2,\mathbb{R})$ to the classical group. This is rather simpler than the operator algebraic approaches in previous works \cites{DCFlo1, DCFlo2}.

If $\mathrm{SL}(2,\mathbb{R}) = KAN$ is the standard Iwasawa decomposition and $M = \{-1,1\}$, then $P = MAN$ is a minimal parabolic subgroup of $\mathrm{SL}(2,\mathbb{R})$. We define an analogue $R_q(\mathfrak{p},M)$ of the Hecke algebra associated to $P$. Our parabolic induction then associates a non-degenerate $R_q(\mathfrak{g},K)$-module to any non-degenerate $R_q(\mathfrak{p},M)$-module. We will show that the non-zero characters of $R_q(\mathfrak{p},M)$ are naturally indexed by pairs $(\varepsilon,\lambda)\in \{-1,1\}\times \mathbb{C}^\times$. Let us denote by $\mathrm{Ind}_q(\varepsilon, \lambda)$ the $R_q(\mathfrak{g},K)$-module induced from the character associated to such a pair. For any $q \in \mathbb{R}^*_+ \setminus\{1\}$, we characterize the non-degenerate simple modules over $R_q(\mathfrak{g},K)$ as follows.

\begin{theoremA} $\quad$
\begin{enumerate}[label = (\roman*)]
    \item The module $\mathrm{Ind}_q(\varepsilon, \lambda)$ is simple unless $\lambda = \sigma q^{n}$ for $\sigma = \pm 1$ and $n$ an integer such that $\varepsilon = (-1)^{n+1}$. When simple, the modules $\mathrm{Ind}_q(\varepsilon, \lambda)$ and $\mathrm{Ind}_q(\varepsilon, \lambda^{-1})$ are isomorphic.
    \item Let $\sigma \in \{-1,1\}$, $n \in \mathbb{N}$ and let $\varepsilon = (-1)^{n+1}$. Then $\mathrm{Ind}_q(\varepsilon, \sigma q^n)$ has exactly two distinct simple submodules $D_{\sigma,n}^+$ and $D_{\sigma,n}^-$ while $\mathrm{Ind}_q(\varepsilon, \sigma q^n)$ has a unique simple submodule $Q_{\sigma,-n}$. The latter is $n$-dimensional. Moreover, the quotient module $\mathrm{Ind}_q(\varepsilon, \sigma q^n)/ D_{\sigma,n}^+ \oplus D_{\sigma,n}^-$ is isomorphic to $Q_{\sigma,n}$ while $\mathrm{Ind}_q(\varepsilon, \sigma q^{-n})/Q_{\sigma,n}$ is isomorphic to $D_{\sigma,n}^+ \oplus D_{\sigma,n}^-$.
    \item If $\sigma \in \{-1,1\}$ then $\mathrm{Ind}_q(-1, \sigma)$ is the direct sum of two inequivalent simple submodules $D_{\sigma,1}^+$ and $D_{\sigma,1}^-$.
    \item Every non-degenerate simple $R_q(\mathfrak{g},K)$-module is equivalent to a submodule of $\mathrm{Ind}_q(\varepsilon, \lambda)$ for some pair $(\varepsilon,\lambda)$.
    \item The existence of a non-zero intertwining map $\mathrm{Ind}_q(\varepsilon, \lambda) \to \mathrm{Ind}_q(\varepsilon', \lambda')$ implies that $\varepsilon = \varepsilon'$ and $\lambda' \in \{\lambda, \lambda^{-1}\}$.
\end{enumerate}
\end{theoremA}

Among these non-degenerate simple modules, we identify those which are \textit{unitarizable} thanks to the previous work of De Commer and Dzokou Talla \cite{DCDTsl2R}. As in the classical theory, they correspond up to Hilbert completion to the irreducible representations of $C^*_q(G)$.

In order to understand the classical limit $q \to 1$, we perform a change of field. Let $\mathbb{A}$ be the ring of complex-valued real analytic functions on $\mathbb{R}^*_+$ and let $\mathbb{F}$ be its field of fractions. We introduce a version $U_{\mathbb{F}}(\mathfrak{g})$ of the $q$-deformed enveloping algebra which is defined over $\mathbb{F}$. Then, inspired by Lusztig's integral form \cite{Lusztig}*{§3.1.13}, we define an $\mathbb{A}$-subalgebra $U_{\mathbb{A}}(\mathfrak{g})$ of $U_{\mathbb{F}}(\mathfrak{g})$ such that $\mathbb{F}\otimes_\mathbb{A}U_{\mathbb{A}}(\mathfrak{g}) = U_{\mathbb{F}}(\mathfrak{g})$. This $\mathbb{A}$-form specializes to $U_q(\mathfrak{g})$ at each $q \in \mathbb{R}^*_+\setminus \{1\}$ and to the classical enveloping algebra $U(\mathfrak{g})$ at $q =1$. Admitting that the non-degenerate modules over $R_q(\mathfrak{g},K)$ are the same as certain \textit{integral} $U_q(\mathfrak{g})$-modules (we refer to the fifth section), our second main result can be expressed as follows.

\begin{theoremB}
    Let $\lambda \in \mathbb{C}$ and $\varepsilon \in \{-1,1\}$. We write $\mathrm{I}({\varepsilon,\lambda})$ for the associated classical parabolically induced $(\mathfrak{g}, K)$-module. There exists an integral $U_{\mathbb{F}}(\mathfrak{g})$-module $\mathrm{I}_\mathbb{F}(\varepsilon, \lambda)$ and an $U_{\mathbb{A}}(\mathfrak{g})$-submodule of it $\mathrm{I}_\mathbb{A}(\varepsilon, \lambda)$ such that $\mathrm{I}_\mathbb{F}(\varepsilon, \lambda) = \mathbb{F}\otimes_\mathbb{A} \mathrm{I}_\mathbb{A}(\varepsilon, \lambda)$ and:
    \begin{enumerate}[label = (\arabic*)]
        \item the specialization of $\mathrm{I}_\mathbb{A}(\varepsilon, \lambda)$ at any $q\neq 1$ is $\mathrm{Ind}_q(\varepsilon, q^\lambda)$ while its specialization at $q = 1$ is $\mathrm{I}(\varepsilon,\lambda)$,
        \item for any $q$ in a suitable neighborhood of $1$, there is a natural one-one correspondence between the submodules of $\mathrm{I}_\mathbb{F}(\varepsilon, \lambda)$ and those of the specialization of $\mathrm{I}_\mathbb{A}(\varepsilon, \lambda)$ at $q$.
    \end{enumerate}
\end{theoremB}

Let us mention that the module $\mathrm{I}_\mathbb{F}(\varepsilon, \lambda)$ above is explicitly constructed using our deformed parabolic induction.

Finally, in the course of our analysis, we establish several structural results regarding $U_q(\mathfrak{g})$, including the existence of certain PBW-type bases and the determination of its center.

\subsection*{Related works.}
Our work is inspired by earlier developments due to De Commer and Dzokou Talla \cites{DCDTsl2R, DCDTinvariant, DCinduction}, still only concerning $\mathrm{SL}(2,\mathbb{R})$. As already mentioned, these authors classified the simple unitarizable non-degenerate $R_q(\mathfrak{g},K)$-modules, which amounts to a description of the spectrum of $C^*_q(G)$. Their method is quite similar to the original approach of Bargmann \cite{BargmannSL2R} in the classical setting. They also defined in \cite{DCDTinvariant} a $q$-analogue of the regular representation. In the recent preprint \cite{DCinduction}, De Commer determined which irreducible representations of $C^*_q(G)$ are weakly contained in the regular one. Moreover, he introduced an induction procedure within the framework of Doi--Kopinen modules. We refer to the sixth section for a comparison with our own approach.

In this paper, we are no longer exclusively concerned with the operator algebraic framework. In particular, the classical limit is analyzed in a purely algebraic way using the $\mathbb{A}$-form $U_{\mathbb{A}}(\mathfrak{g})$, by contrast with previous works \cites{DCFlo1, DCFlo2, Blanchard}. In an accompanying paper \cite{YGE} though, we will draw inspiration from this algebraic treatment to construct a continuous field of $q$-deformed reduced group C*-algebras $\{ C^*_{\mathrm{red},q}(G)\}_{q \in \mathbb{R}^*_+}$ specializing to the classical reduced group C*-algebra at $q = 1$. 

If $G$ is a complex semisimple group, the representation theory of $R_q(\mathfrak{g},K)$ and $C^*_q(G)$ have already been investigated. The irreducible representations of the former have been classified by Voigt--Yuncken \cite{VoigtYuncken}*{Theorems 6.47 and 6.48}. The same authors also proved an analogue of the Plancherel formula \cite{VYplancherel}, see also \cite{BuffRoche}. Moreover, a description of a large part of the spectrum of $C^*_q(G)$ is due to Arano \cites{Arano1,Arano2}. The specific case of $\mathrm{SL}(2,\mathbb{C})$ was studied in detail by Pusz and Woronowicz \cites{PWLorentz, PWRep}.

%    Finally, we point out that another kind of $q$-deformation of $SL(2, \mathbb{R})$ exists \cites{BytTes, Ip}. It is associated to a certain real form \cite{TwRealforms} of the Drinfeld--Jimbo quantum group $U_q(\mathfrak{sl_2})$ and the deformation parameter $q$ is then unimodular.

\subsection*{Structure of the paper} In Section 1 we introduce some background material on approximately unital algebras, including induction and restriction functors. We then recall the construction of the $q$-deformed convolution algebras of $\mathrm{SL}(2,\mathbb{R})$ in Section 2. In Section 3, we construct a PBW-type basis of $U_q(\mathfrak{g})$ and use the associated filtration to determine its center. In Section 4, we construct $U_\mathbb{A}(\mathfrak{g})$, then we study its structure. Section 5 explains how the representation theory of $R_q(\mathfrak{g},K)$ is related to that of $U_q(\mathfrak{g})$ and $C^*_q(G)$. We introduce the $q$-analogue of parabolic induction in Section 6. Finally, Sections 7 and 8 are devoted to the classification of non-degenerate simple $R_q(\mathfrak{g},K)$-modules in terms of parabolic induction, as well as the study of their classical limit.

\subsection*{Acknowledgments} I am grateful to Bob Yuncken and Alexandre Afgoustidis for their support, as well as Kenny De Commer for stimulating discussions.

\section{Approximately unital algebras}

Hecke algebras belong to the larger class of approximately unital algebras. Within that framework, induction and restriction of representations are reinterpreted as simple change of ring functors (see \cite{Renard}*{Chapitre 1} and \cite{KnappVoganInd}*{Chapter II}). This section gathers elementary facts on this class of algebras that we use throughout the article. In what follows, the algebras are defined over a fixed field.

\begin{definition}
    An associative algebra $A$ is said to be approximately unital if for every finite subset $F \subset A$, there exists an idempotent $e \in A$ such that $ae =ea=a$ for all $a \in F$.
\end{definition}

For the rest of the section, we fix an approximately unital algebra $A$. There is a natural ring topology on it, which we define as the unique translation invariant topology such that the collection of all sets $$\{ a \in A: ae = ea = 0 \},$$
for $e$ an idempotent of $A$, form a base of neighborhoods of $0$. Note that this topology is discrete on every unital subalgebra of $A$.

\begin{definition} \label{completion}
    By the completion of $A$ we mean its completion with respect to the above ring topology. We denote it by $\bar{A}$.
\end{definition}

The completion of $A$ is a unital algebra which contains $A$ as an ideal. Its unit is the limit of the net of all idempotents of $A$, which is a directed set for the following order:
$$e \leq e' \Longleftrightarrow ee' = e'e = e.$$
More generally, an increasing net of idempotents of $A$ is called an approximate unit of $A$ if it converges to the unit of $\bar{A}$.

Let us now provide elementary facts about the representation theory. In this paper, we only deal with left modules.

\begin{definition}\label{defNd}
    Let $A$ be an approximately unital algebra. An $A$-module $V$ is said to be non-degenerate if $AV = V$, or equivalently if for all $v \in V$ there exists an idempotent $e \in A$ such that $ev=v$. We denote by $\mathsf{Mod}_A$ the category of non-degenerate $A$-modules with $A$-linear maps as morphisms.
\end{definition}

There exists a close relationship between $A$-modules and $\bar{A}$-modules. Let us say that a unital $\bar{A}$-module $V$ is continuous if the multiplication $A \times V \to V$ is continuous when $V$ is endowed with the discrete topology. Such a module is necessarily non-degenerate as an $A$-module. Indeed, for all $v \in V$, we must have
$$v = \lim_{i \in I} e_i v,$$
where $(e_i)_{i \in I}$ is the net of all idempotents of $A$. Conversely, a non-degenerate $A$-module $V$ is also a continuous $\bar{A}$-module in a natural way: it identifies with $\bar{A} \otimes_A V$, which is a continuous $\bar{A}$-module because so is $A$. This forms a category equivalence between $\mathsf{Mod}_A$ and the category of continuous unital $\bar{A}$-modules (whose morphisms are the $\bar A$-linear maps).

As a remark, we point out that if $A$ is unital then the topology introduced above is discrete. In that case, $A$ is already complete and $\bar{A}  = A$. Moreover, the non-degenerate $A$-modules are the same as unital $A$-modules.

We now turn to the induction and restriction functors. For any $A$-module $V$, the submodule $AV$ is non-degenerate, we call it the non-degenerate part of $V$. We consider another approximately unital algebra $B$ and we assume that $A$ is equipped with a structure of non-degenerate $B$-module such that $b.(a_1a_2) = (b.a_1)a_2$ for all elements $b \in B$, $a_1,a_2 \in A$. We say that $A$ is a non-degenerate $B$-algebra.

\begin{definition}\label{induction}
    Let $V$ be a non-degenerate $B$-module. We endow the space $\mathrm{Hom}_B(A, V)$ of $B$-linear maps $A \to V$ with the structure of $A$-module induced by right multiplication of $A$ on itself. The $A$-module induced from $V$, denoted by $\mathrm{Ind}_B^A(V)$, is defined as the non-degenerate part of $\mathrm{Hom}_B(A, V)$.
\end{definition}

Note that $\mathrm{Ind}_B^A$ defines a functor $\mathsf{Mod}_B \to \mathsf{Mod}_A$. In fact, $\mathrm{Ind}_B^A$ is the right adjoint of the restriction (or forgetful) functor $\mathrm{Res}_B^A : \mathsf{Mod}_A \to \mathsf{Mod}_B$ associating to any non-degenerate $A$-module $V$ the $B$-module $A \otimes_A V\, (= V \text{ as a vector space})$.

\begin{example}\label{exampleR}

We conclude this short section by presenting a rich source of examples of approximately unital algebras. Let us work over the field $\mathbb{C}$. If $G$ is a Lie group and $K \subset G$ a compact subgroup, the Hecke algebra $R(\mathfrak{g},K)$ of the pair $(\mathfrak{g}, K)$ is defined as the convolution algebra of $K$-finite distributions on $G$ supported on $K$. It is an approximately unital algebra, see \cite{KnappVoganInd}*{§I.3}. If $G$ is a real reductive group with Lie algebra $\mathfrak{g}$ and $K$ a maximal compact subgroup of $G$ then $\mathsf{Mod}_{R(\mathfrak{g},K)}$ is, up to equivalence, the category of $(\mathfrak{g}, K)$-modules.

Now let $P$ be a parabolic subgroup of $G$ and let $L \subset P$ be a Levi factor. We denote by $\mathfrak{p}$ and $\mathfrak{l}$ their respective Lie algebras. Let $M = K \cap P = K \cap L$. The Hecke algebra $R(\mathfrak{g}, K)$ is a non-degenerate $R(\mathfrak{p},M)$-algebra. Every $(\mathfrak{l}, M)$-module $V$ can be extended to a $R(\mathfrak{p},M)$-module by imposing that the nilpotent radical of $\mathfrak{p}$ acts trivially, then the $(\mathfrak{g}, K)$-module $$\mathrm{Ind}_{R(\mathfrak{p},M)}^{R(\mathfrak{g}, K)}(V)$$ is said to be parabolically induced from $V$. Up to a twist by a certain character, this corresponds to the infinitesimal version of parabolic induction for Hilbert representations \cite{vandenBan}*{§1--2}.
\end{example}

In Sections 2 and 6 respectively, we construct approximately unital algebras $R_q(\mathfrak{g},K)$ and $R_q(\mathfrak{p},M)$ for $G = \mathrm{SL}(2,\mathbb{R})$ and $q \in \mathbb{R}_+^*\setminus\{1\}$. The $q$-analogue of parabolic induction will then be defined in the same way as above.

\section{Construction of the deformation}

We now recall the definition of the $q$-deformations of the convolution algebras of $\mathrm{SL}(2,\mathbb{R})$ introduced in \cites{DCDTsl2R,DCquantisation}. Before doing so, we just sketch the main ideas of the construction in the general case of a real semisimple linear algebraic group, refering to \cite{DCquantisation} for a complete exposition. In the following, the universal enveloping $\ast$-algebra of a real Lie algebra $\mathfrak{g}$ refers to the universal enveloping algebra of its complexification endowed with the unique ring involution such that the elements of $\mathfrak{g}$ are skew-adjoint and whose restriction to $\mathbb{C}$ is the complex conjugation.

\subsection*{Overview of the general Case}
Let $G_\mathbb{C} \subset \mathrm{GL}(n,\mathbb{C})$ be a semisimple linear algebraic group defined over $\mathbb{R}$ which is stable under conjugate-transpose. Let $G = G_\mathbb{C} \cap \mathrm{GL}(n,\mathbb{R})$ be the associated group of real points, which we assume to be connected. We denote by $U$ the intersection of $G_\mathbb{C}$ with the special unitary group and $K = U \cap G$. Then $U$ and $K$ are maximal compact subgroups of $G_\mathbb{C}$ and $G$ respectively. We write $\mathfrak{g}$, $\mathfrak{g}_\mathbb{C}$, $\mathfrak{u}$, $\mathfrak{k}$ for the Lie algebras of these groups. Let $\theta$ be the Cartan involution on $\mathfrak{g}$ canonically associated to the compact real form $\mathfrak{u}$ of $\mathfrak{g}_\mathbb{C}$. We extend it to $\mathfrak{g}_\mathbb{C}$ by complexification.

Now let us choose a $\theta$-stable Cartan subalgebra $\mathfrak{t} \subset \mathfrak{u}$ and a certain Weyl chamber of its complexification. We denote by $\mathfrak{n}_+$ the sum of the positive root spaces of $\mathfrak{g}_\mathbb{C}$ with respect to its Cartan subalgebra $\mathfrak{t}_\mathbb{C}$ and our choice of Weyl chamber. Then $i\mathfrak{t}\oplus \mathfrak{n}_+$ is a real Lie subalgebra of $\mathfrak{g}_\mathbb{C}$. Let $B$ be the Killing form of $\mathfrak{g}_\mathbb{C}$. Its imaginary part $\mathrm{Im}(B)$ is a non-degenerate invariant bilinear form on $(\mathfrak{g}_\mathbb{C})_{|\mathbb{R}}$ (the underlying real Lie algebra of $\mathfrak{g}_\mathbb{C}$). One can show that $((\mathfrak{g}_\mathbb{C})_{|\mathbb{R}}, \mathfrak{u}, i\mathfrak{t}\oplus \mathfrak{n}_+)$ is a Manin triple \cite{EtingofSchiffman}*{§4.1} for $(\mathfrak{g}_\mathbb{C})_{|\mathbb{R}}$ endowed with $\mathrm{Im}(B)$. This automatically equips $\mathfrak{u}$ and $i\mathfrak{t}\oplus \mathfrak{n}_+$ with a structure of Lie bialgebra for which they are dual to each other. In particular, $(\mathfrak{g}_\mathbb{C})_{|\mathbb{R}}$ can be viewed as the classical Drinfeld double \cite{EtingofSchiffman}*{§4.2} of $\mathfrak{u}$. Moreover we have $\mathfrak{g} = \mathfrak{k} \oplus \mathfrak{k}^\perp$, where $\mathfrak{k}^\perp \subset \mathfrak{u}^* = i\mathfrak{t}\oplus \mathfrak{n}_+$ is the orthogonal of $\mathfrak{k}$. 

We fix a parameter $q \in \mathbb{R}_+ \setminus\{1\}$. Let $U_q(\mathfrak{u})$ be the usual quantized universal enveloping $\ast$-algebra of $\mathfrak{u}$, see for instance \cite{VoigtYuncken}*{§4.3}. Using Letzter's work \cite{Letzter99}, one can construct a coideal $\ast$-subalgebra $U_q(\mathfrak{k})$ of $U_q(\mathfrak{u})$, so that $(U_q(\mathfrak{u}), U_q(\mathfrak{k}))$ may be seen as a quantization of the symmetric pair $(\mathfrak{u}, \mathfrak{k})$. Let us also consider the quantized $\ast$-algebra of regular functions on $U$, denoted by $\mathcal{O}_q(U)$. The latter is dually paired with $U_q(\mathfrak{u})$. We define
$$\mathcal{O}_q(K\backslash U) = \{ f \in \mathcal{O}_q(U) : \forall x \in U_q(\mathfrak{k}), \,\,  f \triangleleft x = \epsilon(x)f\},$$
where $\epsilon$ is the counit of $U_q(\mathfrak{u})$ and $\triangleleft$ denotes the right regular action of $U_q(\mathfrak{u})$ on $\mathcal{O}_q(U)$.
By the quantum duality principle \citelist{ \cite{ConciniKacProcesi}*{§7.6} \cite{Song} \cite{GavariniQDP}}, $(\mathcal{O}_q(U), \mathcal{O}_q(K\backslash U))$ can be interpreted as a quantization of the pair $(\mathfrak{u}^*, \mathfrak{k}^\perp)$, at the level of universal enveloping $\ast$-algebras. Moreover, the Drinfeld double construction $D(U_q(\mathfrak{u}), \mathcal{O}_q(U))$ associated to the dual pair of Hopf algebras $(U_q(\mathfrak{u}), \mathcal{O}_q(U))$ may be regarded as a quantization of $(\mathfrak{g}_\mathbb{C})_{|\mathbb{R}} = \mathfrak{u} \oplus \mathfrak{u}^*$, see \cite{EtingofSchiffman}*{§12.2}. In this context, the subalgebra of $D(U_q(\mathfrak{u}), \mathcal{O}_q(U))$ generated by $U_q(\mathfrak{k})$ and $\mathcal{O}_q(K\backslash U)$ may be interpreted as a quantized enveloping $\ast$-algebra of $\mathfrak{g} = \mathfrak{k} \oplus \mathfrak{k}^\perp$. The reader should be warned that this is not a quantization of $\mathfrak{g}$ for a certain Lie bialgebra structure but rather a quantization of $\mathfrak{g}$ seen as a subalgebra of the Lie bialgebra $(\mathfrak{g}_\mathbb{C})_{|\mathbb{R}}$. As a reflection of the fact that $\mathfrak{g}$ is not a sub-bialgebra of $(\mathfrak{g}_\mathbb{C})_{|\mathbb{R}}$, the resulting quantization is not a Hopf subalgebra of $D(U_q(\mathfrak{u}), \mathcal{O}_q(U))$ in general, but it is still a coideal.

The analogue of the Hecke algebra is constructed in the same way, except that $U_q(\mathfrak{k})$ is replaced by an integrated version of it. In the classical context, this corresponds to replacing the distributions on $K$ supported on $\{1\}$ by the $K$-finite distributions on $K$. The quantized versions of the enveloping algebra and the Hecke algebra are naturally equipped with a $\ast$-structure. In contrast to the classical case, the $q$-analogue of the Hecke algebra admits an enveloping C*-algebra that we interpret as the quantized group C*-algebra.

\subsection*{Conventions and notations}

Antipodes, counits and coproducts will be generically denoted by $S$, $\epsilon$ and $\Delta$ respectively. We will freely use the sumless Sweedler notation. If $A$ is a unital algebra $C$ a counital coalgebra, a dual pairing between $A$ and $C$ is by definition a bilinear form $\langle \cdot\,, \cdot \rangle$ on $A \times C$ such that $\langle a\otimes b, \Delta(c)\rangle  = \langle ab, c\rangle$ and $\langle 1, c\rangle = \epsilon(c)$. Given such a dual pairing, we may equip $C$ with a structure of $A$-bimodule whose left and right actions $\triangleright$ and $\triangleleft$ are defined by
$$a \triangleright c  = c_{(1)}\langle a, c_{(2)}\rangle, \qquad c \triangleleft a  = \langle a, c_{(1)}\rangle c_{(2)}.$$
The notations $\triangleright$ and $\triangleleft$ will always be used for that purpose.

As we will consider algebras and modules defined over a field other than $\mathbb{C}$, we have decided to use a quite general class of fields for most of the paper, including this section. Let us define an involutive extension of $\mathbb{C}$ as an extension field of $\mathbb{C}$ equipped with an involutive automorphism whose restriction to $\mathbb{C}$ is the complex conjugation. For \textit{the rest of the paper}, $\mathbb{K}$ denotes any involutive extension of $\mathbb{C}$. The image of any element $\lambda$ of $\mathbb{K}$ by the involution will be denoted by $\lambda^*$. We call $\lambda$ self-adjoint if $\lambda^* = \lambda$. The notions of $\ast$-algebra, Hopf $\ast$-algebra and dual pairing of Hopf $\ast$-algebras as in \cite{KlimykSchmudgen}*{§1.2.7} extend without ambiguity to algebras defined over $\mathbb{K}$. The same goes for all the constructions of \cites{DCDTsl2R,DCquantisation} which do not involve norms. In what follows, except if it is explicitly mentioned otherwise, the tensor products are taken with respect to the field $\mathbb{K}$.

Moreover, in any ring, for any invertible element $q$ and any integer $n$, we write
$$[n]_q = q^n + q^{n-2} + ... + q^{-n}, \qquad \{n\}_q = q^n + q^{-n}.$$

Finally, for \textit{the rest of the paper}, we write: $G = \mathrm{SL}(2,\mathbb{R})$, $K = \mathrm{SO}(2)$, $U = \mathrm{SU}(2)$, $\mathfrak{g} = \mathfrak{sl}(2,\mathbb{R})$, $\mathfrak{k} =\mathfrak{so}(2)$, $\mathfrak{u} =\mathfrak{su}(2)$.

\subsection*{The quantized enveloping algebra} For the rest of the section, we fix $q \in \mathbb{K}^\times$. We assume that it admits self-adjoint square roots and we fix one of them $q^{1/2}$. We also assume that $q$ is not a root of unity. 

Let $U_q(\mathfrak{sl_2})$ be the Drinfeld--Jimbo algebra \cite{Jantzen}*{Chapter 1} associated to $\mathfrak{sl_2}$ and with deformation parameter $q$. Let $E$, $F$, $k$ and $k^{-1}$ be its canonical generators. We endow $U_q(\mathfrak{sl_2})$ with the Hopf algebra structure such that
$$\Delta(E) = E \otimes 1 + k \otimes E, \quad \Delta(F) = F \otimes k^{-1} + 1 \otimes F, \quad \Delta(k) = k \otimes k,$$
and with the $\ast$-structure determined by
$$E^* = Fk, \qquad F^* = k^{-1}E, \qquad k^* = k.$$
The resulting Hopf $\ast$-algebra is denoted by $U_q(\mathfrak{u})$.

Let $\mathcal{O}_q(\mathrm{SL}(2))$ be the Hopf algebra of matrix coefficients of type 1 finite dimensional representations of $U_q(\mathfrak{sl_2})$ (we refer to \cite{Jantzen}*{Chapter 5} for the meaning of type 1 representations). The canonical pairing between $\mathcal{O}_q(\mathrm{SL}(2))$ and $U_q(\mathfrak{sl_2})$ is a non-degenerate dual pairing of Hopf algebras, we denote it by $\langle \cdot \,,\cdot \rangle$. There is a unique $\ast$-structure on $\mathcal{O}_q(\mathrm{SL}(2))$ so that $\langle \cdot \,,\cdot \rangle$ becomes a Hopf $\ast$-algebra pairing with $U_q(\mathfrak{u})$. Equipped with that $\ast$-structure, $\mathcal{O}_q(\mathrm{SL}(2))$ is instead denoted by $\mathcal{O}_q(U)$. Another common way to describe $\mathcal{O}_q(U)$ is as follows: it is the universal unital $\ast$-algebra generated by the matrix entries of a $2 \times 2$ matrix $u = (u_{\alpha \beta})_{\alpha, \beta}$ subject to the relations
$$u^* = \mathrm{diag}(1,q)\, \mathrm{adj}(u) \,\mathrm{diag}(1,q^{-1}), \qquad u^*u = uu^* = 1,$$
$\mathrm{adj}(u)$ denoting the adjugate matrix of $u$; the coalgebra structure is expressed by $\Delta(u_{\alpha\beta}) = \sum_\gamma u_{\alpha\gamma} \otimes u_{\gamma\beta}$ and $\epsilon (u) = 1$; the pairing with $U_q(\mathfrak{u})$ is characterized by
$$\langle u, E \rangle = \begin{pmatrix}
    0 & q^{1/2}\\ 0 & 0
\end{pmatrix}, \quad \langle u, F \rangle = \begin{pmatrix}
    0 & 0 \\ q^{-1/2} & 0
\end{pmatrix}, \quad \langle u, k \rangle = \begin{pmatrix}
    q& 0 \\ 0 & q^{-1}
\end{pmatrix}.$$

One can define a left coideal $\ast$-subalgebra $U_q(\mathfrak{k}) \subset U_q(\mathfrak{u})$ and a right coideal $\ast$-subalgebra $\mathcal{O}_q(K\backslash U) \subset \mathcal{O}_q(U)$ as follows. The first one is the unital subalgebra of $U_q(\mathfrak{u})$ generated by the self-adjoint element \begin{equation}\label{theta} \theta = iq^{-1/2}(E- Fk).\end{equation} Then we put
$$\mathcal{O}_q(K\backslash U) = \{ y \in \mathcal{O}_q(U) : y \triangleleft \theta = 0\}.$$
The latter has a simple description in terms of generators and relations \cite{DCDTsl2R}*{Lemma 2.4} that we now present. Let $X,Y,Z \in \mathcal{O}_q(U)$ be defined by
\begin{equation} \label{defXYZ} u^*\begin{pmatrix}
    0 & -1 \\ 1 & 0
\end{pmatrix} u = \begin{pmatrix}
    iq^{-1}Z & -Y \\
    X & -iqZ
\end{pmatrix}. \end{equation}
Then $X,Y,Z$ belong to $\mathcal{O}_q(K\backslash U)$ and generate it as a unital algebra. They satisfy $X^*=Y, \, Z^* = Z$ and the following relations:
\begin{equation}\label{XYZ}
\begin{gathered}
    XZ = q^2ZX, \quad ZY = q^2YZ,\\
    XY + q^2Z^2 = 1, \quad YX + q^{-2}Z^2 = 1.
\end{gathered}
\end{equation}
In fact, $\mathcal{O}_q(K\backslash U)$ can be realized as the universal unital algebra generated by $X,Y,Z$ subject to the above relations.

Let $U_q((\mathfrak{g}_\mathbb{C})_{|\mathbb{R}})$ be the Drinfeld double Hopf algebra \citelist{\cite{KlimykSchmudgen}*{§8.2.1} \cite{VoigtYuncken}*{§2.2.4}} associated to the skew-pairing
$$\begin{array}{cccll} U_q(\mathfrak{u})^{\mathrm{cop}} \times \mathcal{O}_q(U) &\longrightarrow&  \mathbb{K} \\
(x,y) & \longmapsto &\langle y, x \rangle,\end{array}$$
where $U_q(\mathfrak{u})^{\mathrm{cop}}$ denotes the $\ast$-algebra $U_q(\mathfrak{u})$ endowed with the opposite coproduct. Let us equip  $U_q((\mathfrak{g}_\mathbb{C})_{|\mathbb{R}})$ with the $\ast$-structure such that the natural embeddings $U_q(\mathfrak{u})^{\mathrm{cop}} \hookrightarrow U_q((\mathfrak{g}_\mathbb{C})_{|\mathbb{R}})$ and $\mathcal{O}_q(U)\hookrightarrow U_q((\mathfrak{g}_\mathbb{C})_{|\mathbb{R}})$ are morphisms of Hopf $\ast$-algebras.

\begin{definition}
    The quantized enveloping $\ast$-algebra $U_q(\mathfrak{g})$ of $\mathfrak{g}$ is defined as the subalgebra of $U_q((\mathfrak{g}_\mathbb{C})_{|\mathbb{R}})$ generated by $U_q(\mathfrak{k})$ and $\mathcal{O}_q(K\backslash U)$.
\end{definition}

It follows from the above that $U_q(\mathfrak{g})$ is a right coideal of $U_q((\mathfrak{g}_\mathbb{C})_{|\mathbb{R}})$ and that it is $\ast$-stable. Concretely, $U_q(\mathfrak{g})$ can be realized as the universal unital $\ast$-algebra generated by copies of $U_q(\mathfrak{k})$ and $\mathcal{O}_q(K\backslash U)$ and whose elements are subject to the following exchange relations:
$$xy = y_{(1)}( x\triangleleft y_{(2)}), \qquad (y \in \mathcal{O}_q(K\backslash U), x \in U_q(\mathfrak{k})).$$
The multiplication induces linear isomorphisms $\mathcal{O}_q(K\backslash U) \otimes U_q(\mathfrak{k}) \to U_q(\mathfrak{g})$ and $U_q(\mathfrak{k}) \otimes \mathcal{O}_q(K\backslash U) \to U_q(\mathfrak{g})$, see \cite{DCDTsl2R}*{Lemma 1.4}. Moreover, one can also describe $U_q(\mathfrak{g})$ as the universal unital algebra generated by $X,Y,Z,\theta$ subject to the relations (\ref{XYZ}) and the following ones \cite{DCDTsl2R}*{Lemma 2.6}:
\begin{equation}\label{thetaXYZ}
    \begin{gathered}
        qX\theta - q^{-1}\theta X = q\theta Y - q^{-1}Y\theta = [2]_qZ, \\
        \theta Z - Z \theta = Y-X.
    \end{gathered}
\end{equation}
We point out that the $\ast$-algebra $U_q(\mathfrak{g})$ does not depend on the choice of the square root $q^{1/2}$. At the level of the generators though, replacing $q^{1/2}$ by $-q^{1/2}$ corresponds to the change of variables $\theta \mapsto - \theta,\, X \mapsto -X,\, Y \mapsto -Y,\, Z \mapsto Z$.

Let us mention that the authors of \cites{DCDTsl2R, DCDTinvariant} introduced an extra parameter $t$ in their deformation, each value of that parameter corresponding to a distinct conjugate of $\mathrm{SL}(2,\mathbb{R})$ in $\mathrm{SL}(2,\mathbb{C})$. Here, we are concerned with the case $t = 0$ corresponding to $\mathrm{SL}(2,\mathbb{R})$ itself.

\subsection*{Quantization of the Hecke algebra}
Let us now define the quantized Hecke algebra $R_q(\mathfrak{g},K)$. This will need several steps. We refer to \cite{DCquantisation} for more generality and precision regarding the upcoming constructions.

First of all, let $\mathcal{O}'_q(U^n)$ be the algebraic dual of $\mathcal{O}_q(U)^{\otimes n}$ for all $n \in \mathbb{N}$. The coalgebra structure on $\mathcal{O}_q(U)^{\otimes n}$ induces by duality a structure of unital algebra on $\mathcal{O}'_q(U^n)$. Moreover, we equip the latter with the coarsest topology such that all the elements of $\mathcal{O}_q(U)^{\otimes n}$, considered as linear functionals $\mathcal{O}'_q(U^n) \to \mathbb{K}$, are continuous when $\mathbb{K}$ is equipped with the discrete topology. This endows $\mathcal{O}'_q(U^n)$ with a structure of complete topological ring. Let us denote by $\langle \cdot\,,\cdot \rangle$ the canonical pairing $\mathcal{O}_q(U)^{\otimes n} \times \mathcal{O}'_q(U^n) \to \mathbb{K}$. The topology on $\mathcal{O}'_q(U^n)$ is precisely defined so that the closure of a linear subspace $V \subset \mathcal{O}'_q(U^n)$ is its bi-annihilator with respect to this pairing. As a corollary, we see that $U_q(\mathfrak{u})^{\otimes n}$ is dense in $\mathcal{O}'_q(U^n)$. Hence we can interpret $\mathcal{O}'_q(U^n)$ as a completion of $U_q(\mathfrak{u})^{\otimes n}$ for the above mentioned ring topology. In this way, the antipode $S : U_q(\mathfrak{u}) \to U_q(\mathfrak{u})$, the involution $\ast : U_q(\mathfrak{u}) \to U_q(\mathfrak{u})$ and the coproduct $\Delta : U_q(\mathfrak{u}) \to U_q(\mathfrak{u})^{\otimes 2}$ extend by continuity to linear maps $\mathcal{O}'_q(U) \to \mathcal{O}'_q(U)$ and $\mathcal{O}'_q(U) \to \mathcal{O}'_q(U^2)$ respectively, which we label with the same letters. With these new operations, $\mathcal{O}'_q(U)$ is almost a Hopf $\ast$-algebra: it satisfies all the axioms except that the coproduct take values in $\mathcal{O}'_q(U^2)$. In addition, the pairing $\mathcal{O}_q(U) \times \mathcal{O}'_q(U) \to \mathbb{K}$ satisfies the axioms of a pairing of Hopf $\ast$-algebras (for the product-coproduct compatibility we use the pairing $\mathcal{O}_q(U)^{\otimes 2} \times \mathcal{O}'_q(U^2) \to \mathbb{K}$). Finally, the actions $\triangleleft$ and $\triangleright$ of $\mathcal{O}_q(U)$ on $U_q(\mathfrak{u})$ extend to actions on $\mathcal{O}'_q(U)$ by continuity.

Let $\mathcal{O}'_q(K)$ and $\mathcal{O}'_q(U \times K)$ be the closures of $U_q(\mathfrak{k})$ and $U_q(\mathfrak{g}) \otimes U_q(\mathfrak{k})$ in $\mathcal{O}'_q(U)$ and $\mathcal{O}'_q(U^2)$, respectively. Since $U_q(\mathfrak{k})$ is a left coideal of $U_q(\mathfrak{g})$ we have
$$\Delta[\mathcal{O}'_q(K)] \subset \mathcal{O}'_q(U \times K).$$
A different but equivalent way of putting this is
\begin{equation}\label{generalizedcoideal} \mathcal{O}_q'(K) \triangleleft \mathcal{O}_q(U) \subset \mathcal{O}_q'(K). \end{equation}
Of course, $\mathcal{O}'_q(K)$ is also a $\ast$-subalgebra of $\mathcal{O}'_q(U)$. We define the quantized Hecke algebra of $K$ as the sum of all finite dimensional two-sided ideals of $\mathcal{O}'_q(K)$ and we denote it by $R_q(K)$. Note that $R_q(K)$ is itself a two-sided ideal of $\mathcal{O}'_q(K)$.

Let $\mathcal{O}'_q(G)$ be the universal $\ast$-algebra containing $\mathcal{O}'_q(K)$ and $\mathcal{O}_q(K\backslash U)$ as unital $\ast$-subalgebras and whose elements are subject to the following relations (this is essentially \cite{DCquantisation}*{Definition 3.1}):
\begin{equation}\label{exchange} xy = y_{(1)}( x\triangleleft y_{(2)}), \qquad (y \in \mathcal{O}_q(K\backslash U), x \in \mathcal{O}_q'(K)).\end{equation}
This expression is well defined because of the inclusion (\ref{generalizedcoideal}) and the fact that $\mathcal{O}_q(K\backslash U)$ is a right coideal of $\mathcal{O}_q(U)$. Note also that this is the same definition as $U_q(\mathfrak{g})$ except that we replaced $U_q(\mathfrak{k})$ by its completion $\mathcal{O}_q'(K)$.

\begin{definition}
    The quantized Hecke algebra of the pair $(\mathfrak{g}, K)$ is defined as the two-sided ideal of $\mathcal{O}'_q(G)$ generated by $R_q(K)$ and it is denoted by $R_q(\mathfrak{g},K)$.
\end{definition}

As observed in \cite{DCquantisation}*{§3}, $R_q(K)$ is stable by $U_q(\mathfrak{u}) \,\triangleright $ and $\triangleleft\, U_q(\mathfrak{u})$ so the multiplication map induces linear isomorphisms
\begin{equation}\label{ProduitGen} R_q(K) \otimes \mathcal{O}_q(K\backslash U) \rightarrow R_q(\mathfrak{g},K)\quad \text{and}\quad   \mathcal{O}_q(K\backslash U)\otimes R_q(K) \rightarrow R_q(\mathfrak{g},K).\end{equation}
Since $R_q(K)$ is a $\ast$-algebra, so is $R_q(\mathfrak{g},K)$.

Finally, we give a concrete description of $R_q(K)$, which will make $R_q(\mathfrak{g},K)$ easier to handle. We begin with the following fundamental result:
\begin{lemma}\label{diagtheta} Let $V$ be an irreducible $n+1$-dimensional $U_q(\mathfrak{u})$-module. Then $\theta$ (defined by (\ref{theta})) acts on $V$ as a diagonalizable operator whose set of eigenvalues is
$$\{ [m]_q : m \in \mathbb{Z}, \,|m| \leq n, n-m\,\text{even} \}.$$
\end{lemma}

This result is already known for $\mathbb{K} = \mathbb{C}$ and $q \in (0,1)$, see \citelist{\cite{DCDTsl2R}*{Theorem 2.1}\cite{Koo}*{Theorem 4.3}}.
Here we adapt the proof to more general fields than $\mathbb{C}$.

\begin{proof} Before actually proving the lemma, let us make some preparations. First, let $s$ be an indeterminate and let $A$ be the localization of $\mathbb{C}[s]$ by the multiplicatively closed subset generated by $s$ and $s^j-1$ for all $j\in\mathbb{N}$. Consider the sequence $(P_j)_{0\leq j \leq n}$ of polynomials in one variable $t$ and with coefficients in $A$ defined by
$$t P_j = (1-s^{j-n})P_{j+1} - (1-s^j)s^{-n}P_{j-1} \quad (0 \leq j\leq n-1), \qquad P_{0} = 1,$$
with the convention $P_{-1} = 0$. Note that $P_j$ is of degree $j$ and has the same parity as $j$ for all $j = 0,1,...,n$. Consider also the following quadratic form on $A[t]$:
\begin{gather*}(f,g) = \frac{1}{(-1;s)_{n+1}} \sum _{j=0}^nf(t_j)g(t_j)w_j,\\ w_j = \frac{(s^j+s^{n-j})(s^2;s^2)_n}{(s^2;s^2)_j(s^2;s^2)_{n-j}}, \quad t_j = s^{-j}-s^{j-n},\end{gather*}
where we used the Pochhammer symbol $(a;b)_r = \prod_{j=0}^r (1-ab^j)$. We have
$$(P_j,P_k) = \delta_{j,k}(-s^{-n})^j\frac{(s;s)_j}{(s^{-n};s)_j} \qquad (0\leq j,k\leq n).$$
Indeed, the evaluation of these relations at values of $s$ in $(0,1)$ are just the orthogonality relations of the dual $q$-Krawtchouk polynomials, see \cite{qOrthogonal}*{§14.17}.
If $\alpha_n$ is the leading coefficient of $P_n$ then let us write $Q = \alpha_n\prod_{j=0}^n(t-t_j)$. Since $tP_n-Q$ has degree strictly less than $n$ and $Q$ is in the kernel of the above quadratic form, we have
$$tP_n = Q + \sum_{j=0}^n\frac{(tP_n, P_j)}{(P_j,P_j)}P_j.$$
If $j < n-1$ then $tP_j$ has degree less than $n$, so $(tP_n,P_j) = (P_n,tP_j) = 0$. Moreover, since $t_{n-j} = -t_j$ and $w_j = w_{n-j}$ for all $j$, the fact that $tP_n$ and $P_n$ have opposite parity implies that $(tP_n,P_n) = 0$. We also have
\begin{align*}
    (tP_n, P_{n-1}) &= (P_n, (1-s^{-1})P_n + (1-s^{n-1})s^{-n}P_{n-2})\\
    &= (1-q^{-1})(P_n,P_n) = (1-s^{-n})(P_{n-1},P_{n-1}).
\end{align*}
We deduce from these facts that $tP_n = Q + (1-s^{-n})P_{n-1}$.

Now let us prove lemma. The evaluation $s= q^2$ induces a ring morphism $A \to \mathbb{K}$. Let $(\tilde P_j)_{0\leq j \leq n}$ and $\tilde Q$ be the polynomials over $\mathbb{K}$ obtained from $(P_j)_{0\leq j \leq n}$ and $Q$ by pushforward with respect to that morphism. Let $(v_j)_{0\leq j\leq n}$ be a basis of $V$ such that $Ev_j = [j+1]_qv_{j+1}, Fv_j = [n-j+1]_q v_{j-1}, Kv_j = q^{2j -n}v_j$, with the convention $v_{-1} = v_{n+1} =0$. Let $\lambda \in \mathbb{K}$ and let $w = \sum_jw^jv_j$ be a non-zero element of $V$. A direct computation shows that $w$ is an eigenvector of $\theta$ with eigenvalue $q^n(q-q^{-1})^{-1}\lambda$ if and only if
$$\lambda \tilde{w}^j = (1-q^{2j-2n})\tilde{w}^{j+1} - (1-q^{2j})q^{-2n}\tilde{w}^{j-1} \qquad (0 \leq j \leq n),$$
where $\tilde{w}_j = (-i)^jq^{j^2/2 +(1-n)j}w^j$ if $0 \leq j \leq n$ and $\tilde{w}^j = 0$ otherwise. Our preliminary analysis shows that this condition is equivalent to the following one:
$$\begin{cases}
    \tilde{w}^j = \tilde P_j(\lambda)\tilde{w}^0 \quad (0\leq j \leq n)\\
    \tilde{Q}(\lambda)\tilde{w}^0 =0,
\end{cases}$$
Thus, we see that $q^n(q-q^{-1})^{-1}\lambda$ is an eigenvalue of $\theta$ if and only if $\lambda$ is a root of $\tilde{Q}$, equivalently if it is of the form $q^{-2j} - q^{2j-2n}$ for $0\leq j \leq n$. In conclusion, the action of $\theta$ on $V$ is diagonalizable with set of eigenvalues
$\{[n-2j]_q : j = 0,...,n \}.$
\end{proof}

Thus, for each $n \in \mathbb{Z}$, the generic spectral projection of $\theta$ onto its $[n]_q$-eigenspace is a well-defined self-adjoint idempotent of $\mathcal{O}'(K)$; we denote it by $e_n$. In addition, for every family of scalars $(\lambda_n)_{n \in \mathbb{Z}}$, the sum $\sum_{n \in \mathbb{Z}} \lambda_ne_n$
is convergent in $\mathcal{O}'(K)$. One readily sees that every element of $\mathcal{O}'(K)$ is of that form and that $R_q(K)$ is the linear span of the family $(e_n)_{n \in \mathbb{Z}}$.

In particular, $R_q(K)$ is an approximately unital algebra. From (\ref{ProduitGen}), we infer that $R_q(G)$ is also an approximately unital algebra. Moreover, any approximate unit of $R_q(K)$ is an approximate unit of $R_q(G)$. An example of such an approximate unit is $(e_{(m)} = \sum_{|n|\leq m}e_n)_{m\in \mathbb{N}}$.

\begin{remark}\label{completionR} Let us endow $\mathcal{O}_q'(G)$ with the ring topology such that $$[(1-e_{(m)})\mathcal{O}_q'(G)(1-e_{(m)})]_{m \in \mathbb{N}}$$ is a neighborhood basis for $0$. The restriction of this topology to $R_q(\mathfrak{g},K)$ and $R_q(K)$ coincides with the topology of approximately unital algebra in both cases. In addition, the dual topology on $\mathcal{O}'_q(K)$ introduced at the beginning of the current section coincides with the subset topology for the inclusion $\mathcal{O}'_q(K) \subset \mathcal{O}_q'(G)$. In particular, $\mathcal{O}'_q(K)$ is the completion of $R_q(K)$ in the sense of Definition \ref{completion}. By contrast, $\mathcal{O}_q'(G)$ is only a dense subring of the completion of $R_q(\mathfrak{g},K)$. \end{remark}

\subsection*{Completion into a C*-algebra} Let $\mathbb{K} = \mathbb{C}$. Unlike the classical situation, the quantized Hecke algebra $R_q(\mathfrak{g},K)$ has a universal C*-completion. We denote the latter by $C^*_q(G)$ and interpret it as a $q$-analogue of the maximal group C*-algebra of $G$. In fact, $R_q(\mathfrak{g},K)$ satisfies a stronger property, introduced in \cite{DCquantisation}*{Definition 2.3} under the name \textit{strong uniform C*-boundedness}. Namely, any pre-Hilbert space $\ast$-representation of $R_q(\mathfrak{g},K)$ is bounded and therefore can be completed into a Hilbert $\ast$-representation of $C^*_q(G)$. We refer to \cite{DCquantisation}*{Propositions 2.14 and 3.5} for a proof.

\section{The center of the quantized enveloping algebra}

In this section, we explicitly determine the center of $U_q(\mathfrak{g})$, more precisely we prove that it is generated as an algebra by a single element. We use this to prove an analogue of the Harish-Chandra isomorphism which will motivate the definition of the infinitesimal character for $R_q(\mathfrak{g},K)$-modules.

We use the same conventions as in the previous section. Recall the definition of the generators $X$, $Y$, $Z$ and $\theta$ of $U_q(\mathfrak{g})$. Throughout this section, for every $n \in \mathbb{Z}$, we write
$$W^{(n)} = \begin{cases}
    Y^n \quad \text{if } n \geq 0 \\
    X^{-n} \quad \text{otherwise}.
\end{cases}$$

\begin{proposition} \label{structureAN}
    For every $(n,m) \in \mathbb{Z}^2$ with $m \geq 0$, let $b_{n,m} = W^{(n)}Z^m$.
    For each $r \in \mathbb{Z}_{\geq 0}$, let $\mathcal{O}_q(K \backslash U)^{(r)} = \mathrm{span}(b_{n,m} : n + |m| \leq r)$. Then $(b_{n,m})_{n \geq 0, m \in \mathbb{Z}}$ is a basis of $\mathcal{O}_q(K \backslash U)$ and $(\mathcal{O}_q(K \backslash U)^{(r)})_{r \geq 0}$ is a filtration of the algebra $\mathcal{O}_q(K \backslash U)$.
\end{proposition}

\begin{proof}
    From (\ref{XYZ}) we derive the following relations:
    \begin{equation}\label{actionb}
        \begin{gathered}
            Xb_{n,m} = \begin{cases}
                b_{n-1,m} - q^{4n-2} b_{n-1,m+2}\quad \text{if }n>0\\
                b_{n-1,m}  \quad \text{else},
            \end{cases}\\
            Y b_{n,m} = \begin{cases}
                b_{n+1,m} - q^{4n-2} b_{n+1,m+2}\quad \text{if }n < 0\\
                b_{n+1,m}  \quad \text{else},
            \end{cases}\\
            Z b_{n,m} = q^{2n}b_{n,m+1}.
        \end{gathered}
    \end{equation}
    This shows that the linear span of $(b_{n,m})_{n,m}$ is a left ideal of $\mathcal{O}_q(K \backslash U)$, hence all of $\mathcal{O}_q(K \backslash U)$ because $b_{0,0} = 1$. This also proves that $\mathcal{O}_q(K \backslash U)^{(1)} \mathcal{O}_q(K \backslash U)^{(r)} \subset \mathcal{O}_q(K \backslash U)^{(r+1)}$ for all $r \in \mathbb{Z}_{\geq0}$. Since
    $$\mathcal{O}_q(K \backslash U)^{(r')} = \mathrm{span}\left[\mathcal{O}_q(K \backslash U)^{(1)}\mathcal{O}_q(K \backslash U)^{(r'-1)} \right]$$
    for all $r' \in \mathbb{N}$, a easy inductive argument shows that $(\mathcal{O}_q(K \backslash U)^{(r)})_{r \geq 0}$ is a filtration of the algebra $\mathcal{O}_q(K \backslash U)$.

    Finally, let us prove that the elements $b_{n,m}$ are linearly independent. Let $V$ be a vector space having a basis $(v_{n,m})_{n,m}$ indexed by $\mathbb{Z} \times \mathbb{Z}_{\geq 0}$. We let $X, Y, Z$ act linearly on $V$ via (\ref{actionb}), just replacing $b_{n,m}$ by $v_{n,m}$. These actions satisfy (\ref{XYZ}), so they extend by universality to a unique structure of $\mathcal{O}_q(K \backslash U)$-module on $V$. Then we have $b_{n,m}v_{0,0} = v_{n,m}$, which proves that the $b_{n,m}$ are indeed linearly independent.
\end{proof}

\begin{corollary}\label{filtrationUq}
    The family  $(\theta^{m_1}W^{(n)}Z^{m_2})_{m_1,m_2 \in \mathbb{Z}_{\geq 0}}^{n \in \mathbb{Z}}$ is a basis of $U_q(\mathfrak{g})$. For each $r \in \mathbb{Z}_{\geq 0}$, let us write
    $$U_q(\mathfrak{g})^{(r)} = \sum_{n \leq r}\theta^n \mathcal{O}_q(K\backslash U)^{(r-n)}.$$
    Then $(U_q(\mathfrak{g})^{(r)})_{r\geq 0}$ is a filtration of the algebra $U_q(\mathfrak{g})$.
\end{corollary}

\begin{proof}
    The first claim follows from the previous proposition and the fact that the multiplication induces a linear isomorphism $U_q(\mathfrak{k}) \otimes \mathcal{O}_q(K\backslash U) \to U_q(\mathfrak{g})$. Regarding the second assertion, an induction on $r,n \geq 0$ using relations (\ref{thetaXYZ}) shows that
    $$\mathcal{O}_q(K\backslash U)^{(r)} \theta^n \subset \sum_{0\leq m\leq n}\theta^m \mathcal{O}_q(K\backslash U)^{(r)}.$$
    Together with the last proposition, this proves that $(U_q(\mathfrak{g})^{(r)})_{r\geq 0}$ is a filtration of the algebra $U_q(\mathfrak{g})$.
\end{proof}

We can now describe explicitly the center $\mathcal{Z}_q(\mathfrak{g})$ of $U_q(\mathfrak{g})$.

\begin{theorem} \label{center} The center of $U_q(\mathfrak{g})$ is the unital subalgebra generated by the element
$$\Omega = qX + q^{-1}Y + (q-q^{-1})\theta Z.$$
\end{theorem}

\begin{proof}
    The fact that $\Omega$ belongs to the center is proven in \cite{DCDTsl2R}*{Proposition 2.8}. Our goal is to prove that, conversely, every element of the center is a polynomial evaluated at $\Omega$. Let us denote by
    $$\mathcal{G}(U_q(\mathfrak{g})) = \oplus_{r \geq 0}\, \mathcal{G}_r(U_q(\mathfrak{g}))$$
    the graded algebra associated to the filtration of $U_q(\mathfrak{g})$ introduced in Corollary \ref{filtrationUq}. We denote by $\bar{W}^{(n)}$, $\bar Z$, $ \bar \theta$ the respective images of $W^{(n)},Z, \theta$ in $\mathcal{G}_1(U_q(\mathfrak{g}))$. Note that
    $$(\bar{\theta}^{m_1}\bar{W}^{(n)}\bar Z^{m_2})_{m_1,m_2 \in \mathbb{Z}_{\geq 0}}^{n \in \mathbb{Z}}$$
    is a basis of $\mathcal{G}(U_q(\mathfrak{g}))$. Let $\bar\Omega = (q-q^{-1}) \bar \theta \bar Z$ be the image of $\Omega$ in $\mathcal{G}_2(U_q(\mathfrak{g}))$.

    The theorem will be proved if we show that $\mathbb{K}[\bar\Omega]$ contains the graded algebra $\mathcal{G}(\mathcal{Z}_q(\mathfrak{g}))$ associated to $\mathcal{Z}_q(\mathfrak{g})$. Let $z \in \mathcal{G}(\mathcal{Z}_q(\mathfrak{g}))$. We decompose $z$ along the above basis:
    $$z = \sum_{m_1,m_2,n} z_{m_1,m_2,n} \,\bar \theta^{m_1}\bar{W}^{(n)}\bar Z^{m_2}.$$
    The relations (\ref{XYZ}) and (\ref{thetaXYZ}) simplify in the graded algebra, we have in particular
    \begin{gather*}
        \bar Z\bar \theta = \bar \theta \bar Z, \quad \bar Z\bar{W}^{(n)} = q^{2n} \bar{W}^{(n)}\bar Z, \quad \bar \theta\bar{W}^{(n)} = q^{-2n} \bar{W}^{(n)}\bar \theta.
    \end{gather*}
    Using these relations and the fact that $z$ commutes with $\bar \theta$, we get
    $$z_{m_1,m_2,n} = q^{2n} z_{m_1,m_2,n} \qquad (m_1,m_2 \geq 0,\forall n \in \mathbb{Z}).$$
    Thus $z_{m_1,m_2,n} = 0$ if $n \neq 0$. Then, the fact that $z$ commutes with $\bar W^{(1)}$ yields
    $$q^{2m_1} z_{m_1,m_2,0} = q^{2m_2} z_{m_1,m_2,0}\qquad (m_1,m_2 \in \mathbb{Z}_{\geq 0}),$$
    showing that $z$ is a linear combination of the elements $\bar \theta^m \bar Z^m = (q-q^{-1})^{-m}\bar\Omega^m$.
    
\end{proof}

\subsection*{An analogue of the Harish-Chandra isomorphism}

For a split semisimple Lie group $G$ (such as $\mathrm{SL}(2,\mathbb{R})$) with infinitesimal Iwasawa decomposition $\mathfrak{g} = \mathfrak{k} \oplus \mathfrak{a} \oplus \mathfrak{n}$, the composition of the projection $U(\mathfrak{k}) \otimes U(\mathfrak{a}) \otimes U(\mathfrak{n}) = U(\mathfrak{g}) \to U(\mathfrak{a})$ with a certain $\rho$-shift identifies the center of $U(\mathfrak{g})$ with the algebra of Weyl group invariants of $U(\mathfrak{a})$. This is known as the Harish-Chandra isomorphism \cite{HarishChandraZ}, see also \cite{Wallach}*{§3.2}. Here we check that such an isomorphism holds for the center of $U_q(\mathfrak{g})$.

As mentioned in the beginning of last section, $\mathcal{O}_q(K\backslash U)$ is a quantized analogue of the universal enveloping algebra of $\mathfrak{k}^\perp$, which corresponds to the $\mathfrak{a} \oplus \mathfrak{n}$ part of the infinitesimal Iwasawa decomposition. In the classical context, $U(\mathfrak{a})$ is the abelianization of $U(\mathfrak{a}\oplus\mathfrak{n})$ and $U(\mathfrak{n})$ is the kernel of the associated quotient map $U(\mathfrak{a}\oplus\mathfrak{n}) \to U(\mathfrak{a})$. Here, by abelianization we mean the left adjoint of the forgetful functor from the category of unital commutative algebras to that of unital algebras. The abelianization of $\mathcal{O}_q(K\backslash U)$ should therefore be considered as an analogue of $U(\mathfrak{a})$ while the ideal of $\mathcal{O}_q(K\backslash U)$ generated by the commutators could be thought as an analogue of $U(\mathfrak{n})$.

\begin{proposition} \label{ab}
    Let $\mathbb{K}[t,t^{-1}]$ be the ring of Laurent polynomials and let $\mathrm{ab}$ be the unique algebra morphism $\mathcal{O}_q(K\backslash U) \to \mathbb{K}[t,t^{-1}]$ such that
    $$\mathrm{ab}(X) = t, \quad \mathrm{ab}(Y) = t^{-1}, \quad \mathrm{ab}(Z) = 0.$$
    Under that morphism, $\mathbb{K}[t,t^{-1}]$ identifies with the abelianization of $\mathcal{O}_q(K\backslash U)$. The kernel of $\mathrm{ab}$ is the (left, right or two-sided, they are all the same) ideal generated by $Z$.
\end{proposition}

\begin{proof}
    The morphism $\mathrm{ab}$ is well defined because of the description of $\mathcal{O}_q(K\backslash U)$ by generators and relations (see Section 2). The fact that the left, right and two-sided ideals generated by $Z$ are the same is due to the $q$-commuting relations in the first line of (\ref{XYZ}). By Proposition \ref{structureAN}, we have
    $$\mathcal{O}_q(K\backslash U) = Z. \mathcal{O}_q(K\backslash U) \oplus \mathrm{span}(W^{(n)} : n \in \mathbb{Z}).$$
    Moreover, the restriction of $\mathrm{ab}$ to $\mathrm{span}(W^{(n)} : n \in \mathbb{Z})$ is a linear isomorphism, so the kernel of $\mathrm{ab}$ is indeed $Z. \mathcal{O}_q(K\backslash U)$. To check that $\mathrm{ab}$ identifies $\mathbb{K}[t,t^{-1}]$ with the ablianization of $\mathcal{O}_q(K\backslash U)$, it is thus enough to prove that $Z$ lies in the kernel of any commutative quotient of $\mathcal{O}_q(K\backslash U)$. Let $\pi : \mathcal{O}_q(K\backslash U) \to Q$ be such a quotient. We have
    $$q^2\pi(Z^2) = 1-\pi (X)\pi(Y) = 1-\pi (Y)\pi(X) = q^{-2}\pi(Z^2),$$
    so $\pi(Z^2) = 0$ and $\pi(X)\pi(Y) = \pi(Y) \pi(X) = 1$. Then, using $q^2YZ = ZY$, we get
    $$\pi(Z) = \pi (ZYX) = \pi(q^{2}YZX)= q^2\pi(Z),$$
    and we conclude that $\pi(Z) = 0$.
\end{proof}

\begin{theorem} \label{HCiso}
    Let $\zeta_0$ be the unique unital character of $U_q(\mathfrak{k})$ such that $\zeta_0(\theta) = 0$. We define a linear map $\gamma : U_q(\mathfrak{g}) \to \mathbb{K}[t,t^{-1}]$ by the following formula:
    $$\gamma(xy) = \zeta_0(x) \mathrm{ab}(y), \qquad (x \in U_q(\mathfrak{k}), y \in \mathcal{O}_q(K\backslash U)).$$
    Let $\sigma$ be the unique automorphism of $\mathbb{K}[t,t^{-1}]$ such that $\sigma(t) = q^{-1}t$. Then the restriction of $\tilde{\gamma} = \sigma \circ \gamma$ to $\mathcal{Z}_q(\mathfrak{g})$ is an algebra isomorphism onto the algebra of Laurent polynomials that are invariant under the transformation $t \mapsto t^{-1}$.
\end{theorem}

\begin{proof}
    The main step is to prove that $\tilde\gamma$ is multiplicative when restricted to $\mathcal{Z}_q(\mathfrak{g})$. Identifying $U_q(\mathfrak{g})$ with $U_q(\mathfrak{k}) \otimes \mathcal{O}_q(K\backslash U)$ through multiplication, we see that $\gamma = \zeta_0 \otimes \mathrm{ab}$. Let us show that the restriction of $\zeta_0 \otimes \mathrm{id}$ to $\mathcal{Z}_q(\mathfrak{g})$ is multiplicative. Let $z \in \mathcal{Z}_q(\mathfrak{g})$ and $x \in U_q(\mathfrak{g})$. Using that $z$ is central, we get
    $$zx - (\zeta_0\otimes \mathrm{id})(z)(\zeta_0\otimes \mathrm{id})(x)= [x-(\zeta_0\otimes \mathrm{id})(x)]z + [z -(\zeta_0\otimes \mathrm{id})(z)](\zeta_0\otimes \mathrm{id})(x).$$
    Note that the right hand side of the expression lies in $\theta U_q(\mathfrak{g})$, hence
    $$ (\zeta_0\otimes \mathrm{id})(zx)= (\zeta_0\otimes \mathrm{id})(z)(\zeta_0\otimes \mathrm{id})(x).$$
    Composing with the algebra morphisms $\mathrm{id}\otimes \mathrm{ab}$ and $\sigma$, we get that $\tilde{\gamma}$ is an algebra morphism $\mathcal{Z}_q(\mathfrak{g}) \to \mathbb{K}[t,t^{-1}]$.
    
    A short computation yields $\tilde\gamma(\Omega) = t + t^{-1}$. Then, theorem \ref{center} implies that $\tilde\gamma$ is injective on $\mathcal{Z}_q(\mathfrak{g})$ and that $\tilde\gamma[\mathcal{Z}_q(\mathfrak{g})]$ is the algebra of invariant Laurent polynomials.
\end{proof}

\section{Construction of the $\mathbb{A}$-form}

Our aim now is to construct the $\mathbb{A}$-form of the quantized enveloping $\ast$-algebra of $\mathfrak{g}$ mentioned in the introduction, then show that it specializes at $q=1$ to the classical universal enveloping $\ast$-algebra. Our approach will be more general in fact, since we construct quantized versions of this enveloping algebra over a wide class of rings.

\subsection*{From involutive fields to general rings} Let $R$ be any unital commutative ring such that $2 \in R^\times$ and let $q \in R^\times$. We define $\mathrm{U}(q,R)$ as the universal unital $R$-algebra generated by elements $\theta, x, y, z$ subject to the following relations:
\begin{equation}\label{AqR}
    \begin{gathered}
        qx \theta - q^{-1}\theta x = q \theta y - q^{-1}y\theta = [2]_q z - \theta,\qquad 
        \theta z - z \theta = y - x,\\
        qzx - q^{-1}x z = qy z - q^{-1}zy = -z,\\
        xy + q^2z^2 = yx + q^{-2}z^2, \qquad x + y + (q-q^{-1})(xy + q^2z^2) = 0.
    \end{gathered}
\end{equation}
The relevance of this algebra for our purpose consists in the following proposition.
\begin{proposition}\label{specialcases} Here are alternative descriptions of $\mathrm{U}(q,R)$ in two special cases.
    \begin{enumerate}[label = (\roman*)]
        \item Assume that $q =1$. Then $y = - x$ and $\mathrm{span}_R(\theta, x, z)$ is a Lie subalgebra of $\mathrm{U}(1,R)$ isomorphic to $\mathfrak{sl}_2(R)$. In addition, $\mathrm{U}(1,R)$ is the universal enveloping algebra of $\mathrm{span}_R(\theta, x, z)$.
        \item Assume that $R$ is an involutive extension of $\mathbb{C}$ and $q$ has self-adjoint square roots and is not a root of unity. The algebra $\mathrm{U}(q,R)$ identifies with $U_q(\mathfrak {g})$ through the equalities
        \begin{equation} \label{changeVariables}x = \frac{X-1}{q-q^{-1}}, \quad y = \frac{Y-1}{q-q^{-1}}, \quad z = \frac{Z}{q-q^{-1}},\end{equation}
        and the identification of the generator $\theta \in \mathrm{U}(q,R)$ with the identically written element of~ $U_q(\mathfrak{g})$.
    \end{enumerate}
\end{proposition}

The proof of \textit{(\romannumeral 2)} is clear, one just has to check that the relations (\ref{AqR}) are equivalent to the combination of (\ref{XYZ}) and (\ref{thetaXYZ}) under the change of variables (\ref{changeVariables}).

\begin{proof}[Proof of (\romannumeral 1)] The relations (\ref{AqR}) reduce to
\begin{gather*}
     y = - x, \quad
    [x, z] = z, \quad [x , \theta] = 2z - \theta, \quad [z, \theta] = 2 x.
\end{gather*}
It follows that $\mathrm{U}(1,R)$ is the enveloping algebra of its Lie subalgebra $\mathrm{span}_R(x, z, \theta)$. Because of the hypothesis $2 \in R^\times$, we have $\mathrm{span}_R(x, z, \theta) = \mathrm{span}_R(e,f,h)$, where
\begin{equation}\label{sl2triple} e= z, \quad f = \theta - z, \quad h = 2x.\end{equation}
One can easily check that $(e,f,h)$ is an $\mathfrak{sl_2}$-triple, which concludes the proof.
    
\end{proof}

Regarding the general structure of $\mathrm{U}(q,R)$, here is our fundamental result.

\begin{theorem} \label{basis}
    Let $a = y - x$ and $b = -xy - q^2z^2$. The collection of all elements $\theta^mz^na^d b^p$ for $m,n,p \in \mathbb{Z}_{\geq 0}$ and $d \in \{0,1\}$ is an $R$-basis of $\mathrm{U}(q,R)$.
\end{theorem}

The proof will be very similar to that of the Poincaré-Birkhoff-Witt theorem for enveloping algebras. It uses the following lemma in a crucial way.

\begin{lemma} \label{technical}
    Let $V$ a free $R$-module having a basis $(v_{m,n,d,p})_{m,n,d,p}$ indexed by $\mathbb{Z}_{\geq 0}^2 \times \{0,1\} \times \mathbb{Z}_{\geq 0}$. There exists a structure of $\mathrm{U}(q,R)$-module on $V$ such that
    \begin{equation}\label{mndp}(\theta^mz^na^d b^p)v_{0,0,0,0} = v_{m,n,d,p},\qquad ((m,n,d,p) \in \mathbb{Z}_{\geq 0}^2 \times \{0,1\} \times \mathbb{Z}_{\geq 0}).\end{equation}
\end{lemma}

We first prove the theorem assuming this lemma and then prove the latter.

\begin{proof}[Proof of the theorem]
Let $V$ be the $\mathrm{U}(q,R)$-module considered in the lemma. Since the family $(v_{m,n,d,p})_{m,n,d,p}$ is linearly independent, so is $(\theta^mz^na^d b^p)_{m,n,d,p}$ because of the property (\ref{mndp}).

It remains to show that the linear span $S$ of all the elements $\theta^mz^na^d b^p$ is equal to $\mathrm{U}(q,R)$. For that, it is enough to check that $S$ is a left ideal of $\mathrm{U}(q,R)$. For every $r \in \mathbb{Z}_{\geq 0}$, let $$S_r = \mathrm{span}\,\{ \theta^mz^na^d b^p : m + n + d +p \leq r\}.$$
We prove by induction on $r$ that the multiplication on the left of an element of $S_r$ by one of the generators $\theta, x, y, z$ lies in $S_{r+1}$. For $r = 0$, this property is easily checked from the following fact:
\begin{equation}\label{xyab} x = 2^{-1}(-a + (q-q^{-1})b), \quad y = 2^{-1}(a + (q-q^{-1})b).\end{equation}
Now, let us assume that the property is true for some $r \in \mathbb{Z}_{\geq 0}$. We have automatically $\theta S_{r+1} \subset S_{r+2}$. Let us then check it for $z, y, x$ in this order. Let $m,n,d,p$ be indices such that $ m + n +d + p \leq r+1$.

If $m = 0$ then we have $\bar {z}(\theta^mz^na^d b^p) = z^{n+1}a^d b^p \in S_{r+2}$. If $m \neq 0$ then
\begin{align*}
    \bar {z}(\theta^mz^na^d b^p) = z \theta  (\theta^{m-1}z^na^d b^p) = (\theta z + x - y)\theta^{m-1}z^na^d b^p.
\end{align*}
By induction hypothesis, $(x - y)\theta^{m-1}z^na^d b^p$ and $z\theta^{m-1}z^na^d b^p$ both belong to $S_{r+1}$. Having checked before that $\theta S_{r+1} \subset S_{r+2}$, we get that
$$(\theta z + x - y)\theta^{m-1}z^na^d b^p \in S_{r+2},$$
which proves that $\bar {z} S_{r+1} \subset S_{r+2}$.

We continue with $y $. Regarding (\ref{xyab}), we have $y (\theta^mz^na^d b^p) \in S_{r+2}$ when $m,n,d$ are all equal to $0$. If $m = n = 0$ and $d = 1$, then
\begin{align*}
    y (\theta^mz^na^d b^p) &= 2^{-1}a^2b^p + (q-q^{-1})bab^p.
\end{align*}
Now the equalities in the third line of (\ref{AqR}) are equivalent to the combination of the two following ones:
\begin{gather*}
    (q-q^{-1})^2b^2 - a^2 + 2\{2\}_qz^2 = - 4b,\\
    (q-q^{-1})(ab-ba) = 2(q^2-q^{-2})z^2.
\end{gather*}
Thus we have
\begin{gather*}
    y (\theta^mz^na^d b^p) = [2b + \{2\}_qz^2 + 2^{-1}(q-q^{-1})^2b^2+(q-q^{-1})ab - (q^2-q^{-2})z^2]b^p.
\end{gather*}
Since $p \leq r$, this element belongs to $S_{r+2}$. If $m = 0$ but $n \neq 0$ then
\begin{gather*}y(\theta^mz^na^d b^p) = y z(z^{n-1}a^d b^p) = (q^{-2}z y - q^{-1}z)(z^{n-1}a^d b^p),\end{gather*}
with $q^{-1}z(z^{n-1}a^d b^p) \in S_{r+1}$.
By induction hypothesis, $y (z^{n-1}a^d b^p) \in S_{r+1}$ and since we proved above that $z S_{r+1} \subset S_{r+2}$ we have $y(\theta^mz^na^d b^p) \in S_{r+2}$. It remains the case $n \neq 0$:
$$y (\theta^mz^na^d b^p) = y \theta (\theta^{m-1}z^na^d b^p) = [q^2\theta y +q(\theta- [2]_qz)](\theta^{m-1}z^na^d b^p).$$
The same type of argument as in the previous case shows that $y(\theta^mz^na^d b^p) \in S_{r+2}$. To sum up, we just proved that $y S_{r+1} \subset S_{r+2}$. The elements $x$ and $y$ satisfy the same kind of relations, see (\ref{AqR}), so the arguments used for the proof of $yS_{r+1} \subset S_{r+2}$ also work for $x$.

In conclusion, since $x $, $y$, $z$ and $\theta$ generate $\mathrm{U}(q,R)$ as an algebra, we have indeed $\mathrm{U}(q,R)S \subset S$ and thus $S = \mathrm{U}(q,R)$.
\end{proof}

Now we give a proof of lemma \ref{technical}.

\begin{proof}[Proof of the lemma]
For any $r \in \mathbb{Z}_{\geq 0}$, let $$V_r = \mathrm{span}\,\{ v_{m,n,d,p} : m + n + d +p \leq r\}.$$
We recursively define linear actions of $\theta , x, y, z$ on $V$ mapping $V_r$ to $V_{r+1}$ for all $r$. For any $m,n,d,p$, the action of $\theta$ on $v_{m,n,d,p}$ is given by
$$\theta v_{m,n,d,p} = v_{m+1,n,d,p}.$$
Then, we define
$$z v_{m,n,d,p} = \begin{cases}
    \theta z v_{m-1,n,d,p} + x v_{m-1,n,d,p} - y v_{m-1,n,d,p} \quad \text{if }m\neq 0\\
    v_{0,n+1,d,p} \quad \text{if }m=0.
\end{cases}$$
The action of $x$ and $y$ are defined by
\begin{gather*}x v_{m,n,d,p} = \begin{cases}
    q^{-2}\theta x v_{m-1,n,d,p}+ q^{-1}[2]_qzv_{m-1,n,d,p} - q^{-1}\theta v_{m-1,n,d,p} \quad \text{if }m\neq 0\\
    q^2z x v_{0,n-1,d,p}+ qzv_{0,n-1,d,p} \quad \text{if }m=0, n\neq 0\\
    2^{-1}(q-q^{-1})v_{0,0,0,p+1} - 2^{-1}v_{0,0,1,p} \quad \text{if }m=n=d=0,
    \end{cases}\\
y v_{m,n,d,p} = \begin{cases}
    q^{2}\theta y v_{m-1,n,d,p}-q[2]_qzv_{m-1,n,d,p} + q\theta v_{m-1,n,d,p} \quad \text{if }m\neq 0\\
    q^{-2}z y v_{0,n-1,d,p}- q^{-1}zv_{0,n-1,d,p} \quad \text{if }m=0, n\neq 0\\
    2^{-1}(q-q^{-1})v_{0,0,0,p+1} + 2^{-1}v_{0,0,1,p} \quad \text{if }m=n=d=0,
    \end{cases}\end{gather*}
and\begin{gather*}
(x + y)v_{0,0,1,p} = v_{0,0,1,p+1} - 2(q^2-q^{-2})v_{0,0,0,p},\\
(y -x)v_{0,0,1,p} = (q-q^{-1})^2v_{0,0,0,p+2} + 4 v_{0,0,0,p+1} + 2\{2\}_qz^2v_{0,0,0,p}.\end{gather*}

Then one can prove by induction that the defining relations (\ref{AqR}) of $\mathrm{U}(q,R)$ are satisfied for these actions. More precisely, one can check by induction that for each $r\geq 0$, the relations (\ref{AqR}) hold when applied to an element of $V_r$.

Let us assume that the relations (\ref{AqR}) hold on $V_{r-1}$. Then the first line of relations applied on elements of $V_r$ follows from the very definition of the actions of $\theta, x, y, z$ (see their expression on $v_{m,n,d,p}$ in the case $m\neq 0$). The other ones need more computations but they all rely on the same kind of techniques (which are similar to the ones used in the proof of the theorem). For that reason we only detail the proof that $qyz - q^{-1}z y = -z$ on $V_r$.

Let $m,n,d,p$ be indices such that $m+n+d+p \leq r$. Assume first that $m \neq 0$. We have $v_{m,n,d,p} = \theta v_{m-1,n,d,p}$ and by using the exchange relations involving $\theta$:
\begin{align*}
    q^{-1}zyv_{m,n,d,p} &= z (q^{-1}y \theta) v_{m-1,n,d,p}\\
    &= z (q\theta y-[2]_qz+\theta)v_{m-1,n,d,p}\\
    &= (z \theta )(qy - \mathrm{id}_V)v_{m-1,n,d,p} - [2]_qz^2v_{m-1,n,d,p}\\
    & = (\theta z + x - y)(qy - \mathrm{id}_V)v_{m-1,n,d,p} - [2]_qz^2v_{m-1,n,d,p},
\end{align*}
\begin{align*}
    q y z v_{m,n,d,p} &= qy(z \theta) v_{m-1,n,d,p}\\
    & = qy(\theta z + x - y)v_{m-1,n,d,p}\\
    & = q^2(q^{-1} y\theta) z v_{m-1,n,d,p}+ q y (x - y)v_{m-1,n,d,p}\\
    & = (q^3\theta y- q^2[2]_qz +q^2\theta)zv_{m-1,n,d,p}+ q y (x - y)v_{m-1,n,d,p}.
\end{align*}
Hence we get
\begin{align*}
    (q y z &- q^{-1}zy)v_{m,n,d,p} =\\ &[q^2\theta (qy z - q^{-1}z y)+y-x + (q^2-1)(\theta z - [2]_q z^2) + q(yx - xy)]v_{m-1,n,d,p}.
\end{align*}
Using the recursion hypothesis, we have
\begin{gather*}(qy z - q^{-1}z y)v_{m-1,n,d,p} = -z v_{m-1,n,d,p},\\
(yx - xy)v_{m-1,n,d,p} = (q^2-q^{-2})z^2v_{m-1,n,d,p}.
\end{gather*}
Gathering everything we get
$$(q y z - q^{-1}zy)v_{m,n,d,p} = (y - x - \theta z)v_{m-1,n,d,p} = - z \theta v_{m-1,n,d,p} = -z v_{m,n,d,p}.$$

Now, if we assume that $m=0$, we have by definition
$$qy z v_{0,n,d,p} = q y v_{0,n+1,d,p} = q^{-1}z y v_{0,n,d,p}-zv_{0,n,d,p}.$$

Thus, the relation $qyz - q^{-1}z y = -z$ is indeed satisfied on $V_r$. As explained before, the other relations can be checked along the same lines. One just have to treat them in the same order as they appear in (\ref{AqR}).

To sum up, the actions of $\theta, x, y, z$ on $V$ extend to a structure of $\mathrm{U}(q,R)$-module on $V$. The fact that $\theta^m z^n a^d b^p  v_{0,0,0,0} = v_{m,n,d,p}$ for all possible indices $m,n,d,p$ ($d \in \{0,1\})$ is quite straightforward from the explicit actions of $\theta, x, y , z$ on $V$.

\end{proof}

One of the interesting features of our construction is that it is functorial. Namely, if $\varphi : R \to S$ is a ring homomorphism then the universal property of $\mathrm{U}(q,R)$ implies the existence of a unique ring homomorphism $\mathrm{U}(\varphi) : \mathrm{U}(q,R) \to \mathrm{U}(\varphi(q),S)$ such that
$$A(\varphi)_{|R} = \varphi, \quad A(\varphi)(x) = x, \quad A(\varphi)(y) = y,\quad  A(\varphi)(z) = z, \quad A(\varphi)(\theta) = \theta.$$
Moreover, since they share the same universal property, the algebras $S \otimes_\varphi \mathrm{U}(q,R)$ and $\mathrm{U}(\varphi(q),S)$ are naturally isomorphic.

An important consequence of Theorem \ref{basis} is the following one:

\begin{corollary} \label{changering}
    If $R \subset S$ is a ring extension, then the associated ring morphism $\mathrm{U}(q,R) \to \mathrm{U}(q,S)$ is injective, which means that $\mathrm{U}(q,R)$ identifies with the $R$-subalgebra of $\mathrm{U}(q,S)$ generated by $\theta, x , y, z$.
\end{corollary}

\subsection*{The $\mathbb{A}$-form and its specializations}

Let $\mathbb{A}$ be the ring of complex-valued real analytic functions on $\mathbb{R}_+^*$ and $\mathbb{F}$ its field of fractions. The pointwise complex conjugation defines a ring involution on $\mathbb{A}$. This involution uniquely extends to an automoprhism of $\mathbb{F}$, so that the latter becomes an involutive extension of $\mathbb{C}$. Let us denote by $\boldsymbol{q}$ the identity function on $\mathbb{R}^*_+$. It is an invertible element of $\mathbb{A}$, it is not a root of unity and it admits self-adjoint square roots.

\begin{definition} \label{defUA}
    Let us denote by $U_\mathbb{F}(\mathfrak{g})$ the deformed enveloping $\ast$-algebra $U_{\boldsymbol{q}}(\mathfrak{g})$ defined over $\mathbb{F}$. The $\mathbb{A}$-form $U_\mathbb{A}(\mathfrak{g})$ of $U_\mathbb{F}(\mathfrak{g})$ is defined as the unital $\mathbb{A}$-subalgebra of $U_\mathbb{F}(\mathfrak{g})$ generated by the elements
    $$\theta, \quad x = \frac{X-1}{\boldsymbol{q}-\boldsymbol{q}^{-1}}, \quad y = \frac{Y-1}{\boldsymbol{q}-\boldsymbol{q}^{-1}}, \quad z = \frac{Z}{\boldsymbol{q}-\boldsymbol{q}^{-1}}.$$
\end{definition}

By Proposition \ref{specialcases}, $U_\mathbb{F}(\mathfrak{g})$ identifies with $\mathrm{U}(\boldsymbol{q},\mathbb{F})$. Then, Corollary \ref{changering} indicates that $U_\mathbb{A}(\mathfrak{g})$ identifies with $\mathrm{U}(\boldsymbol{q},\mathbb{A})$. The discussion before Corollary \ref{changering} combined with Proposition \ref{specialcases} leads to the following fact.

\begin{proposition} For every $q \in \mathbb{R}_+^*$, let us denote by $\mathrm{ev}_q : \mathbb{A} \to \mathbb{C}$ the evaluation at $q$. For each such $q$, the specialization of $U_\mathbb{A}(\mathfrak{g})$ at $q$, that is $\mathbb{C} \otimes _{\mathrm{ev}_q} U_\mathbb{A}(\mathfrak{g})$, is isomorphic to $U(\mathfrak{sl_2}(\mathbb{C}))$ if $q = 1$ and it identifies with $U_q(\mathfrak{g})$ via (\ref{changeVariables}) otherwise.
\end{proposition}

It is clear that $U_\mathbb{A}(\mathfrak{g})$ is a $\ast$-subring of $U_\mathbb{F}(\mathfrak{g})$. Since  $\mathrm{ev}_q : \mathbb{A} \to \mathbb{C}$ is a morphism of $\ast$-rings, the specialization of $U_\mathbb{A}(\mathfrak{g})$ at each $q \in \mathbb{R}_+^*$ inherits a structure of $\ast$-algebra over $\mathbb{C}$. If $q \neq 1$, the latter is readily seen to coincide with the $\ast$-structure of $U_q(\mathfrak{g})$. What about the specialization at $q=1$ ?

Let $U(\mathfrak{g})$ be the enveloping $\ast$-algebra of $\mathfrak{g} = \mathfrak{sl}(2,\mathbb{R})$. We denote by $(E, F, H)$ the canonical basis of $\mathfrak{sl}(2,\mathbb{R})$. The $\ast$-structure of $\mathbb{C} \otimes _{\mathrm{ev}_1} U_\mathbb{A}(\mathfrak{g}) = \mathrm{U}(1,\mathbb{C})$ is characterized by:
$\theta^* = \theta, \, x^* = y = -x, \, z^* = z.$ Now one can check that if we set
\begin{equation}\label{specialization1} \theta = i(E-F), \quad x = \frac{H}{2}, \quad z = iE,\end{equation}
then $\mathbb{C} \otimes _{\mathrm{ev}_1} U_\mathbb{A}(\mathfrak{g})$ identifies with $U(\mathfrak{g})$ as a $\ast$-algebra.

\section{Elementary aspects of the representation theory}

We now turn to the representation theory of the $q$-deformed convolution algebras introduced in Section 2. The goal of the present section is to explain the relationships between several representation categories of these algebras. Ultimately, we will end up with several equivalent descriptions of the $q$-analogues of the $(\mathfrak{g},K)$-modules and of the $q$-analogues of the unitary representations of $G$. At the end of the section, we study a natural analogue of admissibility for these representations.

For the two following sections, we work again over a fixed involutive extension $\mathbb{K}$ of $\mathbb{C}$ and we consider an element $q \in \mathbb{K}^\times$ which is not a root of unity and has self-adjoint square-roots. The generators $X,Y,Z, \theta$ of $U_q(\mathfrak{g})$ are defined with respect to a fixed square root $q^{1/2}$ of $q$.

\subsection*{The $q$-analogues of $(\mathfrak{g},K)$-modules}  The classical $(\mathfrak{g},K)$-modules can be equally described as non-degenerate $R(\mathfrak{g},K)$-modules or as $U(\mathfrak{g})$-modules satisfying certain integrality conditions along $K$. Such an equivalence also holds in our context.

As $R_q(\mathfrak{g},K)$ is an approximately unital $\mathbb{K}$-algebra, one can consider its category of non-degenerate modules $\mathsf{Mod}_{R_q(\mathfrak{g},K)}$. Recall from Remark \ref{completionR} that $\mathcal{O}'_q(G)$ is a dense subring of the completion of $R_q(\mathfrak{g},K)$. Consequently, there is category equivalence between $\mathsf{Mod}_{R_q(\mathfrak{g},K)}$ and the category of continuous $\mathcal{O}'_q(G)$-modules (see the paragraph just after Definition \ref{defNd}).

\begin{definition}
    A $U_q(\mathfrak{g})$-module is said to be integral if the action of $\theta$ on it is diagonalizable with eigenvalues of the form $[n]_q$ for $n \in \mathbb{Z}$.
\end{definition}

Recall from Section 2 that $\mathcal{O}_q'(K)$ is the closure of $U_q(\mathfrak{g})$ in $\mathcal{O}'_q(G)$ (here we use again the content of Remark \ref{completionR}). Since $\mathcal{O}'_q(G)$ is generated as an algebra by $\mathcal{O}_q'(K)$ and $\mathcal{O}_q(K\backslash U)$, we see that $U_q(\mathfrak{g})$ is dense in $\mathcal{O}'_q(G)$. If $V$ is a given integral $U_q(\mathfrak{g})$-module then the map of the action $U_q(\mathfrak{g}) \times V \to V$ is continuous, where $V$ is endowed with the discrete topology. By completion, $V$ can be equipped with a structure of continuous $\mathcal{O}'_q(G)$-module. Conversely, a continuous $\mathcal{O}'_q(G)$-module is automatically an integral $U_q(\mathfrak{g})$ by restriction. Moreover, in this context, because of the continuity of the action, a $\mathbb{K}$-linear map intertwining the action of $U_q(\mathfrak{g})$ automatically intertwines the action of $\mathcal{O}'_q(G)$. Thus we get a category equivalence between integral $U_q(\mathfrak{g})$-modules and continuous $\mathcal{O}'_q(G)$-modules.

In conclusion, non-degenerate $R_q(\mathfrak{g},K)$-modules, integral $U_q(\mathfrak{g})$-modules and continuous $\mathcal{O}'_q(G)$-modules are essentially the same objects; in addition the related notions of intertwining maps all coincide.

\begin{convention} 
    By a $(\mathfrak{g},K)_q$-module, we mean a vector space over $\mathbb{K}$ equipped with one of these three equivalent structures of module. The intertwining maps between these objects will be referred to as morphisms of $(\mathfrak{g},K)_q$-modules or $(\mathfrak{g},K)_q$-intertwiners.
\end{convention}

Let $V$ be any $(\mathfrak{g},K)_q$-module. Recall from Section 2 that the $e_n$ are spectral projections of $\theta$, so that $e_n V$ is just the $[n]_q$-eigenspace of the action of $\theta$ on $V$ for all $n \in \mathbb{Z}$. Hence $V$ is the direct sum of its subspaces $e_n V$ as $n$ ranges over $\mathbb{Z}$. This corresponds to the isotypic decomposition of $V$ when the latter is considered as a $R_q(K)$-module (or equivalently, as an $U_q(\mathfrak{k})$-module). Accordingly, the subspaces $e_n V$ will be referred to as the $K_q$-isotypical components of $V$. Moreover, an integer $n \in \mathbb{Z}$ such that $e_n V \neq 0$ will be called a $K_q$-type of $V$.

Finally, let us define $(\mathfrak{g},K)_q$-submodules as subspaces that are stable by the action of one (hence all) of the algebras $R_q(\mathfrak{g},K)$, $U_q(\mathfrak{g})$, $\mathcal{O}'_q(G)$. A $(\mathfrak{g},K)$-module will be called simple if $\{0\}$ is its only proper $(\mathfrak{g},K)_q$-submodule. The set of isomorphism classes of simple $(\mathfrak{g},K)_q$-modules will be denoted by $\mathrm{Irr}_q(\mathfrak{g},K)$.

\subsection*{Unitarity} For this part, let $\mathbb{K} = \mathbb{C}$. Our goal is to discuss the link between Hilbert representations of the $q$-deformed group C*-algebra $C_q^*(G)$ and certain $(\mathfrak{g},K)_q$-modules.

An inner product $(\cdot |\cdot)$ on a $(\mathfrak{g},K)_q$-module $V$ will be called invariant if it satisfies $(x^*v | w) = (v | xw) $ for all $x \in R_q(\mathfrak{g},K)$ and $(v,w) \in V^2$. Note that this condition is equivalent when expressed for all elements of $U_q(\mathfrak{g})$ or $\mathcal{O}'_q(G)$ instead of $R_q(\mathfrak{g},K)$. In the following, such an inner product will always be denoted by $(\cdot |\cdot)$.

\begin{definition}
    A unitary $(\mathfrak{g},K)_q$-module is a $(\mathfrak{g},K)_q$-module which is equipped with an invariant inner product and whose $K_q$-isotypical components are Hilbert spaces. A $(\mathfrak{g},K)_q$-module will be called unitarizable if it admits an invariant inner product which turns it into a unitary $(\mathfrak{g},K)_q$-module.

    A non-zero unitary $(\mathfrak{g},K)_q$-module is said to be irreducible if it has not any proper non-zero unitary $(\mathfrak{g},K)_q$-submodule. Two unitary $(\mathfrak{g},K)_q$-modules are said to be unitarily equivalent if there exists an isometric $(\mathfrak{g},K)_q$-intertwiner between them.
\end{definition}

As explained in \cite{DCquantisation}*{Theorem 3.3} in a more general context, there is a natural one-one correspondence between unitary $(\mathfrak{g},K)_q$-modules and Hilbert representations of $C_q^*(G)$. Concretely, if $W$ is a Hilbert representation of $C_q^*(G)$ then $W_0 = R_q(\mathfrak{g},K)W = \oplus_{n\in\mathbb{Z}}e_n W$ is a unitary $(\mathfrak{g},K)_q$-module. Conversely, if $V$ is any unitary $(\mathfrak{g},K)_q$-module then its Hilbert completion $\overline{V}$ is a representation of $C^*_q(G)$ in a natural way. This is due to the strong uniform C*-boundedness of $R_q(\mathfrak{g},K)$, see the end of Section 2. These two operations are inverse to each other and they are functorial, in the sense that intertwining isometries between two unitary $(\mathfrak{g},K)_q$-modules naturally correspond to $C^*_q(G)$-intertwining isometries between their completions. In particular, this establishes a correspondence between irreducible Hilbert representations of $C^*_q(G)$ and irreducible unitary $(\mathfrak{g},K)_q$-modules. From the foregoing, we derive the following fact.

\begin{proposition} \label{unitarydual} The above correspondence induces a bijection between the spectrum of $C^*_q(G)$ and the set of unitary equivalence classes of irreducible unitary $(\mathfrak{g},K)_q$-modules. \end{proposition}

We will see later that the irreducible unitary $(\mathfrak{g},K)_q$-modules are simple and that the spectrum of $C^*_q(G)$ embeds into $\mathrm{Irr}_q(\mathfrak{g},K)$. The fundamental notion coming into play at this level is admissibility.

\subsection*{Admissibility} In the representation theory of general semisimple Lie groups, there are basic results asserting that simple $(\mathfrak{g},K)$-modules and irreducible unitary representations are admissible, see \cite{Wallach}*{§3.4.8 and §3.4.10}. Here we adapt them to our specific context. Let us come back to the case where $\mathbb{K}$ is a general involutive extension of $\mathbb{C}$.

\begin{definition}
    A $(\mathfrak{g},K)_q$-module is said to be admissible if its $K_q$-isotypical components are finite dimensional over $\mathbb{K}$.
\end{definition}

Here is our main admissibility criterion.
\begin{definition}
    Recall from Theorem \ref{center} the definition of the Casimir element $\Omega$.
    A $(\mathfrak{g},K)_q$-module $V$ is said to be basic if $V = U_q(\mathfrak{g})v$ for some $v \in V$ such that both $\Omega v$ and $\theta v$ are proportional to $v$.
\end{definition}

\begin{lemma} \label{admissibilitycriterion}
     If $V$ is a basic $(\mathfrak{g},K)_q$-module then its $K_q$-isotypical components are at most one-dimensional and its $K_q$-types all have the same parity. In particular $V$ is admissible.
\end{lemma}

\begin{corollary}
    Let $V$ be a $(\mathfrak{g},K)_q$-module which is finitely generated over $U_q(\mathfrak{g})$ and on which the action of $\Omega$ is semisimple. Then $V$ is admissible.
\end{corollary}

The proof of the lemma for $\mathbb{K} = \mathbb{C}$ is essentially given by \cite{DCDTsl2R}*{Proposition 3.8}. The case of a general involutive extension of $\mathbb{C}$ need only very minor adaptations that we will nonetheless present carefully, especially because the material needed for this proof will be used again later.

We begin with the definition of certain transitions operators which are elements of $\mathcal{O}_q(K\backslash U)$. Let us consider a 2-dimensional irreducible representation of $U_q(\mathfrak{u})$ having a basis $(\xi_{\mu,\nu})^{\mu,\nu \geq 0}_{\mu + \nu = 2}$ such that
$$E\xi_{\mu,\nu} = [\nu]_q \xi_{\mu+1,\nu-1}, \quad F \xi_{\mu,\nu} = [\mu]_q \xi_{\mu-1,\nu +1},\quad k \xi_{\mu,\nu} = q^{\mu - \nu} \xi_{\mu,\nu}.$$
For every $n \in \mathbb{Z}$, we introduce
\begin{equation}\label{w_n} \begin{gathered} w_n^\pm = q^{\pm n}\xi_{2,0} - q^{\mp n}\xi_{0,2} \mp iq^{1/2}[2]_q \xi_{1,1},\\ \quad w_n= q^{-1}\xi_{2,0} +q \xi_{0,2} + iq^{1/2}(q^n-q^{-n}) \xi_{1,1}.\end{gathered}\end{equation}
One can check that $w_n^{\pm}$ and $w_n$ are eigeinvectors of $\theta + [n]_qk$ with respective eigenvalues $[n\pm 2]_q$ and $[n]_q$. Let $(\xi_{\mu,\nu}^*)_{\mu,\nu}$ be the dual basis of $(\xi_{\mu,\nu})_{\mu,\nu}$. We define matrix coefficients $T_n^{\pm}$, $T_n$ of $U_q(\mathfrak{u})$ as follows:
\begin{equation}\label{transition} \qquad \begin{array}{ll}
    \langle T_n^\pm, x\rangle = (\xi_{2,0}^* + \xi_{0,2}^*, x w_n^\pm), \\
    \langle T_n, x\rangle = (\xi_{2,0}^* + \xi_{0,2}^*, x w_n),
\end{array} \quad \qquad (x \in U_q(\mathfrak{u})).\end{equation}
A short computation shows that $(\xi_{2,0}^* + \xi_{0,2}^*, w_0^\pm) = 0$, hence $\xi_{2,0}^* + \xi_{0,2}^*$ is in the kernel of the transpose of the action of $\theta$. We deduce that $T_n^\pm \triangleleft \theta = T_n \triangleleft \theta = 0$, so $T_n^\pm$ and $T_n$ are elements of $\mathcal{O}_q(K\backslash U)$. Moreover, a computation very similar to the one of \cite{DCDTsl2R}*{Lemma 3.3} leads to the following formulas in the algebra $U_q(\mathfrak{g})$:
\begin{equation}\label{transitions}\begin{gathered}(\theta - [n\pm2]_q)T_n^\pm = (k \triangleright T_n^\pm) (\theta - [n]_q),\\(\theta - [n]_q)T_n = (k \triangleright T_n) (\theta - [n]_q).\end{gathered}\end{equation}
Another important fact is that $\mathrm{span}(X,Y,Z) = \mathrm{span}(T_n, T_n^+, T_n^-)$ with the following explicit formulas for the change of variables:
\begin{equation}\label{TXYZ} T_n^\pm = q^{\pm n} X - q^{\mp n}Y \mp [2]_q Z, \qquad T_n = q^{-1}X + qY + (q^n - q^{-n}) Z.\end{equation}
Finally, using (\ref{XYZ}) and (\ref{thetaXYZ}), one can check that for every $n \in \mathbb{Z}$, we have
\begin{gather}\label{transitions2}T_{n+2}^-T_n^+ = T_n^2 - \{n+1\}_q^2, \qquad T_{n-2}^+T_n^- = T_n^2 - \{n-1\}_q^2, \\
\label{transitions3} \Omega - T_n = (q-q^{-1})Z(\theta - [n]_q).\end{gather}
Thanks to this material, we can now give the proof of Lemma \ref{admissibilitycriterion}.

\begin{proof}[Proof of Lemma \ref{admissibilitycriterion}] 
    Let $\omega \in \mathbb{K}$ and $n \in \mathbb{Z}$ such that $\Omega v = \omega v$ and $\theta v = [n]_q v$. Let us define a family $(v_\mu)_{\mu \in \mathbb{Z}}$ of elements of $V$ as follows:
    $$v_\mu = \begin{cases}
        v \quad  \text{if $\mu = 0$}\\
        T_{n+2(\mu-1)}^+v_{\mu-1} \quad \text{if }\mu >0 \\
        T_{n+2(\mu+1)}^-v_{\mu+1} \quad \text{if }\mu <0.
    \end{cases}$$
    From (\ref{transitions}), it follows that $v_\mu$ is an eigenvector of $\theta$ with eigenvalue $[n+2\mu]_q$, for every $\mu \in \mathbb{Z}$. Using (\ref{transitions3}) and the centrality of $\Omega$, we get
    $$\Omega v_\mu = T_{n+2\mu} v_\mu = \omega v_\mu \qquad (\mu \in \mathbb{Z}).$$
    Then, thanks to (\ref{transitions2}), we have for every $\mu \in \mathbb{Z}$,
    $$T_{n+2\mu}^+v_\mu \in \mathbb{K}v_{\mu+1}, \qquad T_{n+2\mu}^-v_\mu \in \mathbb{K}v_{\mu-1}.$$
    Since $\mathrm{span}(X,Y,Z) = \mathrm{span}(T_{n+2\mu}^+, T_{n+2\mu}^-, T_{n+2\mu})$, we conclude that the span of $\{v_\mu : \mu \in \mathbb{Z}\}$ is a $U_q(\mathfrak{g})$-submodule of $V$ and hence is equal to $V$, which proves the lemma.
\end{proof}

\begin{remark} \label{noneed}
    Let $V$ be a $U_q(\mathfrak{g})$-module and let $v \in V$. Assume that $\Omega v = \omega v$ and $\theta v = [n]_q$ for some pair $(\omega,n) \in \mathbb{K} \times \mathbb{Z}$. If $V = U_q(\mathfrak{g})v$ then the above argument still applies and we get that $V$ is integral. It is then a basic $(\mathfrak{g},K)_q$-module.
\end{remark}

\begin{proof}[Proof of the corollary]
    Let $v_1,..., v_n$ be elements of $V$ that are eigenvectors of $\theta$ and such that
    $V = \sum_{j = 1}^n U_q(\mathfrak{g})v_j$. Since $\Omega$ acts semisimply on $V$ there exists a polynomial $P$ over $\mathbb{K}$ which is a product of distinct irreducible monic polynomials and such that $P(\Omega)v_j = 0$ for all $j$. Then, by centrality of $\Omega$, one has $P(\Omega)V = 0$.
    
    Let $\bar P$ be the conjugate of $P$ with respect to the involution of $\mathbb{K}$. Since $\bar P$ is also a product of distinct irreducible polynomials, so is the lowest common multiple of $\bar P$ and $P$. Hence, without loss of generality, one can assume that $\bar P = P$. Let $L \supset \mathbb{K}$ be a splitting field of $P$. The involution on $\mathbb{K}$ can be extended to $L$ and $L \otimes_\mathbb{K} U_q(\mathfrak{g})$ identifies with the quantized enveloping $\ast$-algebra of $\mathfrak{g}$ over $L$ with parameter $q$. Given that $L$ is a finite extension of $\mathbb{K}$, the admissibility of $L \otimes_\mathbb{K} V$ over $L \otimes_\mathbb{K} U_q(\mathfrak{g})$ implies the admissibility of $V$ over $U_q(\mathfrak{g})$. To sum up, without loss of generality, one can assume that $P$ splits into linear factors.
    
    Then, the action of $\Omega$ on $V$ is diagonalizable with a finite number of eigenvalues. Since $\Omega$ preserves that $K_q$-isotypical components of $V$, each $v_j$ can be expressed as a sum of elements of $V$ that are both eigenvectors of $\Omega$ and $\theta$. Consequently, $V$ is a finite sum of basic $(\mathfrak{g},K)_q$-submodules and we conclude that $V$ is admissible thanks to Lemma \ref{admissibilitycriterion}.
\end{proof}

Let us now analyse the consequences of these results in the case $\mathbb{K} = \mathbb{C}$. A direct application of Lemma \ref{admissibilitycriterion} is that every simple $(\mathfrak{g},K)_q$-module is admissible with at most one-dimensional $K_q$-isotypical components. Indeed, since $\Omega$ is central in $U_q(\mathfrak{g})$ and the latter has countable dimension, Schur's lemma \cite{Bourbaki}*{§3, N°2} implies that $\Omega$ acts as a scalar on any simple $U_q(\mathfrak{g})$-module. Thus, any simple $(\mathfrak{g},K)_q$-module is basic. Such a result also holds for irreducible unitary $(\mathfrak{g},K)_q$-modules.

\begin{proposition} \label{ilcuequivalence}
    Any irreducible unitary $(\mathfrak{g},K)_q$-module is basic and simple, in particular it is also admissible. Two irreducible unitary $(\mathfrak{g},K)_q$-modules are unitarily equivalent if and only if they are isomorphic as $(\mathfrak{g},K)_q$-modules. Moreover, a simple $(\mathfrak{g},K)_q$-module is unitarizable if and only if there exists an invariant inner product on it.
\end{proposition}

\begin{proof}
    Let $V$ be an irreducible unitary $(\mathfrak{g},K)_q$-module. Let us prove that the Casimir element $\Omega$ acts on $V$ as a scalar. Let $\bar V$ be the Hilbert completion of $V$. It is an irreducible representation of $C^*_q(G)$, see Proposition \ref{unitarydual}. Hence the algebra of bounded operators on $\bar V$ representing elements of $C^*_q(G)$ is strongly dense in the space of all bounded operators on $\bar V$. The same goes for $R_q(\mathfrak{g},K)$ as it is norm-dense in $C^*_q(G)$. Let us assume that there exists $v \in V$ such that $\Omega v$ is not proportional to $v$. Then there exists a bounded operator on $\bar V$ mapping both $\Omega v$ and $v$ to $v$. Using the above strong density property, one can find a sequence $(x_n)_{n \in \mathbb{N}}$ of elements of $R_q(\mathfrak{g},K)$ such that
    $$\lim_{n \to + \infty} x_n \Omega v  = \lim_{n \to + \infty} x_n v = v.$$
    Using that $\Omega$ is self-adjoint and commutes with the elements of $R_q(\mathfrak{g},K)$, we have for every $w \in V$
    \begin{align*}
        ( v|w) = \lim_{n \to + \infty} (x_n \Omega v| w) = \lim_{n \to + \infty} ( \Omega x_n v| w) = \lim_{n \to + \infty} (x_n v| \Omega w)  = (v| \Omega w) = (\Omega v| w).
    \end{align*}
    Consequently we have $\Omega v = v$, which contradicts our initial hypothesis. Hence $\Omega v$ is proportional to $v$ for all $v \in V$ and we conclude that $\Omega$ must act as a scalar on~$V$.
    
    Let $v \in V$ be an eigenvector of $\theta$. The $(\mathfrak{g},K)_q$-submodule of $V$ generated by $v$ is basic by Lemma \ref{admissibilitycriterion} and in particular is admissible, which means that it is a unitary $(\mathfrak{g},K)_q$-submodule of $V$. By irreducibility of the latter, we must have $V = U_q(\mathfrak{g})v$. Hence $V$ is basic. Because of its admissibility, the fact that $V$ is irreducible is then equivalent to its simplicity as a $(\mathfrak{g},K)_q$-module. This proves the first statement.

    Let $W$ be another unitary irreducible $(\mathfrak{g},K)_q$-module and assume that there exists a $(\mathfrak{g},K)_q$-isomorphism $f : V \to W$. Note that $V$ and $W$ are the orthogonal direct sums of their respective $K_q$-isotypical components, because $(e_n)_{n\in \mathbb{Z}}$ is a family of disjoint self-adjoint idempotents. Since both $V$ and $W$ are admissible and $f$ maps $e_nV$ to $ e_nW$ for all $n \in \mathbb{Z}$, there exists a linear map $f^* : W \to V$ such that
    $$(f(v)|w) = (v|f^*(w)) \qquad (v \in V, w \in W).$$
    Then $f^*f : V \to V$ is a $(\mathfrak{g},K)_q$-intertwiner. By simplicity of $V$ and Schur's lemma, $f^*f$ is necessarily a positive scalar $t$. Then $t^{-1/2}f : V \to W$ is an isometry, so $V$ and $W$ are unitarily equivalent. This proves the second statement. The last one follows directly from the first.
\end{proof}

As a consequence of the above results and by means of Proposition \ref{unitarydual}, the spectrum of $C^*_q(G)$ identifies with the subset of $\mathrm{Irr}_q(\mathfrak{g},K)$ corresponding to simple $(\mathfrak{g},K)_q$-modules admitting an invariant inner product. The latter has been classified by De Commer and Dzokou Talla \cite{DCDTsl2R}*{Theorem 3.17}. We shall give later a restatement of that classification in terms of parabolically induced modules. Another consequence is the following one.

\begin{corollary}
    The C*-algebra $C^*_q(G)$ is liminal, in particular it is of type I.
\end{corollary}

\begin{proof}
    One can restate the admissibility of irreducible unitary $(\mathfrak{g},K)_q$-modules as follows: for each $n \in \mathbb{Z}$, the element $e_n$ acts as a finite rank projection on any irreducible representation of $C^*_q(G)$. Since $R_q(\mathfrak{g},K) = R_q(K)\mathcal{O}_q(K\backslash U)$ is norm-dense in $C^*_q(G)$, the elements of $C^*_q(G)$ act as compact operators on any of its irreducible representations.
\end{proof}

\section{Analogue of parabolic induction}

In this section, we construct an analogue of parabolic induction. Denoting by $P = MAN$ the Iwasawa decomposition of the standard minimal parabolic subgroup of $G = \mathrm{SL}(2,\mathbb{R})$, we define a $q$-analogue $R_q(\mathfrak{p},M)$ of the Hecke algebra relative to the pair $(\mathfrak{p},M)$. The induction procedure is then essentially defined as suggested in Section 1. We prove that it preserves the unitarity of modules when $\mathbb{K} = \mathbb{C}$. Finally, we explicitly determine the simple submodules and quotients of the $(\mathfrak{g},K)_q$-modules that are parabolically induced from characters; we also exhibit all the intertwining maps between them.

\subsection*{Constructions}
The components of the Iwasawa decomposition of the standard parabolic subgroup $P$ of $G$ are $M = \{1,-1\}$ and
$$A = \left\{ \begin{pmatrix}
    a & 0\\
    0 & a^{-1}
\end{pmatrix} : a \in \mathbb{R}^*\right\}, \qquad N = \left\{ \begin{pmatrix}
    1 & b\\
    0 & 1
\end{pmatrix} : b \in \mathbb{R}\right\}.$$
The analogue of the $AN$ part of $G$ at the level of the enveloping algebra is $\mathcal{O}_q(K \backslash U)$, as explained in the first section. On the other hand, a natural analogue of the group algebra of $M$ is the unital subalgebra of $\mathcal{O}'_q(K)$ generated by $\sum_{n \in \mathbb{Z}} (-1)^n e_n$. Let us brielfy justify this. At $q= 1$, the element $\theta$ specializes to $i(E-F)$ so that $K = \mathrm{SO}(2) = \{ e^{it\theta} : t\in \mathbb{R}\}$. Let $\mathcal{O}'(K)$ be the dual of the space of regular functions on $K$. As a convolution algebra $\mathcal{O}'(K)$ is just a countable product of copies of $\mathbb{C}$; more precisely, if $e_n \in \mathcal{O}'(K)$ is defined by $\langle e_n , f \rangle = (2\pi)^{-1}\int_0^{2\pi} f(e^{it\theta})e^{-itn}dt$ for all $n \in \mathbb{Z}$, the following map is an isomorphism of algebras:
$$\begin{array}{cccll}
    \mathbb{C}^\mathbb{Z} & \longrightarrow & \mathcal{O}'(K) \\
    (\lambda_n)_{n \in \mathbb{Z}} & \longmapsto & \sum_{n \in \mathbb{Z}} \lambda_n e_n.
\end{array}$$
Then $K$ can be embedded into $\mathcal{O}'(K)^\times$ via $e^{it\theta} \mapsto \sum_{n\in\mathbb{Z}}e^{itn}e_n$, so the element $-1 = e^{i\pi\theta} \in M$ identifies with $\sum_{n} (-1)^ne_n$. The group algebra of $M$ then identifies with the unital subalgebra of $\mathcal{O}'(K)$ generated by this element.

Motivated by these considerations, we introduce the following:
\begin{definition}
    The $q$-deformed Hecke algebra relative of the pair $(\mathfrak{p},M)$ is defined as the subalgebra of $\mathcal{O}'_q(G)$ generated by $\mathcal{O}_q(K \backslash U)$ and $(-1)_q = \sum_{n\in \mathbb{Z}}(-1)^n e_n.$
    We denote it by $R_q(\mathfrak{p},M)$.
\end{definition}

Note that $R_q(\mathfrak{p},M)$ is a unital $\ast$-subalgebra of $\mathcal{O}'_q(G)$. Since $(-1)_q^2 = 1$, the linear span of $\{1,(-1)_q\}$ is a subalgebra of $R_q(\mathfrak{p},M)$ which we denote by $R_q(M)$. We have the following basic structure result:
\begin{proposition}
    The center of $R_q(\mathfrak{p},M)$ contains $R_q(M)$. Moreover, the multiplication map induces a isomorphism of algebras $R_q(M) \otimes \mathcal{O}_q(K\backslash U) \to R_q(\mathfrak{p},M)$.
\end{proposition}
\begin{proof}
    Recall the action of $k$ and $\theta$ on any type 1 irreducible finite dimensional $U_q(\mathfrak{u})$-module $V$ are diagonalizable (see Lemma \ref{diagtheta}). Their respective set of eigenvalues are
    $$\{ q^n :| n| < \mathrm{dim}V,\, n - \mathrm{dim}V \text{ odd}\}, \qquad \{ [n]_q : |n| < \mathrm{dim}V,\, n - \mathrm{dim}V \text{ odd}\}.$$ If $\varphi_1 : q^n \mapsto (-1)^n$ and $\varphi_2 : [n]_q \mapsto (-1)^n$ then $\varphi_1 (k)$ and $\varphi_2(\theta)$, interpreted in the sense of algebraic functional calculus, are well-defined elements of $\mathcal{O}'_q(U)$. Almost by definition, we have $(-1)_q = \varphi_2(\theta)$. On the other hand, $\varphi_1 (k)$ and $\varphi_2(\theta)$ act in the same way on any type 1 irreducible finite dimensional $U_q(\mathfrak{u})$-module, so they are equal. Hence we have $(-1)_q = \varphi_2(k)$.
    
    Since $\varphi_1$ is a group homomorphism and $k$ is a grouplike element, so is $(-1)_q$. The algebra $\mathcal{O}_q(K \backslash U)$ is spanned by matrix coefficients of odd dimensional type~1 representations of $U_q(\mathfrak{u})$, hence $y \triangleleft (-1)_q = y$ for all $y \in \mathcal{O}_q(K \backslash U)$. Given the exchange relations (\ref{exchange}) and the fact that $O_q(K \backslash U)$ is right coideal of $O_q(U)$, the elements of $O_q(K \backslash U)$ all commute with $(-1)_q$. This proves the first statement.

    The second statement follows directly from the fact that the multiplication induces a linear isomorphism $\mathcal{O}_q'(K) \otimes \mathcal{O}_q(K \backslash U) \to \mathcal{O}'_q(G)$, see the discussion after Definition 3.1 of \cite{DCquantisation}.
\end{proof}

Since $(-1)_q R_q(K) \subset R_q(K)$, we have $R_q(\mathfrak{p},M) R_q(\mathfrak{g},K) \subset R_q(\mathfrak{g},K)$. Hence $R_q(\mathfrak{g},K)$ is a non-degenerate $R_q(\mathfrak{p},M)$-module for the action by left multiplication. According to Definition \ref{induction} and the discussion following it, there is a well-defined induction functor
\begin{equation}\label{concreteinduction} \mathrm{Ind}_{R_q(\mathfrak{p},M)}^{R_q(\mathfrak{g},K)} : \mathsf{Mod}_{R_q(\mathfrak{p},M)} \to \mathsf{Mod}_{R_q(\mathfrak{g},K)},\end{equation}
which associates a $(\mathfrak{g},K)_q$-module to any unital $R_q(\mathfrak{p},M)$-module.

As we will see later, when $\mathbb{K} = \mathbb{C}$, the induction functor (\ref{concreteinduction}) does not preserve unitarity. In order to fix this, we introduce a modified induction functor which twists the $R_q(\mathfrak{p},M)$-module in the first place. This can be compared to the $\rho$-shift for classical parabolic induction, see \cite{vandenBan}*{§2}.

By applying the function $q^n \mapsto q^{n/2}$ on $k$, we get an element of $\mathcal{O}_q'(U)$ that we denote by $k^{1/2}$. Let $J$ be the isomorphism of $R_q(\mathfrak{p},M)$ defined by
$$J((-1)_q) = 1, \qquad J(x)  = k^{1/2} \triangleright x \quad (\forall x \in \mathcal{O}_q(K\backslash U)).$$
If $V$ is any $R_q(\mathfrak{p},M)$-module, we denote by $V_J$ the $R_q(\mathfrak{p},M)$-module with underlying space $V$ but endowed with the following twisted action of $R_q(\mathfrak{p},M)$:
$$x\cdot v = J(x)v \qquad (x \in R_q(\mathfrak{p},M), v \in V).$$

\begin{definition}
    Let $V$ be a unital $R_q(\mathfrak{p},M)$-module. The $(\mathfrak{g},K)_q$-module
    parabolically induced from $V$, denoted by $\mathrm{Ind}_q(V)$, is the non-degenerate $R_q(\mathfrak{g},K)$-module
    $$\mathrm{Ind}_{R_q(\mathfrak{p},M)}^{R_q(\mathfrak{g},K)}(V_J).$$
\end{definition}

\subsection*{The compact picture and unitarity of induced modules} Let us now describe the parabolically induced $(\mathfrak{g},K)_q$-modules in a more concrete way. Among other applications, this will allow us to show that $\mathrm{Ind}_q$ preserves unitarity.

Let $U_q(\mathfrak{k})^\perp$ denote the annihilator of $U_q(\mathfrak{k})$ in $\mathcal{O}_q(U)$. Let us denote by $\mathcal{O}_q(K)$ the quotient of $\mathcal{O}_q(U)$ by $U_q(\mathfrak{k})^\perp$. From the fact that $U_q(\mathfrak{k})$ is a left coideal subalgebra of $U_q(\mathfrak{u})$, it follows that $U_q(\mathfrak{k})^\perp$ is a left ideal subcoalgebra of $\mathcal{O}_q(U)$. Hence $\mathcal{O}_q(K)$ inherits by quotient a structure of left $\mathcal{O}_q(U)$-module coalgebra. Moreover, recall that $\mathcal{O}_q(U)$ is dually paired with $\mathcal{O}'_q(U)$ in a non-degenerate way and that $\mathcal{O}_q'(K)$ is precisely the annihilator of $U_q(\mathfrak{k})^\perp$ in $\mathcal{O}'_q(U)$. The induced pairing between $\mathcal{O}_q'(K)$ and $\mathcal{O}_q(K)$ is thus non-degenerate.

Besides, considered as linear forms on $\mathcal{O}_q'(K)$, the elements of $\mathcal{O}_q(K)$ are continuous. Since $R_q(K)$ is dense in $\mathcal{O}_q'(K)$, the above pairing restricts to a non-degenerate pairing between $\mathcal{O}_q(K)$ and $R_q(K)$. Concretely, given the content of Lemma \ref{diagtheta}, this pairing identifies $\mathcal{O}_q(K)$ with the subspace of the dual of $R_q(K)$ spanned by the dual basis $(\zeta_n)_{n \in \mathbb{Z}}$ of $(e_n)_{n\in \mathbb{Z}}$. For each $\varepsilon \in \{-1,1\}$, let us write
$$\mathcal{O}_q^\varepsilon(K) = \{ \phi \in \mathcal{O}_q(K) :  \phi \triangleleft (-1)_q= \varepsilon \phi \}.$$
We clearly have $\mathcal{O}_q(K) = \mathcal{O}_q^{-1}(K) \oplus \mathcal{O}^1_q(K)$ and for each $\varepsilon \in \{-1,1\}$,
$$\mathcal{O}_q^\varepsilon(K) = \mathrm{span}(\zeta_n : n \in \mathbb{Z}^\varepsilon),$$
where $\mathbb{Z}^\varepsilon$ denotes the set of even integers if $\varepsilon = 1$, the set of odd integers otherwise.

Let $V$ be any $R_q(\mathfrak{p},M)$-module. We consider $\mathcal{O}_q(K) \otimes V$ as embedded in $\mathrm{Hom}_\mathbb{K}(R_q(K), V)$ as follows: if $\phi \in \mathcal{O}_q(K)$ and $ v \in V$ then $\phi\otimes v $ is the linear map $R_q(K) \to V$ defined by $(\phi\otimes v)(x) = \langle \phi, x \rangle v$ for all $x \in R_q(K)$. The action of $R_q(K)$ on itself by right-multiplication induces on $\mathrm{Hom}_\mathbb{K}(R_q(K), V)$ a structure of left $R_q(K)$-module. The expression of this action on $\mathcal{O}_q(K) \otimes V$ is quite simple:
$$x(\phi \otimes v) = (x \triangleright \phi )\otimes v, \qquad [\forall x \in R_q(K), \forall \phi \in \mathcal{O}_q(K), \forall v \in V].$$
Moreover, $\mathcal{O}_q(K) \otimes V$ is precisely the non-degenerate part of the $R_q(K)$-module $\mathrm{Hom}_\mathbb{K}(R_q(K), V)$. Indeed, for every $n \in \mathbb{Z}$, we have $e_n \mathrm{Hom}_\mathbb{K}(R_q(K), V) = \zeta_n \otimes V$. The non-degenerate part of its submodule $\mathrm{Hom}_{R_q(M)}(R_q(K), V)$ then identifies with
$$\mathcal{O}_q(K)\otimes_{R_q(M)}V = \mathcal{O}^1_q(K)\otimes V^1 \oplus \mathcal{O}^{-1}_q(K)\otimes V^{-1},$$
where $V^\varepsilon = \{ v \in V : (-1)_qv = \varepsilon v \}$ for each $\varepsilon \in \{-1,1\}$.

Let $\mathcal{O}_q(K\backslash U)$ act on $\mathcal{O}_q(K) \otimes V$ on the left as follows:
$$y(\phi \otimes v) = (y_{(2)}\phi )\otimes (k^{1/2} \triangleright y_{(1)} )v, \qquad (y \in \mathcal{O}_q(K\backslash U), \phi \in \mathcal{O}_q(K), v \in V).$$
Because they satisfy the exchange relations (\ref{exchange}), the combined actions of $\mathcal{O}_q(K\backslash U)$ and $R_q(K)$ induce on $\mathcal{O}_q(K) \otimes V$ a structure of $(\mathfrak{g},K)_q$-module. The following proposition give an alternative description of $\mathrm{Ind}_q(V)$.

\begin{proposition}
    Let $\mathrm{res}_{K_q} : \mathrm{Hom}_{R_q(\mathfrak{p},M)}(R_q(\mathfrak{g},K), V) \to \mathrm{Hom}_{R_q(M)}(R_q(K), V)$ be the restriction to $R_q(K)$. This map is bijective and it restricts to an isomorphism of $(\mathfrak{g},K)_q$-modules
    \begin{equation}\label{compact picture}\mathrm{Ind}_q(V) \longrightarrow \mathcal{O}_q(K)\otimes_{R_q(M)}V,\end{equation}
    the latter being considered as a $(\mathfrak{g},K)_q$-submodule of $\mathcal{O}_q(K) \otimes V$.
\end{proposition}

\begin{proof}
    The fact that $\mathrm{res}_{K_q}$ is injective is due to the fact that $R_q(\mathfrak{g},K)$ identifies with $\mathcal{O}_q(K\backslash U) \otimes R_q(K)$ via multiplication, with $\mathcal{O}_q(K\backslash U) \subset R_q(\mathfrak{p},M)$. For the surjectivity, one has to use in addition that $R_q(M)$ is central in $R_q(\mathfrak{p},M)$.
    
    By definition, $\mathrm{Ind}_q(V)$ is the non-degenerate part of $\mathrm{Hom}_{R_q(\mathfrak{p},M)}(R_q(\mathfrak{g},K), V)$. This non-degenerate part does not depend on wether we consider the latter as a $R_q(\mathfrak{g},K)$-module or as a $R_q(K)$-module because both algebras have an approximate unit in common. Clearly, the map $\mathrm{res}_{K_q}$ intertwines the actions of $R_q(K)$. Hence $\mathrm{res}_{K_q}(\mathrm{Ind}_q(V))$ is the non-degenerate part of $\mathrm{Hom}_{R_q(M)}(R_q(K), V)$, that is $\mathcal{O}_q(K) \otimes_{R_q(M)}V$. Thus, $\mathrm{res}_{K_q}$ induces an isomorphism of $R_q(K)$-modules from $\mathrm{Ind}_q(V)$ to $\mathcal{O}_q(K) \otimes_{R_q(M)}V$.

    To conclude the proof, let us check that this isomorphism also intertwines the actions of $\mathcal{O}_q(K\backslash U)$. Let $\xi \in \mathrm{Ind}_q(V)$. Let us write
    $\mathrm{res}_{K_q}(\xi) = \sum_j \phi_j \otimes v_j,$
    where $\phi_j \in \mathcal{O}_q(K)$, $v_j \in V$ for all $j$. If $y \in \mathcal{O}_q(K\backslash U)$ and $x \in R_q(K)$ then
    \begin{align*}
        [\mathrm{res}_{K_q}(y\xi)](x) &= \xi(xy) = \xi [y_{(1)}(x \triangleleft y_{(2)})] = J(y_{(1)})\xi(x \triangleleft y_{(2)})\\
        & = \textstyle\sum_j \langle y_{(2)} \phi_j, x \rangle \otimes (k^{1/2} \triangleright y_{(1)})v_j = [y\, \mathrm{res}_{K_q}(\xi)](x).
    \end{align*}
\end{proof}

Now, let us explain how a $(\mathfrak{g},K)_q$-module induced from a unitary $R_q(\mathfrak{p},M)$-module can itself be naturally unitarized, when $\mathbb{K} = \mathbb{C}$. For that purpose, here is our key result.

\begin{proposition} \label{innerprod}
    Let $(\cdot |\cdot)$ be the sesquilinear form on $\mathcal{O}_q(U)$ defined by
    $$(\phi | \psi) = \langle (\phi^* \triangleleft k) \psi , e_0 \rangle \qquad (\phi,\psi \in \mathcal{O}_q(U)).$$
    This sesquilinear form is hermitian, non-negative and its kernel is the annihilator of $U_q(\mathfrak{k})$ in $\mathcal{O}_q(U)$. The subsequent inner product on $\mathcal{O}_q(K)$, which we write identically, satisfies the following properties:
    \begin{enumerate}[label = (\roman*)]
        \item\label{i} the basis $(\zeta_n)_{n\in \mathbb{Z}}$ is orthogonal and we have $(\zeta_n |\zeta_n) = 2\{n\}_q^{-1}$ for all $n \in \mathbb{Z}$,
        \item\label{ii} $\forall(\phi,\psi) \in \mathcal{O}_q(K)^2, \forall x \in \mathcal{O}_q'(K), \quad(x \triangleright \phi | \psi) = (\phi | x^*\triangleright \psi )$,
        \item \label{iii} $\forall (\phi,\psi) \in \mathcal{O}_q(K)^2, \forall y \in \mathcal{O}_q(U), \quad (y \phi | \psi) = (\phi | (y^* \triangleleft k) \psi )$.
    \end{enumerate}
\end{proposition}

The proof of the proposition is based on the following lemma.

\begin{lemma}\label{Deltae0}
    We have $S(e_0) = e_0 k$ and
    $$\Delta(e_0) = \sum_{n \in \mathbb{Z}}2 \{n\}_q^{-1} S(e_n) k^{-1} \otimes e_n.$$
    Moreover, for any $x \in \mathcal{O}'_q(K)$, we have $(S(x) \otimes 1)\Delta(e_0) = (1 \otimes x) \Delta(e_0).$
\end{lemma}

\begin{proof}[Proof of the lemma]
By construction, we have $x e_0 = e_0 x = \epsilon(x)e_0$ for all $x \in U_q(\mathfrak{k})$, where $\epsilon$ denotes the counit of $U_q(\mathfrak{u})$. Using this, we compute for every $x \in U_q(\mathfrak{k})$:
\begin{align*}
        (S(x) \otimes 1)\Delta(e_0) & = (S(x_{(1)})\otimes 1) \Delta(\epsilon(x_{(2)})e_0) = (S(x_{(1)})\otimes 1) \Delta(x_{(2)}e_0) \\
        & = (S(x_{(1)})\otimes 1) \Delta(x_{(2)}) \Delta(e_0) =(S(x_{(1)})x_{(2)}\otimes x_{(3)})\Delta(e_0) \\
        & = (\epsilon(x_{(1)})\otimes x_{(2)})\Delta(e_0)  = (1 \otimes x)\Delta(e_0).
    \end{align*}
    For the second equality, we used that $\mathscr{U}_q^*(\mathfrak{k})$ is a left coideal of $\mathscr{U}_q^*(\mathfrak{u})$. The map
    $$x \in \mathcal{O}_q'(K) \longmapsto (S(x) \otimes 1 - 1 \otimes x)\Delta(e_0)\in \mathcal{O}'_q(U^2)$$
    is continuous and vanishes on $U_q(\mathfrak{k})$. The latter is dense in $\mathcal{O}'_q(K)$ so we have indeed $(S(x) \otimes 1)\Delta(e_0) = (1 \otimes x)\Delta(e_0)$ for all $x \in \mathcal{O}'_q(K)$.

    The rest of the lemma is somehow hidden throughout \cite{DCDTinvariant}. By Theorem 1.12 of that article, the equality $S(e_0) = e_0 k$ is equivalent to the existence of a non-zero $k^{-1}$-invariant functional on $R_q(K)$, which has been proven a little further (Theorem 3.1). The fact that there exist constants $c_n > 0$ such that
$$\Delta(e_0) = \sum_{n \in \mathbb{Z}}c_n S(e_n) k^{-1} \otimes e_n$$
can be found in the proof of Theorem 1.12, see in particular formula (1.31). Formula (1.33) of the same proof shows how to compute these constants. In our case, we have $c_n e_n K e_n = e_n$. The explicit value of $c_n^{-1}$ for each $n\in\mathbb{Z}$ is then computed in the proof of Theorem 3.1 (it is equal to $\mu_{[a+n]}$ in the case $a=0$).
    
\end{proof}

\begin{proof}[Proof of the proposition]
    For now, we consider $(\cdot |\cdot)$ as a sesquilinear form on $\mathcal{O}_q(U)$. Let us prove first that $(\cdot | \cdot )$ is hermitian. By the above lemma, we have $S(e_0)^* = k e_0$, so for any $\phi,\psi \in \mathcal{O}_q(K\backslash U)$,
\begin{align*}
    \overline{(\psi | \phi)} &= \overline{\langle (\psi^* \triangleleft k)\phi, e_0 \rangle} = \langle \phi^* (\psi^* \triangleleft k)^*, S(e_0)^* \rangle \\
    & = \langle \phi^* (\psi \triangleleft k^{-1}), ke_0 \rangle = \langle [\phi^* (\psi \triangleleft k^{-1})] \triangleleft k, e_0 \rangle\\
    & = \langle (\phi^* \triangleleft k)(\psi \triangleleft k^{-1} k), e_0 \rangle = (\phi | \psi).
\end{align*}
Since $\Delta(e_0) \in \mathcal{O}_q'(U \times K)$, we have $(\phi |\psi)=0$ whenever $\psi$ is in the annihilator of $U_q^*(\mathfrak{k})$ in $\mathcal{O}_q(U)$. The latter is thus included in the kernel of $(\cdot | \cdot)$. This means that $(\cdot |\cdot)$ is well defined as a hermitian form on $O_q(K)$.

Now, we prove the formulas \textit{\ref{ii}} and \textit{\ref{iii}}. Let $\phi, \psi \in \mathcal{O}_q(U)$ and $y \in \mathcal{O}_q(U), x \in \mathcal{O}_q'(K)$. We have
\begin{align*}
    (y \phi | \psi) &= \langle [(\phi^*y^*)\triangleleft k]\psi, e_0 \rangle = \langle (\phi^*\triangleleft k)(y^* \triangleleft k)\psi, e_0 \rangle = (\phi | (y^* \triangleleft k)\psi ),
\end{align*}
as well as
\begin{align*}
    (x\triangleright \phi | \psi ) & = \langle [(x \triangleright \phi)^*\triangleleft k]\psi, e_0 \rangle  = \langle (S(x)^* \triangleright \phi^* \triangleleft k)\psi, e_0 \rangle \\
    & = \langle (\phi^* \triangleleft k) \otimes \psi , \Delta(e_0)(S(x)^* \otimes 1) \rangle.
\end{align*}
Using the last formula of lemma \ref{Deltae0}, we have
$$\Delta(e_0)(S(x)^* \otimes 1) = [(S(x) \otimes 1)\Delta(e_0)]^* = [(1\otimes x) \Delta(e_0)]^* = \Delta(e_0) (1 \otimes x^*),$$
and hence
\begin{align*}
    (x\triangleright \phi | \psi ) & = \langle (\phi^* \triangleleft k) \otimes \psi , \Delta(e_0)(1\otimes x^*) \rangle \\
    & = \langle (\phi^* \triangleleft k) \otimes (x^*\triangleright \psi) , \Delta(e_0) \rangle = (\phi | x^*\triangleright \psi).
\end{align*}
These identities still hold after passing to the quotient.

Finally, in order to complete the proof of the proposition, we just check \textit{\ref{i}}. Using that $S^2$ is the conjugation by $k^{-1}$ and the main formula of the lemma, we obtain
$$(k\otimes 1)\Delta(e_0) = \sum_{n \in \mathbb{Z}}c_n S^{-1}(e_n) \otimes e_n,$$
where $c_n = 2/\{n\}_q$ for all $n \in \mathbb{Z}$. Hence, for every $\phi,\psi \in \mathcal{O}_q(K)$, we have
$$(\psi| \phi) = \langle \psi^*\otimes \phi, (1\otimes k) \Delta(e_0)\rangle = \sum_{n \in \mathbb{Z}}c_n \overline{\langle \psi, e_n \rangle} \langle \phi, e_n \rangle.$$
This concludes the proof.
\end{proof}

Recall that $R_q(\mathfrak{p},M)$ is a $\ast$-algebra. Let us call unitary any $R_q(\mathfrak{p},M)$-module equipped with an invariant inner product for which it is a Hilbert space. By \textit{invariant}, we mean that $x^*$ acts as the adjoint of $x$ with respect to the inner product, for all $x \in R_q(\mathfrak{p},M)$.

\begin{corollary} \label{unitaritypreserved}
    Let $V$ be a unitary $R_q(\mathfrak{p},M)$-module. We equip $\mathcal{O}_q(K)$ with the inner product constructed in Proposition \ref{innerprod}. Let us endow $\mathcal{O}_q(K) \otimes V$ with the tensor product inner product. Equipped with the latter, $\mathcal{O}_q(K) \otimes_{R_q(M)} V$ becomes a unitary $(\mathfrak{g},K)_q$-module.
\end{corollary}

Thus, regarding the isomorphism (\ref{compact picture}), if $V$ is a unitary $R_q(\mathfrak{p},M)$-module then $\mathrm{Ind}_q(V)$ is a unitary $(\mathfrak{g},K)_q$-module. This property is functorial: an intertwining isometry $V \to W$ between two unitary $R_q(\mathfrak{p},M)$-modules induces an intertwining isometry $\mathrm{Ind}_q(V) \to \mathrm{Ind}_q(W)$.

Since it only relies on straightforward computations using the invariance properties of the inner product of Proposition \ref{innerprod}, we omit the proof of this corollary. However, let us comment on it in view of the induction procedure proposed in \cite{DCinduction}. The triple $(\mathcal{O}_q(U), \mathcal{O}_q(K\backslash U), \mathcal{O}_q(K))$ is a unitary Doi-Kopinen datum of coideal type. Using the notation of \cite{DCinduction}, let $L^2_{0}(K_q)$ be the object of $_{\mathcal{O}_q(U)}\mathrm{Rep}^{\mathcal{O}_q(K)}$ whose underlying unitary $\mathcal{O}_q(K)$-comodule is $\mathcal{O}_q(K)$ (equipped with the above inner product) but with the following twisted $\mathcal{O}_q(U)$-action:
$$y \cdot\varphi = (y \triangleleft k^{1/2})\varphi, \qquad (\varphi \in \mathcal{O}_q(K), y \in \mathcal{O}_q(U)).$$
If $V$ is a unitary $R_q(\mathfrak{p},M)$-module, then it is in particular a $\ast$-representation of $\mathcal{O}_q(K\backslash U)$ and its $L^2_0(K_q)$-induction in the sense of \cite{DCinduction}*{Definition 2.3} exactly corresponds to the $R_q(\mathfrak{g},K)$-module $\mathcal{O}_q(K) \otimes V$ via \cite{DCquantisation}*{Theorem 3.3}.

\subsection*{$(\mathfrak{g},K)_q$-modules induced from characters}
In classical representation theory, the interesting parabolically induced representations of $G$ come from representations of $P$ that are trivial on the nilpotent radical $N$. As discussed in Section 3, the analogue of the nilpotent radical at the level of the quantized enveloping algebra is the ideal of $\mathcal{O}_q(K\backslash U)$ generated by $Z$. We are thus inclined to study the $(\mathfrak{g},K)_q$-modules induced from $R_q(\mathfrak{p},M)$-modules on which $Z$ acts trivially. In this regard, from Proposition \ref{ab}, we deduce the following.

\begin{proposition} The kernel of the morphism of algebras $$\mathrm{id} \otimes \mathrm{ab} : R_q(\mathfrak{p},M) \cong R_q(M) \otimes \mathcal{O}_q(K\backslash U) \to R_q(M) \otimes \mathbb{K}[t,t^{-1}]$$
is the ideal of $R_q(\mathfrak{p},M)$ generated by $Z$ (see Proposition \ref{ab} for the definition of $\mathrm{ab}$). This morphism identifies $R_q(M) \otimes \mathbb{K}[t,t^{-1}]$ with the abelianization of $R_q(\mathfrak{p},M)$.

The unital simple modules over $R_q(\mathfrak{p},M)$ on which $Z$ acts trivially correspond to its characters. For any $(\varepsilon, \lambda) \in \{-1,1\} \times \mathbb{K}^\times$, let $\chi_{\varepsilon, \lambda}$ be the unique character of $R_q(\mathfrak{p},M)$ such that $\langle \chi_{\varepsilon, \lambda},(-1)_q\rangle = \varepsilon$,  $\langle \chi_{\varepsilon, \lambda}, X\rangle = \lambda$ and $\langle \chi_{\varepsilon, \lambda}, Y\rangle = \lambda^{-1}$. Every character of $R_q(\mathfrak{p},M)$ is of this form.
\end{proposition}

For every $(\varepsilon , \lambda) \in \{-1,1\} \times \mathbb{K}^\times$, let us denote by $\mathrm{Ind}_q(\varepsilon, \lambda)$ the $(\mathfrak{g},K)_q$ parabolically induced from the character $\chi_{\varepsilon,\lambda}$. In the following, our goal is to study the simplicity of these induced $(\mathfrak{g},K)_q$-modules, as well as the intertwining maps between them.

Let us fix a pair $(\varepsilon, \lambda) \in \{-1,1\} \times \mathbb{K}^\times$. The isomorphism (\ref{compact picture}) identifies $\mathrm{Ind}_q(\varepsilon, \lambda)$ with $\mathcal{O}_q^\varepsilon(K)$. Let us denote by $x_{\varepsilon,\lambda}$ the corresponding action of every $x\in R_q(\mathfrak{g},K)$ on the latter. For any $\phi \in \mathcal{O}_q^\varepsilon(K)$, we have
\begin{gather*}
    x_{\varepsilon,\lambda} \phi = x \triangleright \phi \qquad (x \in R_q(\mathfrak{g},K)),\\ y_{\varepsilon,\lambda} \phi = \langle \chi_{\varepsilon,\lambda}, k^{1/2} \triangleright y_{(1)}\rangle y_{(2)}\phi \qquad (y \in \mathcal{O}_q(K\backslash U)).
\end{gather*}
From the very definition (\ref{defXYZ}) of $X$ and $Y$, one can easily derive
$$(k^{1/2} \triangleright X) = qX, \qquad (k^{1/2} \triangleright Y) = q^{-1}Y.$$
Hence for any $z \in \mathcal{O}_q(K\backslash U)$, we have $\langle \chi_{\varepsilon,\lambda}, k^{1/2} \triangleright z\rangle = \langle \chi_{\varepsilon, q\lambda}, z \rangle$. The above action of $y \in \mathcal{O}_q(K\backslash U)$ on $\phi \in \mathcal{O}_q^\varepsilon(K)$ thus simplifies to
$$y_{\varepsilon,\lambda} \phi =  (y \triangleleft \chi_{\varepsilon, q\lambda})\phi.$$

Recall that $(\zeta_n)_{n \in \mathbb{Z}^\varepsilon}$ is a basis of $\mathcal{O}_q^\varepsilon(K)$. The following proposition describes the action of $\mathcal{O}_q(K\backslash U)$ on this basis.

\begin{proposition}\label{actiontransition} Recall from (\ref{transition}) the definition of the transition operators $T_n^\pm, T_n$. Their action on $\zeta_n \in \mathrm{Ind}_q(\varepsilon, \lambda)$ is given by
$$(T_n^\pm)_{\varepsilon,\lambda} \zeta_n = (\lambda  q^{1\pm n} - \lambda^{-1}q^{-1\mp n})\,\zeta_{n\pm 2}, \qquad (T_n)_{\varepsilon,\lambda} \zeta_n  = (\lambda + \lambda^{-1}) \zeta_n.$$
The Casimir element $\Omega$ acts on $\mathrm{Ind}_q(\varepsilon, \lambda)$ as the scalar $(\lambda + \lambda^{-1})$.
\end{proposition}

\begin{proof}
    Let $y \in \mathcal{O}_q(K\backslash U)$ and $x \in R_q(K)$. For any $n \in \mathbb{Z}^\varepsilon$, we have
    \begin{align*}\langle y_{\varepsilon,\lambda}\zeta_n, x \rangle &= \langle(y \triangleleft \chi_{\varepsilon, q\lambda}) \zeta_n , x \rangle = \langle y \triangleleft \chi_{\varepsilon, q\lambda}, \zeta_n \triangleright x \rangle. \end{align*}
    Let us compute $y \triangleleft \chi_{\varepsilon, q\lambda}$ explicitly for $y \in \{ T_n^\pm, T_n\}$. Recall from (\ref{w_n}) the definition of the vectors $w_n^+, w_n^-, w_n$. Let us denote by $(w^{*+}_n, w_n^{*-}, w^*_n$) the corresponding dual basis. For any index $\mu = +, -, \emptyset $, we have
    \begin{gather*}
        \langle (T_n^\mu)_{\varepsilon,\lambda} \zeta_n, x \rangle = \sum_{\nu =+, -, \emptyset }\langle \chi_{\varepsilon,q \lambda}, T_n^\nu \rangle \langle w^{*\nu}_n, (\zeta_n \triangleright x)w_n^\mu \rangle.
    \end{gather*}
    Let us consider $\zeta_n$ as a linear form on $\mathcal{O}'_q(K)$. Since it is a continuous ring homomorphism, $\zeta_n$ is uniquely determined by its evaluation at $\theta$. We have $\Delta(\theta) = k \otimes \theta + \theta \otimes 1$, so $\zeta_n \triangleright \theta = \theta + [n]_q k$. Recall that for any index $\mu$, we have $(\theta + [n]_q K)w_n^\mu = [n_\mu]_q w_n^\mu$, where $n_\mu = n\pm 2$ if $\mu = \pm$ and $n_\mu = n$ if $\mu = \emptyset$.
    Hence we have
    \begin{align*}\langle (T_n^\mu)_{\varepsilon,\lambda}\zeta_n, x \rangle &= \sum_{\nu=+, -, \emptyset }\langle \chi_{\varepsilon,q \lambda}, T_n^\nu \rangle\langle w_n^{*\nu}, \langle \zeta_{n_\mu}, x \rangle w_n^\mu \rangle = \langle \chi_{\varepsilon,q \lambda}, T_n^\mu \rangle\langle \zeta_{n_\mu}, x \rangle.\end{align*}
    Then, using (\ref{TXYZ}), we get
    \begin{gather*}\langle \chi_{\varepsilon, q\lambda}, T_n^\pm \rangle = \lambda  q^{1\pm n} - \lambda^{-1}q^{-1\mp n},\qquad
    \langle \chi_{\varepsilon, q\lambda}, T_n\rangle = \lambda + \lambda^{-1}.
    \end{gather*}
    This concludes the proof of the formulas involving the transition operators. Formula (\ref{transitions3}) implies that $ \Omega_{\varepsilon,\lambda} \zeta_n = (T_n)_{\varepsilon,\lambda} \zeta_n = (\lambda + \lambda^{-1})\zeta_n$ for all $n \in \mathbb{Z}^\varepsilon$, which proves the last statement.
\end{proof}

Now we can state our main results concerning the $(\mathfrak{g},K)_q$-modules induced from characters of $R_q(\mathfrak{p},M)$. The following two theorems describe the simple subquotients of these modules as well as the intertwining maps between them. Note the striking similarity of these results compared to their classical counterparts, the ones involving the non-unitary principal series of~$\mathrm{SL}(2,\mathbb{R})$.

\begin{theorem}\label{submodules} Let $\varepsilon \in\{-1,1\}$. For every $n\in \mathbb{Z}^{-\varepsilon}_{\geq 0}$ and $\sigma \in \{-1,1\}$, the $(\mathfrak{g},K)_q$-module $\mathrm{Ind}_q(\varepsilon, \sigma q^n)$ has exactly two simple submodules, namely:
    $$D_{\sigma, n}^+ = \mathrm{span}(\zeta_m : m \in \mathbb{Z}^\varepsilon_{>n}), \qquad D_{\sigma, n}^- = \mathrm{span}(\zeta_m : m \in \mathbb{Z}^\varepsilon_{<-n}).$$
    For every $n\in \mathbb{Z}^{-\varepsilon}_{>0}$ and $\sigma \in \{-1,1\}$, the $(\mathfrak{g},K)_q$-module $\mathrm{Ind}_q(\varepsilon, \sigma q^{-n})$ has a unique simple submodule: $$Q_{\sigma,n} = \mathrm{span}(\zeta_m : m \in \mathbb{Z}^\varepsilon, \,|m|<n).$$
    If $\lambda \in \mathbb{K}^\times$ is not of the form $\pm q^n$ for some $n \in \mathbb{Z}^{-\varepsilon}$, then $\mathrm{Ind}_q(\varepsilon, \lambda)$ is simple.
\end{theorem}

\begin{theorem}\label{intertwiners}
    Given $\lambda_1,\lambda_2 \in \mathbb{K}^\times$ and $\varepsilon_1,\varepsilon_2 \in \{-1,1\}$, if there exists a non-zero $(\mathfrak{g},K)_q$-intertwiner $\mathrm{Ind}_q(\varepsilon_1,\lambda_1) \to \mathrm{Ind}_q(\varepsilon_2,\lambda_2)$, then  $(\varepsilon_1,\lambda_1) = (\varepsilon_2, \lambda_2)$ or $(\varepsilon_1, \lambda_1) = (\varepsilon_2, \lambda_2^{-1})$.
    
    If $\varepsilon \in \{-1,1\}$ and $\lambda \in\mathbb{K}^\times$ is not of the form $\pm q^n$ for $n \in\mathbb{Z}^{-\varepsilon}$, then there exists up to a scalar a unique non-zero $(\mathfrak{g},K)_q$-intertwining map $\mathrm{Ind}_q(\varepsilon,\lambda^{-1}) \to \mathrm{Ind}_q(\varepsilon,\lambda)$ and it is an isomorphism. Moreover, for each $n\in \mathbb{Z}^{-\varepsilon}_{> 0}$ and $\sigma \in\{-1,1\}$, we have exact sequences of $(\mathfrak{g},K)_q$-modules:
    \begin{gather*}0  \to Q_{\sigma, n}\to \mathrm{Ind}(\varepsilon,\sigma q^{-n})\to D_{\sigma, n}^+ \oplus D_{\sigma, n}^{-} \to 0,\\
    0 \to D_{\sigma, n}^+ \oplus D_{\sigma, n}^{-} \to \mathrm{Ind}(\varepsilon,\sigma q^n) \to Q_{\sigma, n} \to 0.
    \end{gather*}
    In the case $n =0$, one has $\mathrm{Ind}(\varepsilon,\sigma) = D_{\sigma, 0}^{+} \oplus D_{\sigma, 0}^{-}$.
\end{theorem}

\begin{proof}[Proof of Theorem \ref{submodules}]
    Let $(\varepsilon, \lambda)\in \{-1,1\}\times \mathbb{K}^\times$. Being stable by $R_q(\mathfrak{g},K)$, a $(\mathfrak{g},K)_q$-submodule of $\mathrm{Ind}_q(\varepsilon, \lambda)$ is necessarily of the form $\mathrm{span}(\zeta_n : n \in \Lambda)$ for $\Lambda$ a subset of $\mathbb{Z}^\varepsilon$. Such a subspace is stable by $\mathcal{O}_q(K\backslash U)$ if and only if for every $n \in \Lambda$, we have $(T_n^+)_{\varepsilon,\lambda}\zeta_n = 0$ whenever $n+2 \notin \Lambda$ and $(T_n^-)_{\varepsilon,\lambda}\zeta_n = 0$ whenever $n-2 \notin \Lambda$. This follows from Proposition \ref{actiontransition} and the fact that $\mathrm{span}(X,Y,Z) = \mathrm{span}(T_n, T_n^+, T_n^-)$ for all integer $n$. Exploiting the formulas of Proposition \ref{actiontransition}, we get for every $n \in \mathbb{Z}^{-\varepsilon}$:
    $$(T_{-n-1}^+)_{\varepsilon,\lambda}\zeta_{-n-1} = 0 \Longleftrightarrow (T_{n+1}^-)_{\varepsilon,\lambda}\zeta_{n+1} = 0 \Longleftrightarrow \lambda \in \{ -q^n, q^{n} \}.$$
    The combination of these equivalences and the above characterization of the submodules of $\mathrm{Ind}_q(\varepsilon, \lambda)$ implies the various statements of the theorem.
\end{proof}

\begin{proof}[Proof of Theorem \ref{intertwiners}]
    Assume that there exists a non-zero $(\mathfrak{g},K)_q$-intertwining map $\mathrm{Ind}_q(\varepsilon_1,\lambda_1) \to \mathrm{Ind}_q(\varepsilon_2,\lambda_2)$. Then $\mathrm{Ind}_q(\varepsilon_1,\lambda_1)$ and $\mathrm{Ind}_q(\varepsilon_2,\lambda_2)$ have at least one $K_q$-type in common, so $\varepsilon_1 = \varepsilon_2$. Moreover, the action of $\Omega$ on both $(\mathfrak{g},K)_q$-modules must be the same, hence $\lambda_1 = \lambda_2$ or $\lambda_1 = \lambda_2^{-1}$.

    Let $(\varepsilon, \lambda)\in \{-1,1\}\times \mathbb{K}^\times$ and let $f : \mathrm{Ind}_q(\varepsilon,\lambda^{-1}) \to \mathrm{Ind}_q(\varepsilon,\lambda)$ be a $\mathbb{K}$-linear map. Imposing that $f$ intertwines the action of $R_q(K)$ amounts to assuming that there exists $f_n \in \mathbb{K}$ such that $f(\zeta_n) = f_n \zeta_n$ for all $n \in \mathbb{Z}^\varepsilon$. In that case, $f$ intertwines the action of $\mathcal{O}_q(K\backslash U)$ if and only if for each $\mu = +,-,\emptyset$ and $n \in \mathbb{Z}^\varepsilon$ we have $f((T_n^\mu)_{\varepsilon,\lambda^{-1}}\zeta_n) = (T_n^\mu)_{\varepsilon,\lambda} f(\zeta_n)$. The last condition is equivalent to
    \begin{equation}\label{recursionf} \forall n \in \mathbb{Z}^{-\varepsilon}, \qquad f_{n+1} (\lambda^{-1} q^{n} - \lambda q^{-n}) =  (\lambda q^{n} - \lambda^{-1} q^{-n}) f_{n-1},\end{equation}
    see Proposition \ref{actiontransition}.
    Assume first that $\lambda$ is not of the form $\pm q^{n}$ for $n \in \mathbb{Z}^{-\varepsilon}$. Because of the above recurrence relation, the map which associates to a $(\mathfrak{g},K)_q$-morphism $f : \mathrm{Ind}_q(\varepsilon,\lambda^{-1}) \to \mathrm{Ind}_q(\varepsilon,\lambda)$ the value of $f_m$ for a fixed $m \in \mathbb{Z}^\varepsilon$ is a linear isomorphism. Moreover if $f_m \neq 0$ then $f_n \neq 0$ for all $n \in \mathbb{Z}^\varepsilon$ and $f$ is invertible.

    Now assume that $\lambda = \pm q^n$ for a certain $n \in \mathbb{Z}^{-\varepsilon}\setminus\{0\}$ and let $f$ be a $(\mathfrak{g},K)_q$-intertwining map $\mathrm{Ind}_q(\varepsilon,\lambda^{-1}) \to \mathrm{Ind}_q(\varepsilon,\lambda)$. Evaluating (\ref{recursionf}) at $n$ and $-n$ yields $f_{n-1} = f_{1-n} = 0$. From the same recurrence relations, it follows that $f_{m} = 0$ when $|m| \leq n-1$ if $n>0$ and when $|m| \geq 1 + |n|$ if $n <0$. In the former case, if $f_{n+1}$ and $f_{n-1}$ are chosen non-zero then $f_m \neq 0$ for all $|m|\geq n+1$. In the second case, if $f_{|n|-1}$ is chosen non-zero then $f_m \neq 0 $ for all $|m| \leq |n|-1$. These two cases correspond to the two exact sequences of the theorem.
\end{proof}

\section{Classification of the irreducible representations}

In this section, we restrict to the case $\mathbb{K} = \mathbb{C}$. The assumption that $q$ is not a root of unity and admits self-adjoint square roots amounts to $q \in \mathbb{R}_+^*\setminus \{1\}$. Our goal is to classify the simple $(\mathfrak{g},K)_q$-modules in terms of parabolic induction.

Recall from Theorem \ref{HCiso} the definition of the isomorphism $\tilde \gamma$ from the center $\mathcal{Z}_q(\mathfrak{g})$ of $U_q(\mathfrak{g})$ to the algebra of Laurent polynomials that are invariant by $t \mapsto t^{-1}$. For every $\lambda \in \mathbb{C}^\times$, let us denote by $\mathrm{ev}_\lambda$ the character of $\mathcal{Z}_q(\mathfrak{g})$ corresponding to the composition of $\tilde \gamma$ by the evaluation at $t = \lambda$. It is clear that two such characters $\mathrm{ev}_\lambda$ and $\mathrm{ev}_{\lambda'}$ are equal if and only if $\lambda$ and $\lambda'$ define the same element of $\mathbb{C}^\times /\mathbb{Z}_2$, where the action of $\mathbb{Z}_2$ is given by the inversion map. Moreover, any character of $\mathcal{Z}_q(\mathfrak{g})$ is of this form.
By analogy to the classical terminology, a $(\mathfrak{g},K)_q$ is said to have infinitesimal character $\lambda \in \mathbb{C}^\times /\mathbb{Z}_2$ if $\mathcal{Z}_q(\mathfrak{g})$ acts on it via the character $\mathrm{ev}_\lambda$.

Let us now explain the construction of the universal basic modules, already introduced in \cite{DCDTsl2R}. Let $\lambda \in \mathbb{C}^\times /\mathbb{Z}_2$ and $n \in \mathbb{Z}$. Let us denote by $I_{\lambda,n}$ the left ideal of $U_q(\mathfrak{g})$ generated by the kernel of $\mathrm{ev}_\lambda$ and the kernel of $\zeta_n$, the latter considered as a character of $U_q(\mathfrak{k})$. More concretely, $I_{\lambda,n}$ is the left ideal of $U_q(\mathfrak{k})$ generated by $\Omega - (\lambda + \lambda^{-1})$ and $\theta - [n]_q$. Let us denote by $M_{\lambda,n}$ the quotient of $U_q(\mathfrak{g})$ by $I_{\lambda, n}$. Let $m_{\lambda,n}$ be the image of $1$ in $M_{\lambda,n}$. We have $\Omega m_{\lambda,n} = \mathrm{ev}_\lambda(\Omega)m_{\lambda,n}$ and $\theta m_{\lambda,n} = [n]_q m_{\lambda,n}$. From Remark \ref{noneed}, it follows that $M_{\lambda,n}$ is a basic $(\mathfrak{g},K)_q$-module. Among all basic modules, $M_{\lambda,n}$ is universal in the following sense: if $V$ is any $(\mathfrak{g},K)_q$-module, $v \in V$ is such that $\Omega v = (\lambda + \lambda^{-1})v$ and $\theta v = [n]_q v$, then there exists a unique $(\mathfrak{g},K)_q$-intertwiner $M_{\lambda,n} \to V$ mapping $m_{\lambda,n}$ to $v$.

Let us refer to the modules of the form $M_{\lambda,n}$ for $\lambda \in \mathbb{C}^\times /\mathbb{Z}_2$ and $n \in \mathbb{Z}$ as the universal basic $(\mathfrak{g},K)_q$-modules. Recall that every simple $(\mathfrak{g},K)_q$-module is basic. Because of the above universal property, any simple $(\mathfrak{g},K)_q$-module is a quotient of some universal basic module.

\begin{proposition}
    The universal basic $(\mathfrak{g},K)_q$-modules are non-zero. Moreover, each of them has a greatest proper submodule and hence a unique simple quotient.
\end{proposition}
    The image of any $\lambda \in \mathbb{C}$ in $\mathbb{C}^\times /\mathbb{Z}_2$ is denoted by $[\lambda]$.
\begin{proof}
    Let $\lambda \in \mathbb{C}^\times$ and $n \in \mathbb{Z}$. Let $\varepsilon \in \{-1,1\}$ such that $n \in \mathbb{Z}^\varepsilon$. From Proposition \ref{actiontransition}, we know that $\mathrm{Ind}_q(\varepsilon,\lambda)$ has infinitesimal character $[\lambda]$. Let $f$ be the unique $(\mathfrak{g},K)_q$-intertwiner $M_{[\lambda], n} \to \mathrm{Ind}_q(\varepsilon, \lambda)$ such that $f(m_{[\lambda],n}) = \zeta_n$. Since $\zeta_n \neq 0$, we have $M_{[\lambda],n} \neq 0$.

    Let $A \subset \mathbb{Z}$ be the set of $K_q$-types of $M_{[\lambda], n}$. For each $a \in A$ let us denote by $M_{[\lambda], n}(a)$ the corresponding $K_q$-isotypical component. Since $M_{[\lambda], n}$ is basic, the latter are one-dimensional. Any proper submodule of $M_{[\lambda], n}$ is thus of the form $\oplus_{b \in B}M_{[\lambda], n}(b)$ for some subset $B$ of $A$ not containing $n$. Then, the sum of all proper submodules of $M_{[\lambda], n}$ is proper, which concludes the proof.
\end{proof}

\begin{corollary}
    Every simple $(\mathfrak{g},K)_q$-module is isomorphic to a subquotient of a parabolically induced module $\mathrm{Ind}_q(\varepsilon,\lambda)$ for some $(\varepsilon,\lambda) \in \{-1,1\}\times \mathbb{C}^\times$.
\end{corollary}

\begin{proof}
    As evoked earlier, every simple $(\mathfrak{g},K)_q$-module can be realized as the simple quotient of some universal basic module. Let $\lambda \in \mathbb{C}^\times$ and $n \in \mathbb{Z}$. Let $\varepsilon \in \{-1,1\}$ such that $n \in \mathbb{Z}^\varepsilon$ and $f : M_{[\lambda], n} \to \mathrm{Ind}_q(\varepsilon, \lambda)$ the unique intertwiner such that $f(m_{[\lambda],n}) = \zeta_n$. The simple quotient of $M_{[\lambda], n}$ is a quotient of the image of $f$, hence a subquotient of $\mathrm{Ind}_q(\varepsilon, \lambda)$.
\end{proof}

We can now state the classification theorem.

\begin{theorem}
    Every simple $(\mathfrak{g},K)_q$-module is equivalent to one and only one of the following simple modules:
    \begin{itemize}
        \item $\mathrm{Ind}_q(\varepsilon,\lambda)$ for $\varepsilon \in \{-1,1\}$ and $\lambda \in \mathbb{C}^\times \setminus \{ \pm q^n : n \in \mathbb{Z^\varepsilon}\}$ such that $|\lambda| <1$,
        \item $\mathrm{Ind}_q(1,\lambda)$ for  $\lambda \in \mathbb{C}^\times$ such that $|\lambda | = 1$ and $\Im\lambda \geq 0$,
        \item $\mathrm{Ind}_q(-1,\lambda)$ for  $\lambda \in \mathbb{C}^\times$ such that $|\lambda | = 1$ and $\Im\lambda > 0$,
        \item $D^+_{\sigma,n}$ and $D^-_{\sigma,n}$ for $n \in \mathbb{Z}_{\geq 0}$, $\sigma \in \{-1,1\}$,
        \item $Q_{\sigma, n}$ for $n \in \mathbb{Z}_{> 0}$.
    \end{itemize}
\end{theorem}

\begin{proof}
    Because of the previous result, the proof of the theorem reduces to identifying the equivalence classes of simple subquotients of the $(\mathfrak{g},K)_q$-modules induced from characters (the analogues of the non-unitary principal series). For this, we use Theorems \ref{submodules} and \ref{intertwiners}. If $\varepsilon \in \{-1,1\}$, $n \in \mathbb{Z}^{-\varepsilon}$ and $\sigma \in \{-1,1\}$ then $\mathrm{Ind}_q(\varepsilon, \sigma q^n)$ and $\mathrm{Ind}_q(\varepsilon, \sigma q^{-n})$ both admit Jordan-Holder composition series with the same successive inequivalent quotients: $D_{\sigma,n}^+$, $D_{\sigma,n}^-$ and $Q_{\sigma, n}$. In all the other cases, $\mathrm{Ind}_q(\varepsilon,\lambda)$ is simple and equivalent to $\mathrm{Ind}_q(\varepsilon,\lambda^{-1})$. Finally, if $\mathrm{Ind}_q(\varepsilon_1,\lambda_1)$ and $\mathrm{Ind}_q(\varepsilon_2,\lambda_2)$ have a simple subquotient in common, then they must have at least one $K_q$-type in common, so $\varepsilon_1 = \varepsilon_2$, and the center must act with the same character on both modules, hence $[\lambda_1] = [\lambda_2]$.
\end{proof}

As a consequence of this classification, a simple $(\mathfrak{g},K)_q$-module is uniquely characterized (up to equivalence) by its infinitesimal character and one of its $K_q$-types, as in the classical case. Moreover, any simple $(\mathfrak{g},K)_q$-module can be embedded into a parabolically induced module $\mathrm{Ind}_q(\varepsilon,\lambda)$ for some $(\varepsilon, \lambda) \in \{-1,1\}\times \mathbb{C}^\times$. This is the analogue of the subrepresentation theorem \cite{Casselman}*{Theorem 8.21} applied to $\mathrm{SL}(2,\mathbb{R})$.

Let us also comment on the unitary aspect. By \cite{DCDTsl2R}*{Theorem 3.17}, we see that the simple unitarizable $(\mathfrak{g},K)_q$-modules are, up to equivalence,
\begin{itemize}
    \item $\mathrm{Ind}_q(1,\lambda)$ for  $\lambda \in \mathbb{C}^\times$ such that $|\lambda | = 1$ and $\Im\lambda \geq 0$,
    \item $\mathrm{Ind}_q(-1,\lambda)$ for  $\lambda \in \mathbb{C}^\times$ such that $|\lambda | = 1$ and $\Im\lambda > 0$,
    \item $D^+_{\sigma,n}$ and $D^-_{\sigma,n}$ for $n \in \mathbb{Z}_{\geq 0}$, $\sigma \in \{-1,1\}$,
    \item $Q_{1,1}$ and $Q_{-1,1}$,
    \item $\mathrm{Ind}_q(1,\lambda)$ for $\lambda \in (-1,1)$ such that $|\lambda| > \min (q,q^{-1})$.
\end{itemize}
In the order of the list, these modules are analogues of the even and odd unitary principal series, the discrete series and their limits, the trivial representation, the complementary series. Note that the three last elements of the above list are doubled with respect to their classical counterparts. By Proposition \ref{unitarydual} and Proposition \ref{ilcuequivalence}, the above list also describes the spectrum of $C^*_q(G)$. In \cite{DCDTsl2R}, the deformation parameter $q$ is taken in $(0,1)$ but the proofs extend to the case $q \in \mathbb{R}_+^*\setminus \{1\}$. Note also that the unitarizability of the analogues of the unitary principal series automatically follows from Corollary \ref{unitaritypreserved}.

\begin{remark} Assume that $q \in (0,1)$. The authors of \cite{DCDTinvariant} introduced an analogue of the regular representation for $C^*_q(G)$. Let us say that a Hilbert representation of $C^*_q(G)$ is tempered if it is weakly contained in this regular representation. Then \cite{DCinduction}*{Theorem 5.13} essentially states that the irreducible tempered representations are up to equivalence the Hilbert completions of
\begin{itemize}
    \item $\mathrm{Ind}_q(1,\lambda)$ for  $\lambda \in \mathbb{C}^\times$ such that $|\lambda | = 1$ and $\Im\lambda \geq 0$,
    \item $\mathrm{Ind}_q(-1,\lambda)$ for  $\lambda \in \mathbb{C}^\times$ such that $|\lambda | = 1$ and $\Im\lambda > 0$,
    \item $D^+_{\sigma,n}$ and $D^-_{\sigma,n}$ for $n \in \mathbb{Z}_{\geq 0}$, $\sigma \in \{-1,1\}$.
\end{itemize}
Now observe that the rest of the simple $(\mathfrak{g},K)_q$-modules are given by the collection of the unique simple quotients of $\mathrm{Ind}_q(\varepsilon, \lambda)$ for $(\varepsilon, \lambda) \in \{-1,1\}\times \mathbb{C}^\times$ such that $|\lambda| <1$, which are mutually inequivalent. One can compare this to the Langlands classification for $\mathrm{SL}(2,\mathbb{R})$, see for example \citelist{\cite{Wallach}*{Section 5.6} \cite{vandenBan}*{§5}}. Here, the condition $|\lambda|<1$ is an exponentiation (in base $q$) of the condition $\Re\lambda>0$.
\end{remark}

\section{Convergence to the classical non-unitary dual}

As one may already have noticed, the classification of simple $(\mathfrak{g},K)_q$-modules of the previous section is very similar to the one of the non-unitary (or admissible) dual of $\mathrm{SL}(2,\mathbb{R})$. Our goal now is to make this observation more formal. To be more precise, in this last section, we explicitly show how the $(\mathfrak{g},K)_q$-modules parabolically induced from characters of $R_q(\mathfrak{p},M)$ deform into the classical parabolically induced $(\mathfrak{g},K)$-modules as $q \to 1$, and how their composition series are preserved in this limit. For that purpose, we use all the material introduced in Section 4.

\subsection*{The non-unitary dual of $\mathrm{SL}(2,\mathbb{R})$}
Since we will need it, we briefly recall the classification of simple $(\mathfrak{g},K)$-modules. On this topic, one may consult \cite{Wallach}*{§5.6} or \cite{Vogan}*{Chapter 1, §2-3}.

Let $P = MAN$ be the standard minimal parabolic subgroup of $G$. We have $M = \{-1,1\}$ and if $E,F,H$ are the canonical generators of $\mathfrak{g}$ then $\mathfrak{p} = \mathrm{span}(H,E)$. Let $R(\mathfrak{g},K)$ and $R(\mathfrak{p},M)$ be the Hecke algebras associated to the pairs $(\mathfrak{g},K)$ and $(\mathfrak{p},M)$. We refer to Example \ref{exampleR} and the references therein for more information concerning these algebras.

For each $(\varepsilon,\lambda)\in \{-1,1\} \times \mathbb{C}$, there exists a unique character $\chi_{\varepsilon,\lambda}$ of $R(\mathfrak{p},M)$ such that
$$\chi_{\varepsilon,\lambda}(H) = \lambda+1, \quad \chi_{\varepsilon,\lambda}(E) = 0, \quad \chi_{\varepsilon,\lambda}(-1) = \varepsilon,$$ where $-1$ is considered as an element of $M$. Conversely, every character of $R(\mathfrak{p},M)$ is of this form. For every $(\varepsilon,\lambda)\in \{-1,1\} \times \mathbb{C}$, let us denote by $\mathrm{I}(\varepsilon,\lambda)$ the following induced $(\mathfrak{g},K)$-module:
$$\mathrm{Ind}_{R(\mathfrak{p},M)}^{R(\mathfrak{g},K)}(\chi_{\varepsilon,\lambda}).$$
Let us express these induced modules in the compact picture. We denote by $\theta$ the element $i(E-F)$ of $\mathfrak{k}_\mathbb{C}$. The space $\mathcal{O}(K)$ of regular functions on $K$ has a basis $(\zeta_n)_{n \in \mathbb{Z}}$ defined by $\zeta_n(e^{it\theta}) = e^{itn} \,\, (n \in \mathbb{Z}, t \in \mathbb{R})$. The Hecke algebra $R(K)$ of the pair $(\mathfrak{k},K)$ then identifies with the space of distributions on $K$ spanned by the family $(e_n)_{n \in \mathbb{Z}}$ characterized by  $\langle e_n, \zeta_m\rangle = \delta_{n,m}$. For each $\varepsilon \in \{-1,1\}$, let us write $\mathcal{O}^\varepsilon(K) = \mathrm{span}(\zeta_n : n \in \mathbb{Z}^\varepsilon)$. For any $\lambda \in \mathbb{C}$, the restriction to $R(K)$ yields a linear isomorphism $\mathrm{res}_K : \mathrm{I}(\varepsilon,\lambda) \to \mathcal{O}^\varepsilon(K)$. The action of $\mathfrak{g}_\mathbb{C}$ on $(\zeta_n)_{n \in \mathbb{Z}^\varepsilon}$, considered as a basis of $\mathrm{I}(\varepsilon,\lambda)$ via $\mathrm{res}_K$, is given by
\begin{equation} \label{actionclassical}
    \theta \zeta_n = n \zeta_n, \qquad \kappa_\pm \zeta_n = (\lambda +1 \pm n)\zeta_{n\pm 2},
\end{equation}
where $\kappa_\pm = H \mp i(E+F)$.

Now, let us recall how to extract the non-unitary dual from the parabolically induced $(\mathfrak{g},K)$-modules. If $\varepsilon \in \{-1,1\}$ and $n \in \mathbb{Z}^{-\varepsilon}_{\geq 0}$, then $\mathrm{I}(\varepsilon,n)$ admits two distinct simple modules, that are $D^\pm_n = \mathrm{span}(\zeta_{\pm m} : m \in \mathbb{Z}^\varepsilon_{>n})$. In addition, when $n>0$, the induced $(\mathfrak{g},K)$-module $\mathrm{I}(\varepsilon,-n)$ admits a unique simple submodule, namely $Q_n = \mathrm{span}(\zeta_m : m \in \mathbb{Z}^\varepsilon, |m| <n)$, and we have exact sequences of $(\mathfrak{g},K)$-modules:
\begin{gather*}
    0 \to D^+_n \oplus D_n^- \to \mathrm{I}(\varepsilon,n)\to Q_n\to 0, \\
    0 \to Q_n \to \mathrm{I}(\varepsilon,-n)\to D^+_n \oplus D_n^- \to 0.
\end{gather*}
Finally, $\mathrm{I}(\varepsilon,\lambda)$ is simple and isomorphic to $\mathrm{I}(\varepsilon,-\lambda)$ for each pair $(\varepsilon, \lambda) \in \{-1,1\}\times \mathbb{C}$ such that $\lambda \notin \mathbb{Z}^{-\varepsilon}$. Up to equivalence, the distinct simple $(\mathfrak{g},K)$-modules are the following:
\begin{itemize}
    \item $\mathrm{I}(\varepsilon,\lambda)$ for $\varepsilon \in \{-1,1\}$ and $\lambda \in \mathbb{C}\setminus \mathbb{Z}^{-\varepsilon}$ such that $\Re\lambda \geq 0$,
    \item $D_n^+$ and $D_n^-$ for $n \in \mathbb{Z}_{\geq 0}$,
    \item $Q_n$ for $n \in \mathbb{Z}_{> 0}$.
\end{itemize}

\subsection*{The classical limit} Let $\mathbb{A}$ be the ring of analytic functions on $\mathbb{R}^*_+$ and $\mathbb{F}$ its field of fractions. In Section 4, we constructed an $\mathbb{A}$-form $U_\mathbb{A}(\mathfrak{g})$ of $U_\mathbb{F}(\mathfrak{g})$. The specialization of $U_\mathbb{A}(\mathfrak{g})$ at any $q\neq 1$ is $U_q(\mathfrak{g})$ while the specialization at $1$ is the enveloping $\ast$-algebra of $\mathfrak{g}$.

\begin{notation}
Let us denote by $\boldsymbol{q}$ the identity function on $\mathbb{R}^*_+$. The objects associated to the deformed convolution algebras of $\mathrm{SL}(2,\mathbb{R})$ relative to the field $\mathbb{F}$ and parameter $\boldsymbol{q}$ will be labeled by $\mathbb{F}$ instead of $\boldsymbol{q}$. For example, we write $\mathrm{Ind}_\mathbb{F}$, $\mathcal{O}_\mathbb{F}(K)$ for $\mathrm{Ind}_{\boldsymbol{q}}$, $\mathcal{O}_{\boldsymbol{q}}(K)$. Moreover, the elements of the deformed convolution algebras and their modules defined over $\mathbb{F}$ with deformation parameter $\boldsymbol{q}$ will be written in bold characters.

The usual notations will be reserved for the objects defined relative to the field $\mathbb{C}$ and deformation parameter $q$ in $\mathbb{R}^*_+\setminus \{1\}$. 
\end{notation}

Let $\varepsilon\in\{-1,1\}$ and $\boldsymbol{\lambda}\in \mathbb{A}^\times$. We consider the induced $(\mathfrak{g}, K)_\mathbb{F}$-module $\mathrm{Ind}_\mathbb{F}(\varepsilon,\boldsymbol{\lambda})$. It has a preferred basis $(\boldsymbol{\zeta}_n)_{n \in \mathbb{Z}^\varepsilon}$, which comes from the isomorphism
$$\mathrm{res}_{K_\mathbb{F}} : \mathrm{Ind}_\mathbb{F}(\varepsilon,\boldsymbol{\lambda}) \to \mathcal{O}_\mathbb{F}^\varepsilon(K).$$
Let us denote by $\mathrm{Ind}_\mathbb{A}(\varepsilon,\boldsymbol{\lambda})$ the $\mathbb{A}$-submodule of $\mathrm{Ind}_\mathbb{F}(\varepsilon,\boldsymbol{\lambda})$ spanned by $(\boldsymbol{\zeta}_n)_{n \in \mathbb{Z}^\varepsilon}$.

\begin{lemma}
    If $\mathrm{ev}_1(\boldsymbol{\lambda})=1$ then $U_\mathbb{A}(\mathfrak{g})$ stabilizes $\mathrm{Ind}_\mathbb{A}(\varepsilon,\boldsymbol{\lambda})$.
\end{lemma}

\begin{proof}
    Let us fix $n \in \mathbb{Z}^\varepsilon$. We define
    $$\boldsymbol{t}_n^\pm = (\boldsymbol{q}-\boldsymbol{q}^{-1})^{-1}\boldsymbol{T}_n^\pm, \qquad \boldsymbol{t}_n = (\boldsymbol{q}-\boldsymbol{q}^{-1})^{-1}(\boldsymbol{T}_n -[2]_{\boldsymbol{q}}).$$
    Using Proposition \ref{actiontransition}, we get
    \begin{gather}\label{boldt}(\boldsymbol{t}_n^\pm)_{\varepsilon,\boldsymbol{\lambda}}\boldsymbol{\zeta}_n = \frac{\boldsymbol{\lambda  q}^{1\pm n} - \boldsymbol{\lambda}^{-1}\boldsymbol{q}^{-1\mp n}}{\boldsymbol{q}-\boldsymbol{q}^{-1}}\,\boldsymbol{\zeta}_{n\pm 2}, \qquad (\boldsymbol{t}_n)_{\varepsilon,\boldsymbol{\lambda}} \boldsymbol{\zeta}_n  = \frac{\boldsymbol{\lambda} + \boldsymbol{\lambda}^{-1} -[2]_{\boldsymbol{q}}}{(\boldsymbol{q}-\boldsymbol{q}^{-1})} \boldsymbol{\zeta}_n.\end{gather}
    Then, the hypothesis $\mathrm{ev}_1(\boldsymbol{\lambda}) = 1$ ensures that $\boldsymbol{t}_n^\pm \boldsymbol{\zeta}_n \in \mathbb{A}\boldsymbol{\zeta}_{n\pm 2} $ and $\boldsymbol{t}_n \boldsymbol{\zeta}_n \in \mathbb{A}\boldsymbol{\zeta}_{n}$. Inverting the formulas (\ref{TXYZ}), one gets
    \begin{equation*} \alpha_n\begin{pmatrix}
        \boldsymbol{x} \\
        \boldsymbol{y} \\
        \boldsymbol{z}
    \end{pmatrix}
    = \begin{pmatrix}
        \boldsymbol{q}^{n+1} & \boldsymbol{q}^{1-n} & [2]_{\boldsymbol{q}}\\
        -\boldsymbol{q}^{-n-1} & -\boldsymbol{q}^{n-1} & [2]_{\boldsymbol{q}} \\
        -1 & 1 & \boldsymbol{q}^n-\boldsymbol{q}^{-n}
        \end{pmatrix}
        \begin{pmatrix}
            \{n-1\}_{\boldsymbol{q}}\boldsymbol{t}_n^+\\
            \{n+1\}_{\boldsymbol{q}} \boldsymbol{t}_n^-\\
            \{n\}_{\boldsymbol{q}}\boldsymbol{t}_n
        \end{pmatrix}+\begin{pmatrix}
            \beta_n\\
            \beta_n \\
           \gamma_n
        \end{pmatrix},\end{equation*}
    where $\boldsymbol{x}, \boldsymbol{y}, \boldsymbol{z}$ are the generators of $U_\mathbb{A}(\mathfrak{g})$  from Definition \ref{defUA} and
    $$\alpha_n = \{n-1\}_{\boldsymbol{q}}\{n\}_{\boldsymbol{q}}\{n+1\}_{\boldsymbol{q}}, \qquad \beta_n = \frac{[2]_{\boldsymbol{q}}^2\{n\}_{\boldsymbol{q}} - \alpha_n}{\boldsymbol{q}-\boldsymbol{q}^{-1}}, \qquad \gamma_n = [2]_{\boldsymbol{q}}[n]_{\boldsymbol{q}}\{n\}_{\boldsymbol{q}}.$$
    We have $\alpha_n \in \mathbb{A}^\times$ and $\beta_n,\gamma_n \in \mathbb{A}$, so $\boldsymbol{x}, \boldsymbol{y}, \boldsymbol{z} \in \mathrm{span}_\mathbb{A}(\boldsymbol{t}_n^+, \boldsymbol{t}_n^-, \boldsymbol{t}_n, 1)$. It follows that $\boldsymbol{x}_{\varepsilon,\boldsymbol{\lambda}}\boldsymbol{\zeta}_{n},\, \boldsymbol{y}_{\varepsilon,\boldsymbol{\lambda}}\boldsymbol{\zeta}_{n}$ and $ \boldsymbol{z}_{\varepsilon,\boldsymbol{\lambda}}\boldsymbol{\zeta}_{n}$ belong to $\mathrm{span}_\mathbb{A}(\boldsymbol{\zeta}_{n-2}, \boldsymbol{\zeta}_{n}, \boldsymbol{\zeta}_{n+2})$. Since we also have $\theta \boldsymbol{\zeta}_{n} = [n]_{\boldsymbol{q}} \boldsymbol{\zeta}_{n} \in \mathbb{A}\boldsymbol{\zeta}_{n}$, we conclude that $U_\mathbb{A}(\mathfrak{g})$ preserves  $\mathrm{Ind}_\mathbb{A}(\varepsilon,\boldsymbol{\lambda})$.
\end{proof}

Under the condition $\mathrm{ev}_1(\boldsymbol{\lambda})=1$, one can thus consider for each $q \in \mathbb{R}^*_+$ the specialization of $\mathrm{Ind}_\mathbb{A}(\varepsilon,\boldsymbol{\lambda})$ at $q$, that is $\mathbb{C} \otimes_{\mathrm{ev}_q} \mathrm{Ind}_\mathbb{A}(\varepsilon,\boldsymbol{\lambda})$, as a $\mathbb{C}\otimes_{\mathrm{ev}_q}U_\mathbb{A}(\mathfrak{g})$-module. We recall that the specialization of $U_\mathbb{A}(\mathfrak{g})$ at $1$ identifies with $U(\mathfrak{g})$ via (\ref{specialization1}) while its specialization at $q \neq 1$ identifies with $U_q(\mathfrak{g})$ via (\ref{changeVariables}). We have the following result.

\begin{proposition} Assume that $\mathrm{ev}_1(\boldsymbol{\lambda})=1$ and let $\lambda'$ be the first derivative of $\boldsymbol{\lambda}$ at $1$. For each $q \neq 1$, we consider the linear isomorphism
$$\begin{array}{cccccll}
    \varphi_q & : & \mathbb{C} \otimes_{\mathrm{ev}_q} \mathrm{Ind}_\mathbb{A}(\varepsilon,\boldsymbol{\lambda}) & \longrightarrow & \mathrm{Ind}_q(\varepsilon, \mathrm{ev}_q(\boldsymbol{\lambda}))\\
    & & (1\otimes\boldsymbol{\zeta}_n) & \longmapsto & \zeta_n.
\end{array}$$
We also define
$$\begin{array}{cccccll}
    \varphi_1 & : & \mathbb{C} \otimes_{\mathrm{ev}_1} \mathrm{Ind}_\mathbb{A}(\varepsilon,\boldsymbol{\lambda}) & \longrightarrow & \mathrm{I}(\varepsilon,\lambda')\\
    & & (1\otimes\boldsymbol{\zeta}_n) & \longmapsto & \zeta_n.
\end{array}$$
For each $q \neq 1$, the map $\varphi_q$ is an isomorphism of $U_q(\mathfrak{g})$-modules while $\varphi_1$ is an isomorphism of $U(\mathfrak{g})$-modules.
\end{proposition}

\begin{proof}
    Let us prove that $\varphi_q$ intertwines the action of $U_q(\mathfrak{g})$ if $q\neq 1$. For each $n \in \mathbb{Z}^\varepsilon$, we have
    \begin{equation}\label{txyz}\boldsymbol{t}_n^\pm = \boldsymbol{q}^{\pm n}{\boldsymbol{x}} - \boldsymbol{q}^{\mp n}{\boldsymbol{y}} \mp [2]_{\boldsymbol{q}}{\boldsymbol{z}} \pm [n]_{\boldsymbol{q}}, \quad \boldsymbol{t}_n = \boldsymbol{q}^{-1}{\boldsymbol{x}} + \boldsymbol{q}{\boldsymbol{y}} + (\boldsymbol{q}^n - \boldsymbol{q}^{-n}){\boldsymbol{z}},\end{equation}
    see (\ref{TXYZ}) and Definition \ref{defUA}.
    Hence we have $\boldsymbol{t}_n^\pm, \boldsymbol{t}_n \in U_\mathbb{A}(\mathfrak{g})$. 
    Let us denote by $t_n^\pm$ and $t_n$ the elements of $U_q(\mathfrak{g})$ corresponding to the specializations of $\boldsymbol{t}_n^\pm, \boldsymbol{t}_n$ at $q$. To shorten notations, let us write $\lambda_q$ for $\mathrm{ev}_q(\boldsymbol{\lambda})$. Using (\ref{boldt}), we get
    $$\varphi_q(t_n^\pm(1 \otimes \boldsymbol{\zeta}_n)) = \frac{q^{1 \pm n}\lambda_q - q^{-1\mp n}\lambda_q^{-1}}{q - q^{-1}}\zeta_{n\pm 2}, \quad \varphi_q(t_n(1 \otimes \boldsymbol{\zeta}_n)) = \frac{\lambda_q +\lambda_q^{-1} - [2]_q}{q - q^{-1}}\zeta_n.$$
    Comparing with Proposition \ref{actiontransition}, we obtain
    $$\varphi_q(T_n^\pm(1 \otimes \boldsymbol{\zeta}_n)) = (T_n^\pm)_{\varepsilon,\lambda_q} \zeta_n, \qquad \varphi_q(T_n(1 \otimes \boldsymbol{\zeta}_n)) = (T_n)_{\varepsilon,\lambda_q} \zeta_n.$$
    Since $\mathrm{span}(T_n^+,T_n^-,T_n) = \mathrm{span}(X,Y,Z)$ for all $n \in \mathbb{Z}^\varepsilon$, we deduce that $\varphi_q$ is $\mathcal{O}_q(K\backslash U)$-linear. Moreover, $\varphi_q$ maps $(1 \otimes \boldsymbol{\zeta}_n)$ to $\zeta_n$ for all $n \in \mathbb{Z}^\varepsilon$, so it commutes with the action of $\theta$. We conclude that $\varphi_q$ is $U_q(\mathfrak{g})$-linear.

    Now let us focus on the case $q = 1$. Let $t_n^\pm$ be the elements of $U(\mathfrak{g})$ corresponding to the specializations of $\boldsymbol{t}_n^\pm$ at $1$. Given the identification (\ref{specialization1}), the evaluation of (\ref{txyz}) at $q=1$ yields
    $$\kappa_\pm = t_n^\pm \pm (\theta-n),$$
    where $\kappa_\pm$ are defined by (\ref{actionclassical}). Hence, to prove that $\varphi_1$ is $U(\mathfrak{g})$-linear, it is enough to check that
    $$\varphi_1(t_n(1 \otimes \boldsymbol{\zeta}_n)) = \kappa_\pm \zeta_n \qquad (n \in \mathbb{Z}^\varepsilon).$$
    But this follows from the fact that
    $$\mathrm{ev}_1\left( \frac{\boldsymbol{\lambda  q}^{1\pm n} - \boldsymbol{\lambda}^{-1}\boldsymbol{q}^{-1\mp n}}{\boldsymbol{q}-\boldsymbol{q}^{-1}}\right) = \lambda'+1 \pm n.$$
\end{proof}

We now come to the main result announced in the introduction. It is a straightforward application of the previous proposition, Theorem \ref{submodules} and the structure of the classical parabolically induced $(\mathfrak{g},K)$-modules.

\begin{theorem}
    Let $\varepsilon \in \{-1,1\}$ and $\lambda \in \mathbb{C}$. The specialization of $\mathrm{Ind}_\mathbb{A}(\varepsilon, \boldsymbol{q}^\lambda)$ at each $q \in \mathbb{R}_+^*\setminus \{1\}$ identifies with $\mathrm{Ind}_q(\varepsilon,q^\lambda)$ through $\varphi_q$while its specialization at $1$ identifies with $\mathrm{I}(\varepsilon,\lambda)$ through $\varphi_1$.
    
    For each $(\mathfrak{g},K)_\mathbb{F}$-submodule $V$ of $\mathrm{Ind}_\mathbb{F}(\varepsilon, \boldsymbol{q}^\lambda)$ and every $q\in \mathbb{R}_+^*$, let us denote by $V_q$ the specialization of $V \cap \mathrm{Ind}_\mathbb{A}(\varepsilon, \boldsymbol{q}^\lambda)$ at $q$. Provided that $|\log q||\Im \lambda | <\pi$, the map $V \mapsto V_q$ establishes a one-one correspondence between $(\mathfrak{g},K)_\mathbb{F}$-submodules of $\mathrm{Ind}_\mathbb{F}(\varepsilon, \boldsymbol{q}^\lambda)$ and
    \begin{itemize}
        \item $(\mathfrak{g},K)_q$-submodules of $\mathrm{Ind}_q(\varepsilon,q^\lambda)$ if $q\neq 1$,
        \item $(\mathfrak{g},K)$-submodules of $\mathrm{I}(\varepsilon,\lambda)$ if $q =1$.
    \end{itemize}
\end{theorem}


\begin{thebibliography}{99}
    \bib{Arano1}{article}{
        author = {Arano, Y.},
        title = {Unitary spherical representations of Drinfeld doubles},
        journal = {Journal für die reine und angewandte Mathematik (Crelles Journal)},
        volume = {2018, no. 742},
        date = {2018},
        pages = {157--186}
    }
    \bib{Arano2}{article}{
        author = {Arano, Y.},
        title = {Comparison of Unitary Duals of Drinfeld Doubles and Complex Semisimple Lie Groups},
        journal = {Communications in Mathematical Physics},
        volume = {351},
        pages = {1137--1147},
        date = {2017}
    }
    \bib{BargmannSL2R}{article}{
        title = {Irreducible Unitary Representations of the Lorentz Group},
        author = {Bargmann, V.},
        journal = {Annals of Mathematics},
        volume ={48, no. 3},
        date = {1947},
        pages = {568--640}
    }
    \bib{Blanchard}{article}{
        title = {Déformations de $C^*$-algèbres de Hopf},
        author = {Blanchard, E.},
        journal = {Bulletin de la S.M.F.},
        volume = {124, no. 1},
        date = {1996},
        pages = {141-215}
    }
    \bib{Bourbaki}{book}{
        title = {Algebra},
        subtitle = {Chapter 8},
        author = {Bourbaki, N.},
        publisher = {Springer Cham},
        date = {2022}
    }
    \bib{BuffRoche}{article}{
        title = {Harmonic Analysis on the Quantum Lorentz Group},
        author = {Buffenoir, E.},
        author = {Roche, P.},
        journal = {Communications in Mathematical Physics},
        volume = {207},
        pages = {499--555},
        date = {1999}
    }
    
    \bib{Casselman}{article}{
        title = {Asymptotic behavior of matrix coefficients of admissible representations},
        author = {Casselman, W.},
        author = {Miličić, D.},
        journal = {Duke Mathematical Journal},
        volume = {49, no. 4},
        pages = {869--930},
        date = {1982}
    }
    \bib{ConciniKacProcesi}{article}{
        title = {Quantum coadjoint actions},
        author = {de Concini, C.},
        author = {Kac, V. G.},
        author = {Procesi, C.},
        journal = {Journal of the American Mathematical Society},
        volume = {5, no. 1},
        date ={1992},
        pages = {151--189}
    }
    \bib{GavariniQDP}{article}{
        author = {Ciccoli, N.},
        author = {Gavarini, F.},
        title = {A quantum duality principle for coisotropic subgroups and Poisson quotients},
        journal = {Advances in Mathematics},
        volume = {199, no. 1},
        date = {2006},
        pages = {104--135}
    }
    \bib{DCquantisation}{article}{
        title = {Quantisation of semisimple real Lie groups},
        author = {De Commer, K.},
        journal = {Archivum Mathematicum},
        volume = {060, no.5},
        date = {2024},
        pages = {285--310}
    }
    \bib{DCinduction}{article}{
        title = {Induction for representations of coideal doubles, with an application to quantum {$\mathrm{SL}(2,\mathbb{R})$}},
        author = {De Commer, K.},
        journal = {preprint, arXiv:2409.18551},
        date = {2024}
    }
    \bib{DCDTsl2R}{article}{
        title = {Quantum {$\mathrm{SL}(2,\mathbb{R})$} and its irreducible representations},
        author = {De Commer, K.},
        author = {Dzokou Talla, J. R.},
        journal = {Journal of Operator Theory},
        volume = {91, no. 2},
        date = {2024},
        pages = {443--470}
    }
    \bib{DCDTinvariant}{article}{
        title = {Invariant Integrals on Coideals and Their Drinfeld Doubles},
        author = {De Commer, K.},
        author = {Dzokou Talla, J. R.},
        journal = {International Mathematics Research Notices},
        volume = {2024, no. 14},
        date = {2024},
        pages = {10650–-10677}
    }
    \bib{DCFlo1}{article}{
        title = {The field of quantum $GL(N,\mathbb{C})$ in the C*-algebraic setting},
        author = {De Commer, K.},
        author = {Floré, M.},
        journal = {Selecta Mathematica New Series},
        volume = {25, no. 3},
        date = {2019}
    }
    \bib{DCFlo2}{article}{
        title = {A Field of Quantum Upper Triangular Matrices},
        author = {De Commer, K.},
        author = {Floré, M.},
        journal = {International Mathematics Research Notices},
        volume = {2017, no.16},
        date = {2017},
        pages = {5047-5077}
    }
    \bib{Drinfeld}{article}{
        title = {Quantum groups},
        author = {Drinfel'd, V. G.},
        pages = {798--820},
        book = {
            title = {Proceedings of the International Congress of Mathematicians, Vol. 1, 2 (Berkeley, Calif., 1986)},
            publisher = {American Mathematical Society},
            address = {Providence, RI},
            date = {1987}
        }
    }
    \bib{EtingofSchiffman}{book}{
        title = {Lectures on Quantum Groups},
        author = {Etingof, P.},
        author = {Schiffman, O.},
        publisher = {International Press},
        date = {1998}
    }
    \bib{YGE}{article}{
        title = {The smallest quantum Mackey analogy},
        author = {Gaudillot--Estrada, Y.},
        note = {In preparation}
    }
    \bib{HarishChandraZ}{article}{
        title = {On some applications of the universal enveloping algebra of a semisimple Lie algebra},
        author = {Harish-Chandra},
        journal = {Transations of the American Mathematical Society},
        volume = {70},
        date = {1951},
        pages = {28-96}
    }
    
    \bib{Jantzen}{book}{
        title = {Lectures on Quantum Groups},
        author = {Jantzen, J. C.},
        publisher = {American Mathematical Society},
        series = {Graduate Studies in Mathematics},
        volume = {6},
        date = {1996}
    }
    \bib{KlimykSchmudgen}{book}{
        title = {Quantum Groups and Their Representations},
        author = {Klimyk, A.},
        author = {Schmüdgen, K.},
        series = {Theoretical and Mathematical Physics},
        publisher = {Springer Berlin, Heidelberg},
        date = {1997}
        
    }
    \bib{KnappVoganInd}{book}{
        title = {Cohomological Induction and Unitary Representations},
        author = {Knapp, A.},
        author = {Vogan, D.},
        publisher = {Princeton University Press},
        series = {Princeton Mathematical Series},
        volume = {45},
        date = {1995}
    }
    \bib{qOrthogonal}{book}{
        title = {Hypergeometric Orthogonal Polynomials and Their $q$-Analogues},
        author = {Koekoek, R.},
        author = {Lesky, P.},
        author = {Swarttouw, R.},
        publisher = {Springer Berlin, Heidelberg},
        series = {Springer Monographs in Mathematics},
        date = {2010}
    }
    \bib{Kolb}{article}{
        title = {Quantum symmetric Kac–Moody pairs},
        author = {Kolb, S.},
        journal = {Advances in Mathematics},
        volume = {267},
        date = {2014},
        pages = {395--469}
    }
    \bib{Koo}{article}{
        title = {Askey–Wilson Polynomials as Zonal Spherical Functions on the $\mathrm{SU}(2)$
                Quantum Group},
        author = {Koornwinder, T.},
        journal = {SIAM Journal on Mathematical Analysis},
        volume = {24, no. 3},
        date = {1993},
        pages = {795--813}
    }
    \bib{Langlands}{article}{
        title = {On the classification of irreducible representations of real algebraic groups},
        author = {Langlands, R. P.},
        pages = {101–170},
        book = {
            title = {Representation theory and harmonic analysis on semisimple Lie groups},
            address = {Providence, RI},
            date = {1989},
            publisher = {American Mathematical Society},
            series = {Mathematical Surveys and Monographs},
            volume = {31}
        }
    }
    \bib{Letzter99}{article}{
        author = {Letzter, G.},
        title = {Symmetric Pairs for Quantized Enveloping Algebras},
        journal = {Journal of Algebra},
        volume = {220, no. 2},
        date = {1999},
        pages = {729--767}
    }
    \bib{Lusztig}{book}{
        title = {Introduction to Quantum Groups},
        author = {Lusztig, G.},
        publisher = {Birkhäuser Boston, MA},
        series = {MBC},
        date = {2010}
    }
    \bib{PWLorentz}{article}{
        title = {Quantum deformation of Lorentz group},
        author = {Podlés, P.},
        author = {Woronowicz, S. L.},
        journal = {Communications in Mathematical Physics},
        volume = {139},
        date = {1991},
        pages = {381--431}
    }
    
    \bib{PuszLorentz}{article}{
        title = {Irreducible unitary representations of quantum Lorentz group},
        author = {Pusz, W.},
        journal = {Communications in Mathematical Physics},
        volume = {152},
        pages = {591--626},
        date = {1993}
    }
    \bib{PWRep}{article}{
        title = {Representations of quantum Lorentz group on Gelfand spaces},
        author = {Pusz, W.},
        author = {Woronowicz, S. L.},
        journal = {Reviews in Mathematical Physics},
        volume = {12, no. 12},
        date = {2000},
        pages = {1551--1625}
    }
    \bib{Renard}{book}{
        title = {Représentations des groupes réductifs $p$-adiques},
        author = {Renard, D.},
        series = {Cours Spécialisés},
        volume = {17},
        publisher = {Société mathématique de France},
        date = {2010}
    }
    \bib{Song}{article}{
        title = {Quantum duality principle and quantum symmetric pairs},
        author= {Song, J.},
        journal = {preprint, arXiv:2403.05167},
        date = {2024}
    }
    
    \bib{vandenBan}{article}{
        title = {Induced representations and the Langlands classification},
        author = {van den Ban, E. P.},
        pages = {123–155},
        book = {
            title = {Representation theory and automorphic forms (Edinburgh, 1996)},
            series = {Proc. Sympos. Pure Math.},
            volume = {61},
            publisher = {American Mathematical Society},
            address = {Providence, RI},
            date = {1997}
        }
    }
    \bib{Vogan}{book}{
        title = {Representations of Real Reductive Lie Groups},
        author = {Vogan, D. Jr.},
        series = {Progress in Mathematics},
        volume = {15},
        publisher = {Birkhäuser},
        address = {Boston, MA},
        date = {1981}
    }
    \bib{VoigtYuncken}{book}{
        title = {Complex Semisimple Quantum Groups and Representation Theory},
        author = {Voigt, C.},
        author = {Yuncken, R.},
        series = {Lecture Notes in Mathematics},
        volume = {2264},
        publisher = {Springer Cham},
        date = {2020}
    }
    \bib{VYplancherel}{article}{
        title = {The Plancherel formula for complex semisimple quantum groups},
        author = {Voigt, C.},
        author = {Yuncken, R.},
        journal = {Annales de l'ENS},
        volume = {56, no. 1},
        date = {2023}
    }
    \bib{Wallach}{book}{
        title = {Real reductive groups I},
        author = {Wallach, N.},
        date = {1988},
        publisher = {Academic Press}
    }
    
    
    
    
\end{thebibliography}
\end{document}